\documentclass{amsart}
\usepackage{amssymb}
\usepackage{amsmath}
\usepackage{amsfonts}
\usepackage{graphicx}
\usepackage{epsfig,euscript,enumerate,cancel}

\newtheorem{theorem}{Theorem}[section]

\newtheorem{corollary}[theorem]{Corollary}

\newtheorem{definition}[theorem]{Definition}
\newtheorem{example}[theorem]{Example}

\newtheorem{lemma}[theorem]{Lemma}

\newtheorem{proposition}[theorem]{Proposition}
\newtheorem{remark}[theorem]{Remark}

\numberwithin{equation}{section}
\numberwithin{figure}{section}
\input{tcilatex}

\begin{document}
\title[Equivariant quantum cohomology]{Equivariant quantum cohomology and
Yang-Baxter algebras}
\author{Vassily Gorbounov}
\address[V. G.]{Institute of Mathematics, University of Aberdeen, Aberdeen
AB24 3UE, UK}
\email[V.G.]{v.gorbounov@abdn.ac.uk}
\author{Christian Korff}
\curraddr[C. K.]{School of Mathematics and Statistics, University of
Glasgow, 15 University Gardens, Glasgow G12 8QW, UK}
\email[C. K.]{christian.korff@glasgow.ac.uk}
\urladdr{http://www.maths.gla.ac.uk/~ck/}
\date{12 February, 2014}

\begin{abstract}
There are two intriguing statements regarding the quantum cohomology of
partial flag varieties. The first one relates quantum cohomology to the
affinisation of Lie algebras and the homology of the affine Grassmannian,
the second one connects it with the geometry of quiver varieties. The
connection with the affine Grassmannian was first discussed in unpublished
work of Peterson and subsequently proved by Lam and Shimozono. The second
development is based on recent works of Nekrasov, Shatashvili and of Maulik,
Okounkov relating the quantum cohomology of Nakajima varieties with
integrable systems and quantum groups. In this article we explore for the
simplest case, the Grassmannian, the relation between the two approaches. We
extend the definition of the integrable systems called vicious and
osculating walkers to the equivariant setting and show that these models
have simple expressions in a particular representation of the affine
nil-Hecke ring. We compare this representation with the one introduced by
Kostant and Kumar and later used by Peterson in his approach to Schubert
calculus. We reveal an underlying quantum group structure in terms of
Yang-Baxter algebras and relate them to Schur-Weyl duality. We also derive
new combinatorial results for equivariant Gromov-Witten invariants such as
an explicit determinant formula.
\end{abstract}

\maketitle




\section{Introduction}

Quantum cohomology was introduced in the 90'ies of the last century as a
deformation of the usual multiplication in the cohomology of a manifold and
has been in the centre of the interaction between modern mathematics and
physics ever since; see e.g. \cite%
{Bertram,Buch,BKT,Gepner,Intriligator,Vafa,Witten}. Despite a massive volume
of spectacular results obtained in the past 20 years, the theory of quantum
cohomology is far from its final form and new unexpected connections between
related mathematics and physics continue to appear.\smallskip

In this paper we will discuss the relation between two such surprising
results regarding quantum cohomology, one was discovered some time ago, the
other is relatively new.\smallskip

In unpublished but highly influential work \cite{Peterson} Peterson related
the topology of the affine Grassmannian, and the quantum cohomology of
finite-dimensional flag varieties, both associated with the same algebraic
group $G$. Many of his results have been further explored and proved by
several authors; we in particular refer to the discussions in \cite{LamShim}
and \cite{Rietsch} among others. Peterson's key tool is the affine nil-Hecke
algebra introduced earlier by Kostant and Kumar \cite{KostantKumar} and its
important commutative subalgebra, called nowadays the \emph{Peterson algebra}%
. In our discussion we shall make contact with the following result
originally put forward in \cite{Peterson}: the quantum equivariant
cohomology of the partial flag variety is a module over the affine nil-Hecke
algebra.

The second development in the theory of quantum cohomology is rather recent.
It was initiated in the work of Nekrasov, Shatashvili \cite%
{NekrasovShatashvili} and later developed mathematically by Braverman,
Maulik and Okounkov \cite{BMO,MaulikOkounkov}. These authors work with a
large class of algebraic varieties, the Nakajima varieties, of which the
cotangent space of a partial flag variety is a particular example. Their
result can be described as follows. Consider the disjoint union of the
cotangent spaces $T^{\ast }\limfunc{Gr}_{n,N}$ to the Grassmannians $%
\limfunc{Gr}_{n,N}$ of $n$-dimensional subspaces in the same ambient space
of fixed dimension $N$. It has been known for some time that the equivariant
cohomology of $T^{\ast }\limfunc{Gr}_{n,N}$ can be identified with the space
of states of an integrable system, the Heisenberg spin-chain \cite%
{GV,Ginzburg,Vasserot,Rimanyietal}. The algebra of \textquotedblleft quantum
symmetries\textquotedblright\ of this spin-chain is a quantum group, the
Yangian of $sl_{2}$, making the direct sum of the equivariant cohomologies $%
H_{T}^{\ast }(T^{\ast }\limfunc{Gr}_{n,N})$ with respect to the diemsnion $n$
a module over the Yangian. The Yangian contains a commutative subalgebra,
depending on a parameter, named \emph{Bethe algebra}, see e.g. \cite%
{Rimanyietal, Gorbetal}. It turns out that this subalgebra can be identified
with the equivariant quantum cohomology $QH_{T}^{\ast }(T^{\ast }\limfunc{Gr}%
_{n,N})$ of each individual summand; see \cite{Gorbetal} for a proof. The
nil-Hecke algebra is a part of this construction \cite{BMO}, appearing
presumably due to Schur-Weyl duality for the Yangian.\smallskip

In this article we shall instead consider the union of the quantum
cohomologies of the Grassmannians $\limfunc{Gr}_{n,N}$ themselves rather
than those of their cotangent spaces. While it is currently not known how to
endow the set of quantum cohomologies of all partial flag varieties with the
structure of a quantum group module similar to the work of Nekrasov,
Shatashvili \cite{NekrasovShatashvili} and Maulik, Okounkov \cite%
{MaulikOkounkov}, we link the two mentioned developments for the simplest
case, the Grassmannian and, thus, make first steps towards filling this gap.
Building on the earlier works \cite{KorffStroppel} and \cite{VicOsc}, we
introduce a quantum integrable lattice model whose space of states $\mathcal{%
V}=\tbigoplus_{n=0}^{N}\mathcal{V}_{n}$ is the direct sum of the quantum
equivariant cohomologies, $\mathcal{V}_{n}\cong QH_{T}^{\ast }(\limfunc{Gr}%
_{n,N})$. The quantum integrable system is formulated in terms of certain
solutions of the Yang-Baxter equation and using the latter, one can define
the so-called \emph{Yang-Baxter algebra}, which acts on $\mathcal{V}$ and
plays a role analogous to that of the Yangian in the setting of Maulik and
Okounkov. As in the case of the Yangian also the Yang-Baxter algebra
contains a commutative subalgebra, generated by the set of commuting
transfer matrices, which we explicitly describe in terms of a particular
representation of the affine nil-Hecke algebra. This representation is
different from the one considered by Kostant and Kumar \cite{KostantKumar}\
and Peterson \cite{Peterson}. We discuss the affine Hecke algebra action on
each $\mathcal{V}_{n}$ in the basis of Schubert classes and clarify its
relation with the action defined in \cite{Peterson} by changing to the basis
of idempotents in $QH_{T}^{\ast }(\limfunc{Gr}_{n,N})\otimes \mathbb{F}_{q}$%
, where $\mathbb{F}_{q}$ is the completed tensor product $\mathbb{F}_{q}=%
\mathbb{C}[q^{\pm 1/N}]\widehat{\otimes }\mathbb{F}$ with $\mathbb{F}$ being
the algebraically closed field of Puiseux series in the equivariant
parameters, $\mathbb{F}:=\mathbb{C}\{\!\{T_{1},\ldots ,T_{N}\}\!\}$. In
particular, we identify the counterparts of Peterson's basis elements in the
affine Hecke algebra in our setting and derive explicit formulae for them.

The second main result of our article is that we show that the appearance of
the Yang-Baxter algebra is a special case of Schur-Weyl duality. As
discussed in \cite{Rimanyietal} the equivariant cohomologies $H_{T}^{\ast }(%
\limfunc{Fl}_{\ell })$ of flag varieties - of which the Grassmannian is the
simplest example with $\ell =2$ - naturally allow for an action of the
current algebra $\mathfrak{gl}_{\ell }[z]=\mathfrak{gl}_{\ell }\otimes 
\mathbb{C}[z]$ in the basis of idempotents. Similar to the case of the
cotangent space discussed above, the current algebra contains a large
commutative subalgebra which has been indentified in \emph{loc. cit.} with
the Bethe algebra or integrals of motion of the Gaudin model. This $%
\mathfrak{gl}_{\ell }[z]$-module structure is in Schur-Weyl duality with the
natural action of the symmetric group $\mathbb{S}_{N}$ on the idempotents
which are labelled in terms of $\mathbb{S}_{N}$-cosets. Using our setup we
are able to explicitly describe this symmetric group action in the basis of
Schubert classes and show that it is directly connected to our integrable
model by braiding two lattice columns. In particular, we prove that the
action of the Yang-Baxter algebra commutes with this $\mathbb{S}_{N}$-action
and, thus, must be contained in the current algebra according to Schur-Weyl
duality. Our results extend to the quantum case.\smallskip 

We expect a number of consequences and generalisations from our
construction. The immediate task is to include all partial flag varieties
into the framework set out in this article, as well as to work with quantum
K-theory instead of quantum cohomology. Another task we plan to address in
future work is to describe our result as an appropriate limit of the
construction of Maulik and Okounkov and to investigate the geometric origin
of our action of the Yang-Baxter and nil-Hecke algebra.\medskip

\noindent \textbf{Acknowledgment}. The authors would like to thank the
organisers Laurent Manivel and \'{E}ric Vasserot for their kind invitation
to the ANR workshop \emph{Quantum cohomology and quantum K-theory}, Paris
Diderot, 15-17 January 2014, where the results of this article were
presented. C. K. gratefully acknowledges discussions with Gwyn Bellamy and
V. G. would like to thank the Max Planck Institute for Mathematics Bonn
where part of this work was carried out. Further thanks are due to Catharina
Stroppel for her detailed and helpful comments on a draft manuscript,
Jessica Striker for alerting us to the work \cite{BBT} and Kaisa Taipale for
making an extended abstract available.

\subsection{Overview of the main results}

We give additional details of our construction to help the reader find its
way through what are sometimes lengthy and technical calculations. Starting
point for our discussion is the following presentation of the $T$%
-equivariant quantum cohomology $QH_{T}^{\ast }(\limfunc{Gr}_{n,N})$ of the
Grassmannian due to Givental and Kim \cite{GK}. Let $\limfunc{Gr}_{n,N}$ be
the Grassmannian of subspaces of dimension $n$ in $\mathbb{C}^{N}$ and set $%
k=N-n$. The group $\limfunc{GL}(N)$ induces an action of the torus $T=(%
\mathbb{C}^{\ast })^{N}$. Set $\Lambda =\mathbb{Z}[T_{1},\ldots ,T_{N}]$,
which can be identified with the $T$-equivariant cohomology of a point $%
H_{T}^{\ast }($pt$)$, where the $T_{i}$'s are called the equivariant
parameters. Often it will be more convenient to work with the
\textquotedblleft reversed\textquotedblright\ parameters $T_{N+1-i}$ instead
and throughout this article we will use the notation $t_{i}=T_{N+1-i}$.

\begin{theorem}[Givental-Kim]
There exists an isomorphism of $\Lambda \lbrack q]$-algebras 
\begin{equation}
QH_{T}^{\ast }(\limfunc{Gr}\nolimits_{n,N})\cong \Lambda \lbrack
q][a_{1},\ldots ,a_{n},b_{1},\ldots ,b_{k}]/I_{n}\;,  \label{GK}
\end{equation}%
where $I_{n}$ is the ideal generated by the relations 
\begin{equation}
\sum_{i+j=r}a_{i}b_{j}=e_{r}(T_{1},\ldots ,T_{N})\qquad \text{and\qquad }%
a_{n}b_{k}=T_{1}\cdots T_{N}+(-1)^{n}q  \label{Idef}
\end{equation}%
with $0\leq i\leq n$, $0\leq j\leq k$, $a_{0}=b_{0}=1$, $0\leq r\leq N-1$
and $e_{r}$ is the $r^{\text{th}}$ elementary symmetric polynomial.
\end{theorem}

Note that the defining relations (\ref{Idef}) of $I$ are obtained by
expanding the polynomial identity%
\begin{equation}
\left( \tsum_{i=0}^{n}u^{i}a_{i}\right) \left(
\tsum_{j=0}^{k}u^{j}b_{j}\right) =(-1)^{n}qu^{N}+\prod_{r=1}^{N}(1-uT_{r})
\label{Ipoly}
\end{equation}%
in powers of the indeterminate $u$ and comparing coefficients on both sides.

\subsubsection{Combinatorial construction: non-intersecting lattice paths}

We will define two different types of non-intersecting lattice paths on a
cylinder of circumference $N$. The first type, so-called vicious walkers $%
\gamma $, come in $n$-tuples and the second type, so-called osculating
walkers $\gamma ^{\prime }$, come in $k$-tuples. These two types of lattice
paths are linked via level-rank duality $QH_{T}^{\ast }(\limfunc{Gr}%
_{n,N})\cong QH_{T}^{\ast }(\limfunc{Gr}_{k,N})$. Choosing a particular set
of weights $\limfunc{wt}(\gamma )\in \mathbb{Z}[x_{1},\ldots ,x_{n}]\otimes
\Lambda \lbrack q]$ - here the $x_{i}$'s denote some commuting
indeterminates called \emph{spectral parameters }- we consider the problem
of computing their weighted sums, so-called partition functions denoted by $%
\langle \lambda |Z_{n}(x|T)|\mu \rangle $. The latter depend on the start
and end points of the paths on the cylinder which are fixed in terms of two
partitions $\mu ,\lambda $ which label Schubert classes in $QH_{T}^{\ast }(%
\limfunc{Gr}_{n,N})$. We will prove the following expansion%
\begin{equation}
\langle \lambda |Z_{n}(x|T)|\mu \rangle :=\sum_{\gamma \in \Gamma _{\lambda
,\mu }}\limfunc{wt}(\gamma )=\sum_{\nu \in (n,k)}q^{d}C_{\mu \nu }^{\lambda
,d}(T)s_{\nu ^{\vee }}(x_{1},\ldots ,x_{n}|T)  \label{partition_func}
\end{equation}%
into factorial Schur functions, where $\nu ^{\vee }$ is the partition
obtained by taking the complement of the Young diagram of $\nu $ in the $%
n\times k$ bounding box. The coefficients $C_{\mu \nu }^{\lambda ,d}(T)\in
\Lambda $ are the \emph{equivariant 3-point genus zero Gromov-Witten
invariants} and our partition functions are therefore generating functions
of the latter. Specialising the spectral parameters to be equivariant
parameters, $x=T_{\nu },$ we obtain an explicit determinant formula (\ref%
{detGW}) which expresses Gromov-Witten invariants in terms of $\langle
\lambda |Z_{n}(T_{\nu }|T)|\mu \rangle $.

\subsubsection{Quantum group structures}

Let $V_{n}$ be the $\mathbb{Z}$-linear span of the partitions which label
the start or end positions of the vicious walkers on the lattice. To compute
the partition functions it turns out to be useful to define an operator $%
Z_{n}\in \mathbb{Z}[x_{1},\ldots ,x_{n}]\otimes \limfunc{End}(\mathcal{V}%
_{n})$ with $\mathcal{V}_{n}=\Lambda \lbrack q]\otimes V_{n}$ whose matrix
elements give the partition functions above. We will identify $\mathcal{V}%
_{n}$ with $QH_{T}^{\ast }(\limfunc{Gr}_{n,N})$ and the direct sum $%
\tbigoplus_{n=0}^{N}\mathcal{V}_{n}$ with $\mathcal{V}=\Lambda \lbrack
q]\otimes V^{\otimes N}$ where $V\cong \mathbb{C}^{2}$; this is simply the
parametrisation of Schubert classes in terms of binary strings or 01-words,
see e.g. \cite{KnutsonTao}. It turns out that one can write each $Z_{n}$ as
partial trace over an operator product, $Z_{n}=\limfunc{Tr}_{V^{\otimes
n}}M_{n}\cdots M_{2}M_{1}$, where $M_{i}=M(x_{i})\in \Lambda \lbrack
x_{i},q]\otimes \limfunc{End}(V\otimes \mathcal{V})$. The $M$'s are
solutions to the Yang-Baxter equation, $R_{ii^{\prime }}(x_{i}/x_{i^{\prime
}})M(x_{i})M(x_{i^{\prime }})=M(x_{i^{\prime }})M(x_{i})R_{ii^{\prime
}}(x_{i}/x_{i^{\prime }}),$ and we use them to define the Yang-Baxter
algebra $\subset \Lambda \lbrack q]\otimes \limfunc{End}\mathcal{V}$
mentioned previously. In particular, the operator coefficients for the
partitions functions of vicious and osculating walkers when $n=1,$ 
\begin{equation}
Z_{1}=H(x_{1})=\sum_{r\geq 0}H_{r}x_{1}^{r}\qquad \text{and}\qquad
Z_{1}^{\prime }=E(x_{1})=\sum_{r\geq 0}E_{r}x_{1}^{r}
\end{equation}%
generate a commutative subalgebra in the Yang-Baxter algebra. Here $%
Z_{n}^{\prime }$ denotes the partition function related to osculating
walkers.

\begin{theorem}
Let $\mathbb{P}_{n}\subset \limfunc{End}\mathcal{V}_{n}$ be the $\Lambda
\lbrack q]$-subalgebra generated by $\{E_{i}\}_{i=1}^{n}$ and $%
\{H_{j}\}_{j=1}^{k}$. Mapping $E_{i}\mapsto (-1)^{i}a_{i}$ and $H_{j}\mapsto
b_{j}$ in the notation of (\ref{GK}) provides a canonical isomorphism $%
\mathbb{P}_{n}~\widetilde{\rightarrow }~QH_{T}^{\ast }(\limfunc{Gr}_{n,N})$.
The pre-image of a Schubert class is given by%
\begin{equation}
\tilde{S}_{\lambda }=S_{\lambda }+\sum_{\mu \subset \lambda }(-1)^{|\lambda
/\mu |}\det (e_{\lambda _{i}-\mu _{j}-i+j}(T_{k+1+i-\lambda _{i}},\ldots
,T_{N}))_{1\leq i,j\leq n}~S_{\mu },  \label{Schubert}
\end{equation}%
where $S_{\lambda }=\det (E_{\lambda _{i}^{\prime }-i+j})_{1\leq i,j\leq k}$
and $\lambda ^{\prime }$ is the partition conjugate to $\lambda $.
\end{theorem}

Expanding the transfer matrices into factorial powers $(x_{1}|T)^{r}=%
\prod_{j=1}^{r}(x_{1}-T_{j})$ instead, we will obtain Mihalcea's
presentations \cite{Mihalcea} of $QH_{T}^{\ast }(\limfunc{Gr}_{n,N})$. We
will state explicit formulae relating both sets of generators in the text;
see (\ref{H2facH}), (\ref{facH2H}), (\ref{E2facE}) and (\ref{facE2E}).

\subsubsection{Bethe ansatz and idempotents}

The eigenvalue problem of the transfer matrices $E$ and $H$ can be solved
using the \emph{algebraic Bethe ansatz} or \emph{quantum inverse scattering
method} from exactly solvable lattice models. The latter leads to an
explicit algebraic construction of a common set of eigenvectors $\{Y_{w}\}$,
called \emph{Bethe vectors}, which depend on the solutions of a set of
polynomial equations, called \emph{Bethe ansatz equations}. For the
non-equivariant quantum cohomology of the Grassmannian this has been
discussed in \cite{KorffStroppel,VicOsc}. In order to solve the Bethe ansatz
equations in the equivariant case we need to extend the base field to the
previously mentioned completed tensor product $\mathbb{F}_{q}=\mathbb{C}%
[q^{\pm 1/N}]\widehat{\otimes }\mathbb{F}$ where $\mathbb{F}:=\mathbb{C}%
\{\!\{T_{1},\ldots ,T_{N}\}\!\}$. The Bethe vectors turn out to be the
idempotents of $QH_{T}^{\ast }(\limfunc{Gr}_{n,N})\otimes \mathbb{F}_{q}$
and we show that the latter algebra is semisimple, $QH_{T}^{\ast }(\limfunc{%
Gr}_{n,N})\otimes \mathbb{F}_{q}\cong \tbigoplus_{w}\mathbb{F}_{q}Y_{w}$,
where the summands are labelled by the minimal coset representatives $w$
with respect to $\mathbb{S}_{N}/\mathbb{S}_{n}\times \mathbb{S}_{k}$ and $%
\mathbb{S}_{r}$ denotes the symmetric group in $r$ letters.

Consider the braid matrices$\{\hat{r}_{j}\}_{j=1}^{N}$ which naturally arise
from the mentioned solutions $M_{i}=M(x_{i}|T_{1},\ldots ,T_{N})$ of the
Yang-Baxter equation by braiding two lattice columns in the integrable model,%
\begin{equation}
\hat{r}_{j}M(x|T_{1},\ldots ,T_{N})=M(x_{i}|T_{1},\ldots
,T_{j+1},T_{j},\ldots ,T_{N})\hat{r}_{j},  \label{braidmom}
\end{equation}%
and define operators $\boldsymbol{s}_{j}=(s_{j}\otimes 1)\hat{r}_{j}\in 
\limfunc{End}\mathcal{V}_{n}$ where $\mathcal{V}_{n}=\Lambda \lbrack
q]\otimes V_{n}$ and $s_{j}$ permutes the equivariant parameter $%
T_{j},T_{j+1}$ for $j=1,2,\ldots ,N$. Then we have the following important
results linking the Yang-Baxter algebra with Schur-Weyl duality.

\begin{theorem}
\label{SWduality}\quad \vspace{-0.5cm}\newline

\begin{itemize}
\item[(i)] The operators $\{\boldsymbol{s}_{j}\}_{j=1}^{N}$ define a level-0
action of the affine symmetric group $\mathbb{\hat{S}}_{N}$ on $\mathcal{V}%
_{n}$ with the action of $\boldsymbol{s}_{N}$ depending on $q$.

\item[(ii)] The corresponding $\mathbb{S}_{N}$-action commutes with the
Yang-Baxter algebra, that is $\boldsymbol{s}_{j}M=M\boldsymbol{s}_{j}$ for
all $j=1,2,\ldots ,N-1,$ and the $\mathbb{\hat{S}}_{N}$-action with the
transfer matrices $E$ and $H$.

\item[(iii)] The corresponding $\mathbb{S}_{N}$-action permutes the Bethe
vectors according to the natural $\mathbb{S}_{N}$-action on the cosets $[w]$%
, i.e. $\boldsymbol{s}_{j}Y_{[w]}=Y_{[s_{j}w]}$.
\end{itemize}
\end{theorem}

Expressing an element in the Schubert basis $\{v_{\lambda }\}\subset 
\mathcal{V}_{n}$ in the basis of idempotents, $v_{\lambda }=\sum_{w}\xi
_{\lambda }(w)Y_{w},$ we find that (iii) implies that the coefficients $\xi
_{\lambda }(w)$ obey the Goresky-Kottwitz-MacPherson (GKM) conditions \cite%
{GKM}. Thus, we can identify the coefficients with localised Schubert
classes $\xi _{\lambda }:\mathbb{S}_{N}/\mathbb{S}_{n}\times \mathbb{S}%
_{k}\rightarrow \mathbb{F}_{q}$. For $q=0$ a localised Schubert class takes
values in $\Lambda $, but in the quantum case $q\neq 0$ we need to use $%
\mathbb{F}_{q}$ instead. In the special case when $j$ labels a simple
reflection in $\mathbb{S}_{n}\times \mathbb{S}_{k}$ for which $[s_{j}w]=[w]$%
, (iii) becomes an identity of quantum Knizhnik-Zamolodchikov type,%
\begin{equation}
\hat{r}_{j}Y_{w}(T_{1},\ldots ,T_{N})=Y_{w}(T_{1},\ldots
,T_{j+1},T_{j},\ldots ,T_{N})\;.  \label{qKZintro}
\end{equation}
We will explain this in more detail after (\ref{GKM}) in the text.

\subsubsection{The affine nil-Hecke algebra}

Let $\mathbb{A}_{N}$ denote the affine nil-Coxeter algebra which acts on $%
\Lambda $ via divided difference operators $\partial _{j}$. Peterson used
the affine nil-Hecke ring $\mathbb{\tilde{H}}_{N}=\Lambda \rtimes \mathbb{A}%
_{N}$ of Kostant and Kumar \cite{KostantKumar} to show that there exists a
special basis $\{\jmath _{w}\}_{w}$ indexed by Grassmannian affine
permutations $w\in \mathbb{\hat{S}}_{N}$ of the centraliser subalgebra $%
\mathbb{P}_{N}=\mathcal{Z}_{\mathbb{\tilde{H}}_{N}}(\Lambda )$ which can be
identified with the Schubert classes in $H_{T}(\limfunc{Gr}_{SL_{N}})$.
Furthermore, Peterson states an explicit formula for the action of the
nil-Coxeter algebra $\mathbb{A}_{N}$ on localised Schubert classes $\{\xi
_{w}:\,w\in \mathbb{\hat{S}}_{N}\}$; see also \cite{Lam08} \cite[Eqn (5),
Sec. 6.2]{LamShim}.

An explicit description of the basis elements $\{\jmath _{w}\}_{w}$ is not
known in general, even for type $A$. Instead the basis is characterised
implicitly in terms of two properties, one of which states that the basis
elements act on localised Schubert classes via $\jmath _{w}\xi _{v}=\xi
_{w}\cdot \xi _{v}$ where the right hand side is defined by pointwise
multiplication; see the comments after (\ref{Pet_basis}) in the text for
further details.

In our description of $QH_{T}^{\ast }(\limfunc{Gr}_{n,N})$ we recover the
action of Peterson's basis via the operators (\ref{Schubert}) when working
in the basis of idempotents or Bethe vectors. Starting point is Peterson's
action of the affine\emph{\ }nil-Coxeter algebra on $QH_{T}^{\ast }(\limfunc{%
Gr}_{n,N})$ which in our setting is introduced by defining generalised
difference operators in terms of the braid matrix $\hat{r}_{j}$ appearing in
(\ref{qKZ}). Using the explicit form of the solutions of the Yang-Baxter
equation we then define an action of the (non-extended) affine Hecke algebra 
$\mathbb{H}_{N}\subset \mathbb{\tilde{H}}_{N}$ $\rho :\mathbb{H}%
_{N}\rightarrow \mathbb{Z}[T_{1}^{\pm 1},\ldots ,T_{N}^{\pm 1},q]\otimes 
\limfunc{End}(V^{\otimes N})$ which is different from Peterson's $\mathbb{H}%
_{N}$-action as it commutes with the multiplicative action of $\Lambda $ on $%
\mathcal{V}_{n}$. Furthermore, it takes a simple combinatorial form in the
basis of Schubert classes.

Let $\{\pi _{i}\}_{i=1}^{N}$ denote the generators of $\mathbb{H}_{N}$ and
given a reduced word $w=i_{1}\cdots i_{r}\in \lbrack N]^{r}$ set $\pi
_{w}=\pi _{i_{1}}\cdots \pi _{i_{r}},\;\bar{\pi}_{i}=1-\pi _{i}$ and $%
T_{w}=T_{N+1-i_{1}}\cdots T_{N+1-i_{r}}$.

\begin{proposition}
The transfer matrices can be expressed as cyclic words in the generators of
the affine nil-Hecke algebra. That is, 
\begin{equation}
E_{i}=\sum_{|w|=i}^{\circlearrowright }T_{w}\rho _{n}(\bar{\pi}_{w})\qquad 
\text{and\qquad }H_{j}=\sum_{|w|=j}^{\circlearrowleft }T_{w}\rho _{n}(\pi
_{w})  \label{cyclicHE}
\end{equation}%
with the sums respectively running over all reduced clockwise and
anti-clockwise ordered words $w$ of length $i$ and $j$.
\end{proposition}

Despite the fact that the representation $\rho $ is defined over the Laurent
polynomials in the equivariant parameters, the product $T_{w}\rho _{n}(\pi
_{w})$ gives a well-defined map over the ring of polynomials $\Lambda $.

\begin{theorem}
The action of the $\tilde{S}_{\lambda }$'s defined in (\ref{Schubert})
coincides with the action of the (projected) basis elements of Peterson in
the basis of idempotents,%
\begin{equation}
v_{\lambda }\circledast v_{\mu }:=\tilde{S}_{\lambda }v_{\mu }=\sum_{\alpha
}\xi _{\mu }(\alpha )\tilde{S}_{\lambda }Y_{\alpha }=\sum_{\alpha }(\jmath
_{w^{\lambda }}\xi _{\mu }(\alpha ))Y_{\alpha },\quad \forall \lambda ,\mu ,
\label{facS2PetBasis}
\end{equation}%
where $v_{\lambda },v_{\mu }\in \mathcal{V}_{n}$ correspond to Schubert
classes and $w^{\lambda }$ is the minimal length representative of the coset
in $\mathbb{S}_{N}/\mathbb{S}_{n}\times \mathbb{S}_{k}$ fixed by $\lambda $.
\end{theorem}

As mentioned previously, Peterson gave an explicit combinatorial description
for the action of the nil-Coxeter algebra on a Schubert class in \cite%
{Peterson}; see also \cite{Lam08} \cite[Eqn (5), Sec. 6.2]{LamShim}\ which
in our setting corresponds to (\ref{delta}) in the text. Here we extend the
combinatorial action to the entire affine nil-Hecke algebra and, in
particular, to Peterson's basis elements via (\ref{facS2PetBasis}), which
allows us to give a direct combinatorial definition of the quantum product $%
\circledast $ in $QH_{T}^{\ast }(\limfunc{Gr}_{n,N})$ in the basis of
Schubert classes using the operators (\ref{Schubert}) instead.

\subsection{Structure of the article}

\begin{description}
\item[Section 2] We introduce the solutions of the Yang-Baxter equation and
define the Yang-Baxter algebra.

\item[Section 3] We construct representations of the symmetric group, affine
nil-Coxeter and nil-Hecke algebra and relate the latter to Peterson's action
on Schubert classes. We also derive the expressions of the transfer matrices
as cyclic words in the generators of the affine nil-Hecke algebra and relate
them to affine stable Grothendieck polynomials. This section contains the
proofs of statements (i) and (ii) in Theorem \ref{SWduality}.

\item[Section 4] We explain the combinatorial interpretation of the
Yang-Baxter algebra in terms of non-intersecting lattice paths and derive
the explicit action of the transfer matrices in terms of toric tableaux. We
expand the transfer matrices into factorial powers and show that the
resulting coefficients satisfy Mihalcea's equivariant quantum Pieri rule.

\item[Section 5] We present the algebraic Bethe ansatz and compute the
spectral decomposition of the transfer matrices. This is the central step in
proving our various results. We derive statement (iii) in Theorem \ref%
{SWduality} and show the GKM conditions for the basis transformation into
Bethe vectors. We also explain how the qKZ equations arise.

\item[Section 6] We give a purely combinatorial definition of the
equivariant quantum cohomology ring in terms of the operators (\ref{Schubert}%
). Mapping operators onto their eigenvalues we obtain the various
presentations of $QH_{T}^{\ast }(\limfunc{Gr}_{n,N})$ as coordinate ring. We
state in more detail the relation between our construction and Peterson's
basis.\smallskip 
\end{description}

The table below lists the various action of algebras and groups considered
in this article and might provide a helpful reference point for the reader. 
\begin{equation*}
\begin{tabular}{|c|c|c|c|}
\hline
module & algebra/group & generators & action \\ \hline\hline
$\Lambda =\mathbb{Z}[T_{1},\ldots ,T_{N}]$ & affine symmetric group $\mathbb{%
\hat{S}_{N}}$ & $\{s_{j}\}_{j=1}^{N}$ & (\ref{affsymgroup}) \\ \cline{2-4}
& affine nil-Coxeter algebra $\mathbb{A_{N}}$ & $\{\partial _{j}\}_{j=1}^{N}$
& (\ref{DBGG}) \\ \hline
$\mathcal{V}=\Lambda \otimes V^{\otimes N}$ & affine symmetric group $%
\mathbb{\hat{S}_{N}}$ & $\{\boldsymbol{s}_{j}\}_{j=1}^{N}$ & Prop \ref%
{prop:symmgroupaction} \\ \cline{2-4}
& affine nil-Coxeter algebra $\mathbb{A_{N}}$ & $\{\delta _{j}^{\vee
}\}_{j=1}^{N}$ & (\ref{delta}) \\ \cline{2-4}
& affine nil-Hecke algebra $\mathbb{H_{N}}$ & $\{\pi _{j}\}_{j=1}^{N}$ & (%
\ref{piaction}) \\ \cline{2-4}
& Yang-Baxter & $A,B,C,D$ & (\ref{YBA2Hecke}) \\ \hline
\end{tabular}%
\ 
\end{equation*}%
\smallskip N.B. that when setting $T_{i}=0$ we recover the combinatorial
description of the non-equivariant quantum cohomology $QH^{\ast }(\limfunc{Gr%
}_{n,N})$ of the Grassmannian; see \cite{KorffStroppel} and \cite{VicOsc}.
Under this specialisation the above action of $\mathbb{\hat{S}_{N}}$ on the
space $\mathcal{V}$ becomes trivial and the action of $\mathbb{H_{N}}$ on $%
\mathcal{V}$ reduces to the action of $\mathbb{A_{N}}$ on the tensor factor $%
V^{\otimes N}$. The action of the Yang-Baxter algebra is modified
accordingly but stays well-defined and, more importantly, its commutation
relations remain unchanged; compare with \cite[Eqns (3.9-10)]{VicOsc}. The
above table shows that the equivariant case has a much richer algebraic
structure which allows us to make the connection with the nil-Hecke ring of
Kostant and Kumar and Peterson's basis. The findings in this article are
therefore a non-trivial extension of the previous works \cite{KorffStroppel}
and \cite{VicOsc}. In particular, the connection of the Yang-Baxter algebra
via Schur-Weyl duality to the current algebra action becomes only apparent 
in the equivariant setting.

\section{Yang-Baxter Algebras}

In this section we introduce the underlying algebraic structure of the
lattice models which we will then connect to the quantum cohomology and the
affine nil Hecke ring.

\subsection{Solutions to the Yang-Baxter equation}

Let $V=\mathbb{C}v_{0}\oplus \mathbb{C}v_{1}$ and let $\sigma ^{-}=\left( 
\begin{smallmatrix}
0 & 1 \\ 
0 & 0%
\end{smallmatrix}%
\right) ,\;\sigma ^{+}=\left( 
\begin{smallmatrix}
0 & 0 \\ 
1 & 0%
\end{smallmatrix}%
\right) ,\;\sigma ^{z}=\left( 
\begin{smallmatrix}
1 & 0 \\ 
0 & -1%
\end{smallmatrix}%
\right) $ be the natural representation of $sl(2)$ in terms of Pauli
matrices. The latter act on $V\cong \mathbb{C}^{2}$ via $\sigma
^{-}v_{1}=v_{0}$, $\sigma ^{+}v_{0}=v_{1}$ and $\sigma ^{z}v_{\alpha
}=(-1)^{\alpha }v_{\alpha },~\alpha =0,1$. We introduce the abbreviations $%
V[x_{i}]:=\mathbb{Z}[x_{i}]\otimes V,$ $V[t_{j}]:=\mathbb{Z}[t_{j}]\otimes V$
etc. Define the following $L$-operators $V[x_{i}]\otimes V[t_{j}]\rightarrow
V[x_{i}]\otimes V[t_{j}]$ by setting%
\begin{equation}
L_{ij}=\left( 
\begin{array}{cc}
1_{j}-x_{i}t_{j}\sigma _{j}^{-}\sigma _{j}^{+} & x_{i}\sigma _{j}^{+} \\ 
\sigma _{j}^{-} & x_{i}\sigma _{j}^{-}\sigma _{j}^{+}%
\end{array}%
\right) _{i}  \label{L}
\end{equation}%
and%
\begin{equation}
L_{ij}^{\prime }=\left( 
\begin{array}{cc}
1_{j}+x_{i}t_{j}\sigma _{j}^{+}\sigma _{j}^{-} & x_{i}\sigma _{j}^{+} \\ 
\sigma _{j}^{-} & x_{i}\sigma _{j}^{+}\sigma _{j}^{-}%
\end{array}%
\right) _{i}  \label{L'}
\end{equation}%
where the indices indicate that we have written the operator in matrix form
with respect to the first factor in $V[x_{i}]\otimes V[t_{j}]$.

The following proposition states the key property of the $L$-operators which
underlies the solvability of the statistical lattice models which we will
discuss below.

\begin{proposition}
\label{prop:ybe}The $L,L^{\prime }$-operators satisfy Yang-Baxter equations
of the type%
\begin{equation}
R_{ii^{\prime }}L_{ij}L_{i^{\prime }j}=L_{i^{\prime }j}L_{ij}R_{ii^{\prime
}}\qquad \text{and\qquad }r_{jj^{\prime }}L_{ij}L_{ij^{\prime
}}=L_{ij^{\prime }}L_{ij}r_{jj^{\prime }}  \label{ybe}
\end{equation}%
where $R,r$ can be identified with $4\times 4$ matrices of the form%
\begin{equation}
\left( 
\begin{array}{cccc}
a & 0 & 0 & 0 \\ 
0 & b & c & 0 \\ 
0 & d & e & 0 \\ 
0 & 0 & 0 & f%
\end{array}%
\right) \;,  \label{R}
\end{equation}%
where the matrix entries are given in the following table for each of the
respective cases%
\begin{equation}
\begin{tabular}{|l||c|c|c|c|c|c|}
\hline
& $a$ & $b$ & $c$ & $d$ & $e$ & $f$ \\ \hline\hline
\multicolumn{1}{|c||}{$R_{ii^{\prime }}$} & $1$ & $0$ & $1$ & $x_{i^{\prime
}}/x_{i}$ & $1-x_{i^{\prime }}/x_{i}$ & $x_{i^{\prime }}/x_{i}$ \\ \hline
\multicolumn{1}{|c||}{$R_{ii^{\prime }}^{\prime }$} & $1$ & $%
1-x_{i}/x_{i^{\prime }}$ & $x_{i}/x_{i^{\prime }}$ & $1$ & $0$ & $%
x_{i}/x_{i^{\prime }}$ \\ \hline
\multicolumn{1}{|c||}{$r_{jj^{\prime }}=r_{jj^{\prime }}^{\prime }$} & $1$ & 
$0$ & $1$ & $1$ & $t_{j}-t_{j^{\prime }}$ & $1$ \\ \hline
\end{tabular}
\label{ybesolns}
\end{equation}
\end{proposition}

\begin{proof}
A straightforward but rather tedious and lengthy computation which we omit.
\end{proof}

\subsection{Monodromy matrices \& Yang-Baxter algebras}

It is a known fact that the solutions of the Yang-Baxter equation carry a
natural bi-algebra structure. Here we are only interested in the coproduct.
Define coproducts $\Delta ^{\func{col}}:\limfunc{End}V_{j}\rightarrow 
\limfunc{End}V_{j}\otimes \limfunc{End}V_{j+1}$ and $\Delta ^{\func{row}}:%
\limfunc{End}V_{i}\rightarrow \limfunc{End}V_{i}\otimes \limfunc{End}V_{i+1}$
by setting%
\begin{equation}
\Delta ^{\func{col}}L_{ij}=L_{ij+1}L_{ij}\qquad \text{and\qquad }\Delta ^{%
\func{row}}L_{ij}=L_{i+1j}L_{ij}\;.  \label{cops}
\end{equation}%
One easily verifies that both coproducts are well-defined, i.e. they are
algebra homomorphisms and are coassociative $(\Delta \otimes 1)\Delta
=(1\otimes \Delta )\Delta $. Alternatively, one can define the
\textquotedblleft opposite\textquotedblright\ coproducts,%
\begin{equation}
\dot{\Delta}^{\func{col}}L_{ij}=L_{ij}L_{ij+1}\qquad \text{and\qquad }\dot{%
\Delta}^{\func{row}}L_{ij}=L_{ij}L_{i+1j}\;.
\end{equation}%
We shall denote by $\Delta ^{p}=(\Delta \otimes 1\cdots \otimes 1)\cdots
(\Delta \otimes 1)\Delta $ the $p$-fold application of the various
coproducts in the first factor.

\begin{proposition}
$\Delta ^{\func{col}}L_{ij}$ and $\Delta ^{\func{row}}L_{ij}$ are solutions
to the first and second Yang-Baxter equation in (\ref{ybe}), respectively.
Moreover, the coproduct structures \textquotedblleft
commute\textquotedblright\ 
\begin{equation}
\left( \Delta ^{\func{col}}\otimes 1\right) \Delta ^{\func{row}}=\left(
\Delta ^{\func{row}}\otimes 1\right) \Delta ^{\func{col}}\;.
\end{equation}%
The same statements hold true for the opposite coproducts $\dot{\Delta}$.
\end{proposition}

The coproduct structures enables us to consider tensor products of the
original spaces $V[x_{i}],V[t_{j}]$ which we interpret as rows and columns
of a square lattice where $i$ labels the rows and $j$ the columns; see
Figure \ref{fig:eqc_moms} for an illustration.


\begin{figure}[tbp]
\begin{equation*}
\includegraphics[scale=0.5]{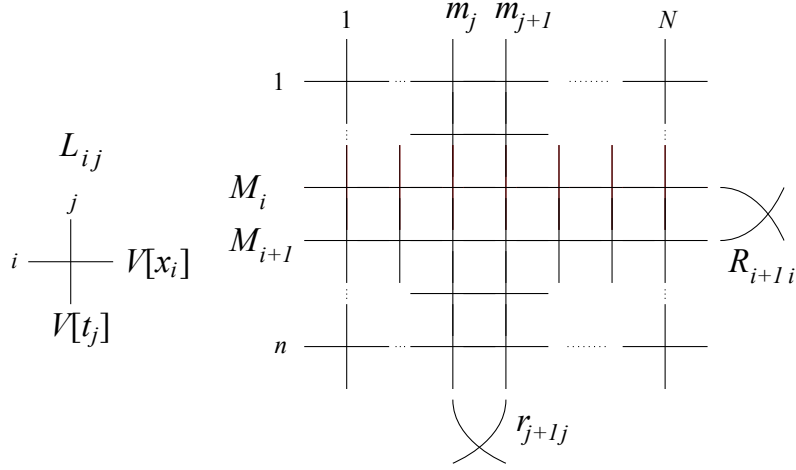}
\end{equation*}%
\caption{Graphical depiction of the $L$-operators and the monodromy
matrices. Each operator $L_{ij}$ is represented by a vertex in the $i$th row
and $j$th column. The square lattice on the right then represents the
operator (\protect\ref{rowmom}) over the tensor product $W_{n}\otimes 
\mathcal{H}$ obtained by either reading out the lattice rows right to left, $%
M_{n}\cdots M_{2}M_{1}$ or the lattice columns bottom to top, $m_{N}\cdots
m_{1}$. Both expressions are equal as $L$-operators in different rows and
columns commute. Braiding two lattice rows or two lattice columns then leads
to the matrices $R_{i+1,i}$ and $r_{j+1,j}$, respectively.}
\label{fig:eqc_moms}
\end{figure}

We consider square lattices of different dimensions for the $L$ and $%
L^{\prime }$-operators which are linked to the dimension of the ambient
space $N=n+k,$ the dimension $n$ of the hyperplanes and their co-dimension $%
k $. Namely, we set 
\begin{equation}
\mathcal{V}=\tbigotimes_{j=1}^{N}V[t_{j}]\cong \mathbb{Z}[t_{1},\ldots
,t_{N}]\otimes V^{\otimes N}  \label{quantum_space}
\end{equation}%
which is called the \emph{quantum space}, and consider the so-called \emph{%
auxiliary spaces }%
\begin{equation}
W_{r}=\tbigotimes_{i=1}^{r}V[x_{i}]\cong \mathbb{Z}[x_{1},\ldots
,x_{r}]\otimes V^{\otimes r},\quad \quad r=n,k\;.  \label{aux_space}
\end{equation}%
We associate the tensor product $W_{n}\otimes \mathcal{V}$ with an $n\times
N $ square lattice on which we will define the vicious walker model and $%
W_{k}\otimes \mathcal{V}$ with an $k\times N$ square lattice for the
osculating walker model. The partition functions of these models are
weighted sums over certain non-intersecting lattice paths and can be
expressed as matrix elements of the following operator $\mathcal{Z}%
:W_{n}\otimes \mathcal{V}\rightarrow W_{n}\otimes \mathcal{V}$ 
\begin{equation}
\mathcal{Z}=(\Delta ^{\func{row}})^{n-1}M_{1}=M_{n}\cdots M_{2}M_{1},\qquad
M_{i}=L_{iN}\cdots L_{i2}L_{i1}\;,  \label{rowmom}
\end{equation}%
where we have defined the following \emph{row-monodromy matrix} $%
M_{i}=(\Delta ^{\func{col}})^{N-1}L_{i1}:V[x_{i}]\otimes \mathcal{V}%
\rightarrow V[x_{i}]\otimes \mathcal{V}$.

To motivate the latter identify the vertex in the $i$th row and $j$th column
of the square lattice in Figure \ref{fig:eqc_moms} with $L_{ij}$, then the
operator (\ref{rowmom}) is obtained by reading out the rows of the $n\times
N $ square lattice right to left starting from the bottom. We define the
row-monodromy matrices $M_{i}^{\prime }$ for the $k\times N$ lattice in
analogous fashion using the $L^{\prime }$-operator instead.

The $L_{ij}$ operators in (\ref{rowmom}) only act non-trivially in the $i$th
row and $j$th column of the lattice, i.e. the $i$th factor in the tensor
product $W_{n}$ and the $j$th factor in $\mathcal{V}$, everywhere else they
act as the identity operator. Thus, we can trivially re-arrange the operator
(\ref{rowmom}) in terms of the \emph{column-monodromy matrices }$%
m_{j}=(\Delta ^{\func{row}})^{n-1}L_{1j}:W_{n}\otimes V[t_{j}]\rightarrow
W_{n}\otimes V[t_{j}],$ 
\begin{equation}
\mathcal{Z}=(\Delta ^{\func{col}})^{N-1}m_{1}=m_{N}\ldots m_{2}m_{1},\qquad
m_{j}=L_{nj}\cdots L_{2j}L_{1j},  \label{colmom}
\end{equation}%
by reading out the lattice columns bottom to top starting at the right.
Again, we define the column-monodromy matrices $m_{j}^{\prime }$ for the $%
L^{\prime }$-operators in an analogous manner.

\begin{corollary}
We have the following identities for respectively the row and column
monodromy matrices,%
\begin{equation}
R_{ii^{\prime }}M_{i}M_{i^{\prime }}=M_{i^{\prime }}M_{i}R_{ii^{\prime
}}\qquad \text{and\qquad }r_{jj^{\prime }}m_{j}m_{j^{\prime }}=m_{j^{\prime
}}m_{j}r_{jj^{\prime }},  \label{mom_ybe}
\end{equation}%
where $i,i^{\prime }=1,\ldots ,n$ and $j,j^{\prime }=1,\ldots ,N$. The
analogous identities hold for $M^{\prime }$ and $m^{\prime }$.
\end{corollary}

\begin{proof}
The Yang-Baxter equations for the monodromy matrices are a direct
consequence of (\ref{ybe}) using the coproduct structures (\ref{cops}).
\end{proof}

The equations in (\ref{mom_ybe}) can be seen as definitions of algebras in $%
\limfunc{End}\mathcal{V}$ and $\limfunc{End}W_{n}$, respectively. Namely,
for any $i$ we can decompose the row monodromy matrix $M=M_{i}$ defined in (%
\ref{rowmom}) over the auxiliary space $V=V[x_{i}]$ as follows,%
\begin{equation}
M=\sum_{a,b=0,1}e^{ab}\otimes M_{ab},\qquad (M_{ab})=\left( 
\begin{array}{cc}
A(x_{i}|t) & B(x_{i}|t) \\ 
C(x_{i}|t) & D(x_{i}|t)%
\end{array}%
\right) _{i}  \label{row_yba}
\end{equation}%
where $e^{ab}$ are the $2\times 2$ unit matrices and the matrix entries $%
A,B,C,D$ are elements in $\mathbb{Z}[x_{i}]\otimes \limfunc{End}\mathcal{V}$%
. Expanding the latter as in $x_{i}$ their coefficients generate the
so-called \emph{row Yang-Baxter algebra }$\subset \limfunc{End}\mathcal{V}$
with the commutation relations of $A,B,C,D$ given in terms of the matrix
elements (\ref{ybesolns}) of $R$ via (\ref{mom_ybe}). The row Yang-Baxter
algebra for the monodromy matrix $M^{\prime }$ associated with $L^{\prime }$
is defined analogously.

Similarly, we decompose the column monodromy matrix (\ref{colmom}) over the $%
j$th column $V[t_{j}]$ setting%
\begin{equation}
m_{j}=\left( 
\begin{array}{cc}
a(x|t_{j}) & b(x|t_{j}) \\ 
c(x|t_{j}) & d(x|t_{j})%
\end{array}%
\right) _{j}  \label{col_yba}
\end{equation}%
where the matrix entries $a,b,c,d$ are now elements in $\mathbb{Z}%
[t_{j}]\otimes \limfunc{End}W_{n}$. Similar to the case of the row
Yang-Baxter algebra the entries $a,b,c,d$ generate the \emph{column
Yang-Baxter algebra} $\subset \limfunc{End}W_{n}$ and their commutation
relations are fixed via the second equality in (\ref{mom_ybe}) with the
matrix elements of $r$ given in (\ref{ybesolns}).

\subsection{Quantum deformation}

We discuss a slight generalisation of the previous results which will allow
us to introduce additional \textquotedblleft quantum
parameters\textquotedblright\ $q_{1},\ldots ,q_{N}$ in the monodromy
matrices.

\begin{lemma}
We have the following simple identity for the $L$-operators%
\begin{equation}
L_{ij}(x_{i};t_{j})\left( 
\begin{smallmatrix}
1 & 0 \\ 
0 & q%
\end{smallmatrix}%
\right) _{i}=L(x_{i}q;t_{j}q^{-1})\;.  \label{qL}
\end{equation}%
The analogous statement holds true for $L_{ij}^{\prime }$.
\end{lemma}

\begin{proof}
This is immediate from the definitions (\ref{L}) and (\ref{L'}).
\end{proof}

Using the last result one proves via a similar computation as before the
following identity:

\begin{lemma}
We have the $q$-deformed Yang-Baxter equation%
\begin{equation}
r_{jj^{\prime }}(q)L_{ij}(x_{i};t_{j})\left( 
\begin{smallmatrix}
1 & 0 \\ 
0 & q%
\end{smallmatrix}%
\right) _{i}L_{ij^{\prime }}(x_{i};t_{j^{\prime }})=L_{ij^{\prime
}}(x_{i};t_{j^{\prime }})\left( 
\begin{smallmatrix}
1 & 0 \\ 
0 & q%
\end{smallmatrix}%
\right) _{i}L_{ij}(x_{i};t_{j})r_{jj^{\prime }}(q),  \label{qybe}
\end{equation}%
where%
\begin{equation}
r_{jj^{\prime }}(q)=\left( 
\begin{array}{cccc}
1 & 0 & 0 & 0 \\ 
0 & 0 & 1 & 0 \\ 
0 & 1 & q^{-1}(t_{j}-t_{j^{\prime }}) & 0 \\ 
0 & 0 & 0 & 1%
\end{array}%
\right) \;.  \label{qr}
\end{equation}
\end{lemma}

Using this simple observation we generalise our previous formulae for the
monodromy matrices by setting%
\begin{equation}
M_{i}(q_{1},\ldots ,q_{N}):=L_{iN}\left( 
\begin{smallmatrix}
1 & 0 \\ 
0 & q_{N}%
\end{smallmatrix}%
\right) _{i}\cdots L_{i2}\left( 
\begin{smallmatrix}
1 & 0 \\ 
0 & q_{2}%
\end{smallmatrix}%
\right) _{i}L_{i1}\left( 
\begin{smallmatrix}
1 & 0 \\ 
0 & q_{1}%
\end{smallmatrix}%
\right) _{i}  \label{qMom}
\end{equation}%
and%
\begin{equation}
m_{j}(q_{j}):=L_{nj}\left( 
\begin{smallmatrix}
1 & 0 \\ 
0 & q_{j}%
\end{smallmatrix}%
\right) _{n}\cdots L_{2j}\left( 
\begin{smallmatrix}
1 & 0 \\ 
0 & q_{j}%
\end{smallmatrix}%
\right) _{2}L_{1j}\left( 
\begin{smallmatrix}
1 & 0 \\ 
0 & q_{j}%
\end{smallmatrix}%
\right) _{1}\;.  \label{qmom}
\end{equation}%
Employing the same type of arguments as in our previous discussion, one
shows that these deformed monodromy matrices satisfy the same type of
Yang-Baxter relations (\ref{qybe}) as the non-deformed ones, the only
difference lies in the braid matrix $r_{jj^{\prime }}$ which is now replaced
by $r_{jj^{\prime }}(q_{j})$. Since according to (\ref{qL}) we can
re-introduce the quantum parameters easily be a simultaneous rescaling of
the equivariant and spectral parameters we shall for simplicity set $%
q_{1}=\cdots =q_{N}=1$ for now. However, below when we discuss the $q\,$%
-deformation of the cohomology ring we will choose $q_{1}=q$ and $%
q_{2}=\cdots =q_{N}=1$.

\subsection{Row-to-row transfer matrices}

We now introduce periodic boundary conditions in the horizontal direction of
the lattice by taking the partial trace of the operator (\ref{rowmom}) over
the auxiliary space $V^{\otimes n}$. We obtain the following operator $Z_{n}:%
\mathbb{Z}[x_{1},\ldots ,x_{n}]\otimes \mathcal{V}\rightarrow \mathbb{Z}%
[x_{1},\ldots ,x_{n}]\otimes \mathcal{V}$,%
\begin{equation}
Z_{n}(x|t)=\limfunc{Tr}_{V^{\otimes n}}M_{n}\cdots M_{2}M_{1}=\limfunc{Tr}%
_{V^{\otimes n}}m_{N}\cdots m_{2}m_{1}\;.  \label{Zdef}
\end{equation}%
The matrix elements of the latter are now partition functions of our lattice
models on a cylinder and we will show that these are generating functions
for Gromov-Witten invariants. We also define an operator $Z_{k}^{\prime }$
using instead the $L^{\prime }$-operators and replacing $n\rightarrow k$
everywhere.

\begin{lemma}
Denote by $H=Z_{1}=A+D$ and $E=Z_{1}^{\prime }=A^{\prime }+D^{\prime }$. We
have the relations 
\begin{equation}
Z_{n}(x|t)=H(x_{n}|t)\cdots H(x_{2}|t)H(x_{1}|t)  \label{Z2H}
\end{equation}%
and%
\begin{equation}
Z_{k}^{\prime }(x|t)=E(x_{k}|t)\cdots E(x_{2}|t)E(x_{1}|t)\;.  \label{Z'2E}
\end{equation}%
The operators $H,E$ are called the \emph{row-to-row transfer matrices}.
\end{lemma}

\begin{proof}
This is immediate from the definitions (\ref{rowmom}), (\ref{Zdef}) and the
fact that the $L$-operators $L_{ij},L_{i^{\prime }j^{\prime }}$ commute if $%
i\neq i^{\prime }$ and $j\neq j^{\prime }$.
\end{proof}

The following statement is the cornerstone for the computation of the
partition functions of our lattice models.

\begin{proposition}[Integrability]
\label{integrability} \emph{All} the row-to-row transfer matrices commute,
that is we have that 
\begin{equation}
H(x_{i})H(x_{i^{\prime }})=H(x_{i^{\prime }})H(x_{i}),\qquad
E(x_{i})E(x_{i^{\prime }})=E(x_{i^{\prime }})E(x_{i})  \label{integrability1}
\end{equation}%
as well as%
\begin{equation}
H(x_{i})E(x_{i^{\prime }})=E(x_{i^{\prime }})H(x_{i})\;.
\label{integrability2}
\end{equation}%
In particular ,the operators $Z_{n}$, $Z_{k}^{\prime }$ are symmetric in the 
$x$-variables.
\end{proposition}

\begin{proof}
The last assertion is a direct consequence of the first Yang-Baxter equation
in (\ref{mom_ybe}):%
\begin{eqnarray*}
Z_{n}(x_{1},\ldots ,x_{n}|t) &=&\limfunc{Tr}_{V^{\otimes
n}}(R_{i,i+1}M_{n}\cdots M_{1}R_{i,i+1}^{-1}) \\
&=&\limfunc{Tr}_{V^{\otimes n}}(M_{n}\cdots R_{i,i+1}M_{i}M_{i+1}\cdots
M_{1}R_{i,i+1}^{-1}) \\
&=&\limfunc{Tr}_{V^{\otimes n}}(M_{n}\cdots M_{i+1}M_{i}R_{i,i+1}\cdots
M_{1}R_{i,i+1}^{-1}) \\
&=&\limfunc{Tr}_{V^{\otimes n}}(M_{n}\cdots M_{i+1}M_{i}\cdots M_{1}) \\
&=&Z_{n}(x_{1},\ldots ,x_{i+1},x_{i},\ldots ,x_{n}|t)
\end{eqnarray*}%
The proof for $Z_{k}^{\prime }$ follows along the same lines. Setting $n=k=2$
we obtain (\ref{integrability1}).

To prove (\ref{integrability2}) one establishes the existence of additional
solutions of the Yang-Baxter equation, 
\begin{equation}
R_{ii^{\prime }}^{\prime \prime }M_{i}M_{i^{\prime }}^{\prime }=M_{i^{\prime
}}^{\prime }M_{i}R_{ii^{\prime }}^{\prime \prime }\qquad \text{and\qquad }%
R_{ii^{\prime }}^{\prime \prime \prime }M_{i}^{\prime }M_{i^{\prime
}}=M_{i^{\prime }}M_{i}^{\prime }R_{ii^{\prime }}^{\prime \prime \prime }\;,
\label{ybe2}
\end{equation}%
where $R^{\prime \prime },R^{\prime \prime \prime }$ are again of the form (%
\ref{R}) with 
\begin{equation}
\begin{tabular}{|l||c|c|c|c|c|c|}
\hline
& $a$ & $b$ & $c$ & $d$ & $e$ & $f$ \\ \hline\hline
\multicolumn{1}{|c||}{$R_{ii^{\prime }}^{\prime \prime }$} & $%
1+x_{i}/x_{i^{\prime }}$ & $1$ & $x_{i}/x_{i^{\prime }}$ & $1$ & $%
x_{i}/x_{i^{\prime }}$ & $0$ \\ \hline
$R_{ii^{\prime }}^{\prime \prime \prime }$ & $0$ & $-1$ & $1$ & $%
x_{i}/x_{i^{\prime }}$ & $-x_{i}/x_{i^{\prime }}$ & $0$ \\ \hline
\end{tabular}%
\end{equation}%
Note that $R^{\prime \prime },R^{\prime \prime \prime }$ are both singular.
However, from the combined Yang-Baxter equations (\ref{ybe2}) one derives
the commutation relations%
\begin{eqnarray}
A(x|t)A^{\prime }(y|t) &=&A^{\prime }(y|t)A(x|t)  \notag \\
A(x|t)D^{\prime }(y|t)-A^{\prime }(y|t)D(x|t) &=&D^{\prime
}(y|t)A(x|t)-D(x|t)A^{\prime }(y|t)  \notag \\
D(x|t)D^{\prime }(y|t) &=&0
\end{eqnarray}%
for the row Yang-Baxter algebras. From the latter we then easily deduce that 
$H(x|t)E(y|t)=E(y|t)H(x|t)$ and the assertion now follows.
\end{proof}

We now state a functional identity that plays an essential role in our
discussion as it directly relates the two row-to-row transfer matrices (\ref%
{H}), (\ref{E}) to the Givental-Kim presentation of $QH_{T}^{\ast }(\limfunc{%
Gr}_{n,N})$.

\begin{proposition}
The transfer matrices obey the following functional operator identity%
\begin{equation}
H(x|t)E(-x|t)=\prod_{j=1}^{N}(1-xt_{j})+qx^{N}\prod_{j=1}^{N}\sigma
_{j}^{z}\;.  \label{QQ}
\end{equation}
\end{proposition}

\begin{proof}
A direct computation along the same lines as in \cite{VicOsc}.
\end{proof}

\subsection{Level-Rank and Poincar\'{e} Duality}

We have the following relation between the two solutions $L,L^{\prime }$ of
the Yang-Baxter equations (\ref{ybe}).

\begin{lemma}
Let $L_{ij}v_{\varepsilon _{1}}\otimes v_{\varepsilon
_{2}}=\sum_{\varepsilon _{1}^{\prime },\varepsilon _{2}^{\prime
}=0,1}(L_{ij})_{\varepsilon _{1}\varepsilon _{2}}^{\varepsilon _{1}^{\prime
}\varepsilon _{2}^{\prime }}v_{\varepsilon _{1}^{\prime }}\otimes
v_{\varepsilon _{2}^{\prime }}$ and similarly define $(L_{ij}^{\prime
})_{\varepsilon _{1}\varepsilon _{2}}^{\varepsilon _{1}^{\prime }\varepsilon
_{2}^{\prime }}$. We have that%
\begin{equation}
(L_{ij}(t))_{\varepsilon _{1}\varepsilon _{2}}^{\varepsilon _{1}^{\prime
}\varepsilon _{2}^{\prime }}=(L_{ij}^{\prime }(-t))_{\varepsilon
_{1}^{\prime }1-\varepsilon _{2}}^{\varepsilon _{1}1-\varepsilon
_{2}^{\prime }}
\end{equation}%
for all $\varepsilon _{i},\varepsilon _{i}^{\prime }=0,1$ with $i=1,2$.
\end{lemma}

\begin{proof}
The matrix elements $(L_{ij})_{\varepsilon _{1}\varepsilon
_{2}}^{\varepsilon _{1}^{\prime }\varepsilon _{2}^{\prime }}$ and $%
(L_{ij}^{\prime })_{\varepsilon _{1}\varepsilon _{2}}^{\varepsilon
_{1}^{\prime }\varepsilon _{2}^{\prime }}$ can be explicitly computed from (%
\ref{L}), (\ref{L'}). They are the weights of the vertex configurations
given in Figure \ref{fig:5vmodels} in the next section where $\varepsilon
_{1},\varepsilon _{2},\varepsilon _{1}^{\prime },\varepsilon _{2}^{\prime }$
are the values of the W, N, E and S edge of the vertex. The assertion is
then easily verified from the weights displayed in Figure \ref{fig:5vmodels}.
\end{proof}

We now translate the relation between $L$ and $L^{\prime }$ into a relation
between the monodromy matrices $M$ and $M^{\prime }$ and, thus, the
associated Yang-Baxter algebras. As before we identify a basis vector $%
v_{\varepsilon _{1}}\otimes \cdots \otimes v_{\varepsilon _{N}}\in
V^{\otimes N}$ in quantum space with its 01-word $\boldsymbol{\varepsilon }%
=\varepsilon _{1}\varepsilon _{2}\cdots \varepsilon _{N}$ and introduce the
following involutions,%
\begin{gather}
\text{(reversal)\qquad }\mathcal{P}:\varepsilon _{1}\varepsilon _{2}\cdots
\varepsilon _{N}\mapsto \varepsilon _{N}\cdots \varepsilon _{2}\varepsilon
_{1}  \label{parity} \\
\text{(inversion)\qquad }\mathcal{C}:\varepsilon _{1}\varepsilon _{2}\cdots
\varepsilon _{N}\mapsto (1-\varepsilon _{1})(1-\varepsilon _{2})\cdots
(1-\varepsilon _{N})  \label{charge}
\end{gather}%
Note that by linear extension these maps become operators $\mathcal{P},%
\mathcal{C}:\mathcal{V}\rightarrow \mathcal{V}$. The following statement is
now an immediate consequence of the last lemma and the definition of the
monodromy matrices and we therefore omit its proof.

\begin{corollary}[level-rank duality]
We have the following transformation laws of the monodromy matrices $%
M,M^{\prime }$ under the joint map $\Theta =\mathcal{PC}:\mathcal{V}%
\rightarrow \mathcal{V}$,%
\begin{equation}
\Theta M(x|t)\Theta =M^{\prime }(x|-T)^{t\otimes 1}  \label{levelrankmom}
\end{equation}%
where $T_{i}=t_{N+1-i}$ and the upper index $t\otimes 1$ means transposition
in the auxiliary space. This in particular implies the following identity
for the row-to-row transfer matrices, $\Theta H(x|t)\Theta =E(x|-T)$.
\end{corollary}

Obviously $\mathcal{P}m_{j}(t_{j})\mathcal{P}=m_{N+1-j}(T_{N+1-j})$ and we
define the dual row monodromy matrices as%
\begin{equation}
M_{i}^{\vee }(T):=(\dot{\Delta}^{\func{col}})^{N-1}L_{i1}(T_{1})=\mathcal{P}%
M_{i}(t)\mathcal{P}=L_{i1}(T_{1})L_{i2}(T_{2})\cdots L_{iN}(T_{N})\;.
\label{dual_mom}
\end{equation}%
The partition functions with respect to the opposite coproduct are then
given by computing matrix element of the operator%
\begin{equation}
Z_{n}^{\vee }=\mathcal{P}Z_{n}\mathcal{P}=\limfunc{Tr}_{V^{\otimes
n}}M_{n}^{\vee }\cdots M_{1}^{\vee }=\limfunc{Tr}_{V^{\otimes
n}}m_{1}(T_{1})m_{2}(T_{2})\cdots m_{N}(T_{N})  \label{dual_Z}
\end{equation}%
and $Z_{k}^{\ast }=\mathcal{P}Z_{k}^{\prime }\mathcal{P}$ for the $L^{\prime
}$-operators. Note that since the dual monodromy matrices are defined in
terms of the opposite coproduct they also obey the Yang-Baxter relation (\ref%
{ybe}), $R_{ii^{\prime }}M_{i}^{\vee }M_{i^{\prime }}^{\vee }=M_{i^{\prime
}}^{\vee }M_{i}^{\vee }R_{ii^{\prime }}$, where the $R$-matrix elements are
given in (\ref{ybesolns}).

\begin{lemma}
The dual transfer matrices $H^{\vee }=Z_{1}^{\vee }$, $E^{\vee }=Z_{1}^{\ast
}$ are the transpose of the transfer matrices $E,H$ with respect to the
standard basis $\{v_{\varepsilon _{1}}\otimes \cdots \otimes v_{\varepsilon
_{N}}\}$ in quantum space.
\end{lemma}

\begin{remark}
Recall that given three Schubert varieties $s_{\lambda },s_{\mu },s_{\nu }$
one has in ordinary ($q=t_{j}=0,~j=1,\ldots ,N$) cohomology that%
\begin{equation*}
\int_{\limfunc{Gr}_{n,n+k}}s_{\lambda }s_{\mu }s_{\nu }=\int_{\limfunc{Gr}%
_{n,n+k}}s_{\lambda ^{\vee }}s_{\mu ^{\vee }}s_{\nu ^{\vee }}=\int_{\limfunc{%
Gr}_{k,n+k}}s_{\lambda ^{\ast }}s_{\mu ^{\ast }}s_{\nu ^{\ast }},
\end{equation*}%
where $\lambda ^{\vee }$ is the partition obtained from $\lambda $ by taking
the complement of its Young diagram in the $n\times k$ bounding box and $%
\lambda ^{\ast }=(\lambda ^{\vee })^{^{\prime }}$ is the partition obtained
by taking in addition the transpose. These first of these two identities is
known as Poincar\'{e} duality, the second as level-rank duality. Employing
the standard bijection between 01-words and partitions, the transformations $%
\lambda \mapsto \lambda ^{\vee }$ and $\lambda \mapsto \lambda ^{\ast }$ are
given by the involutions $\mathcal{P}$ and $\Theta $.
\end{remark}

\section{Representations of the affine nil Hecke algebra}

Schubert calculus of flag varieties can be formulated in terms of divided
difference operators known as Bernstein-Gelfand-Gelfand (BGG) or Demazure
operators. These extend to the equivariant setting in terms of Kostant and
Kumar's nil Hecke ring \cite{KostantKumar} and the latter plays a central
role in Peterson's \cite{Peterson} (unpublished) lecture notes which have
been highly influential in the development of the subject. In this article
we link the affine nil-Hecke ring to the quantum group structures which
underlie the quantum integrable models mentioned in the introduction,
vicious and osculating walkers. Namely the $R,r$-matrices which respectively
braid lattice rows and columns induce representations of the affine
nil-Hecke and nil-Coxeter algebra. The braiding of lattice columns turns out
to be Peterson's action of the affine nil-Coxeter algebra on Schubert
classes. To keep the article self-contained we briefly review the necessary
definitions.

\subsection{Divided difference operators}

Our quantum space (\ref{quantum_space}) consists of two tensor factors. The
polynomial factor $\Lambda $ is naturally endowed with several well-studied
actions. Consider first the ring $\tilde{\Lambda}=\mathbb{Z}[t_{1}^{\pm
1},\ldots ,t_{N}^{\pm 1}]$ of Laurent polynomials in the equivariant
parameters $t_{i}=T_{N+1-i}$ and let $\Lambda \subset \tilde{\Lambda}$ be
the subring of polynomials. The latter forms a left module for the \emph{%
extended affine symmetric group} $\mathbb{\tilde{S}}_{N} $. Namely, $\mathbb{%
\tilde{S}}_{N}$ is generated by the simple reflections $\{s_{1},\ldots
,s_{N-1}\}$ and $\varpi $ which act naturally on $\tilde{\Lambda}$ via the
formulae%
\begin{eqnarray}
s_{j}f(t_{1},\ldots ,t_{N}) &=&f(t_{1},\ldots ,t_{j+1},t_{j},\ldots
,t_{N}),\qquad j=1,\ldots ,N-1  \notag \\
\varpi f(t_{1},\ldots ,t_{N}) &=&f(t_{N}+\ell ,t_{1},t_{2},\ldots ,t_{N-1})
\label{affsymgroup}
\end{eqnarray}%
where $\ell \in \mathbb{Z}$ is called the \emph{level}. Here we choose $\ell
=0$. The subgroup $\mathbb{\hat{S}}_{N}\subset \mathbb{\tilde{S}}_{N}$
generated by $\{s_{1},\ldots ,s_{N}\}$ with $s_{N}=\varpi s_{1}\varpi ^{-1}$
is the \emph{affine symmetric group}. While we have defined this action on $%
\tilde{\Lambda}$ it obviously restricts to $\Lambda $.

\begin{definition}[\protect\cite{FominStanley}]
The affine nil-Coxeter algebra $\mathbb{A}_{N}$ is the unital, associative $%
\mathbb{Z}$-algebra generated by $\{a_{j}\}_{j=1}^{N}$ and relations 
\begin{equation*}
\left\{ 
\begin{array}{l}
a_{j}^{2}=0 \\ 
a_{j}a_{j+1}a_{j}=a_{j+1}a_{j}a_{j+1}%
\end{array}%
\right. ~,
\end{equation*}%
where all indices are understood modulo $N$. The (finite) nil-Coxeter
algebra $\mathbb{A}_{N}^{\limfunc{fin}}$ is the subalgebra generated by $%
\{a_{j}\}_{j=1}^{N-1}$.
\end{definition}

Using the level-0 action (\ref{affsymgroup}) one defines the following
representation of $\mathbb{A}_{N}^{\limfunc{fin}}$ in terms of divided
difference operators $\{\partial _{j}\}_{j=1}^{N-1}$ which are known as
Bernstein-Gelfand-Gelfand (BGG) or Demazure operators, 
\begin{eqnarray}
\partial _{j}f(t_{1},\ldots ,t_{N})
&=&(t_{j}-t_{j+1})^{-1}(1-s_{j})f(t_{1},\ldots ,t_{N})  \label{DBGG} \\
&=&\frac{f(t_{1},\ldots ,t_{N})-f(\ldots ,t_{j+1},t_{i},\ldots )}{%
t_{j}-t_{j+1}}\;.  \notag
\end{eqnarray}%
If we set in addition $\partial _{N}=\varpi ^{-1}\partial _{1}\varpi
=(t_{N}-t_{1})^{-1}(1-s_{N})$ with $\ell =0$, then this representation
extends to the affine nil-Coxeter algebra $\mathbb{A}_{N}$.

It is not obvious from their definition, but the difference operators $%
\partial _{i}$ map the subring of polynomials $\Lambda =\mathbb{Z}%
[t_{1},\ldots ,t_{N}]$ into itself and we will identify the latter with the
equivariant cohomology of a point.

\begin{proposition}[\protect\cite{FominStanley}]
The map $a_{j}\mapsto \partial _{j}$ is a representation of the affine
nil-Coxeter algebra $\mathbb{A}_{N}$.
\end{proposition}

Note that the ring $\tilde{\Lambda}$ acts on itself via multiplication. One
has the following cross-relation or \textquotedblleft Leibniz
rule\textquotedblright\ in the endomorphism ring $\limfunc{End}\tilde{\Lambda%
}$,%
\begin{equation}
\partial _{j}f=(s_{j}f)\partial _{j}+(\partial _{j}f)\;,\qquad f\in \tilde{%
\Lambda}\;.  \label{LeibnizBGG}
\end{equation}%
The semidirect or cross product $\mathbb{A}_{N}\ltimes \tilde{\Lambda}$ then
yields the affine nil-Hecke ring of Kostant and Kumar \cite{KostantKumar}.
The latter is a representation of what is often called the \emph{extended} 
\emph{affine nil-Hecke algebra} $\mathbb{\tilde{H}}_{N}=\mathbb{\tilde{H}}%
_{N}(0)$ of type $\mathfrak{gl}_{N}$ in the literature.

For our discussion we do not require the full extended affine nil-Hecke
algebra but instead only the \emph{affine nil-Hecke algebra} $\mathbb{H}_{N}=%
\mathbb{H}_{N}(0)$ which is contained as a subalgebra, $\mathbb{H}%
_{N}\subset \mathbb{\tilde{H}}_{N}$. The latter is finitely generated and
can be defined as follows:

\begin{definition}[affine nil-Hecke algebra]
The affine nil-Hecke algebra $\mathbb{H}_{N}$ is the unital, associative
algebra generated by $\{\pi _{1},\pi _{2},\ldots ,\pi _{N}\}$ subject to the
relations%
\begin{equation}
\left\{ 
\begin{array}{cc}
\pi _{i}^{2}=-\pi _{i}, &  \\ 
\pi _{i}\pi _{i+1}\pi _{i}=\pi _{i+1}\pi _{i}\pi _{i+1}, &  \\ 
\pi _{i}\pi _{j}=\pi _{j}\pi _{i}, & \quad |i-j|>1%
\end{array}%
\right. \;,  \label{nilHecke}
\end{equation}%
where all indices are understood modulo $N$. The finite nil-Hecke algebra $%
\mathbb{H}_{N}^{\limfunc{fin}}$ is the subalgebra generated by $\{\pi
_{1},\pi _{2},\ldots ,\pi _{N-1}\}$.
\end{definition}

It is convenient to also introduce the alternative set of generators $\bar{%
\pi}_{j}:=\pi _{j}+1$ which obey the same braid relations, but satisfy $\bar{%
\pi}_{j}^{2}=\bar{\pi}_{j}$\ instead. The images of the generators $\pi _{j}$
and $\bar{\pi}_{j}$ in the representation of divided difference operators
are $s_{j}-\partial _{j}t_{j}$ and $s_{j}-t_{j+1}\partial _{j}$,
respectively.\medskip

The actions of $\mathbb{\tilde{S}}_{N}$, $\mathbb{A}_{N}$ and $\mathbb{%
\tilde{H}}_{N}$ on $\Lambda $ can be extended in the obvious manner to the
quantum space (\ref{quantum_space}) by acting trivially on the second tensor
factor $V^{\otimes N}$. By abuse of notation we will use the same symbols
for the respective operators in $\limfunc{End}\mathcal{V}$.

\subsection{Representations of the nil-Coxeter algebra in quantum space}

We now show that the solutions of the Yang-Baxter equation (\ref{mom_ybe})
from the previous section also define actions of the affine nil-Hecke
algebra $\mathbb{H}_{N}$ on the quantum space $\mathcal{V}$ but in terms of
generalised divided difference operators which involve the braid matrix $r$
from (\ref{mom_ybe}).

Let $p_{jj^{\prime }}:V_{j}\otimes V_{j^{\prime }}\rightarrow V_{j}\otimes
V_{j^{\prime }}$ be the flip operator, i.e. $p(v_{\varepsilon _{1}}\otimes
v_{\varepsilon _{2}})=v_{\varepsilon _{2}}\otimes v_{\varepsilon _{1}}$ and
set $\hat{r}_{j}=p_{j+1,j}r_{j+1,j}$, $\check{r}_{j}=p_{j,j+1}r_{j,j+1}$
where $r$ is the matrix given in (\ref{ybesolns}) and all indices are
understood modulo $N$. Define the cyclic shift operator $\Omega :\mathcal{V}%
\rightarrow \mathcal{V}$ by%
\begin{equation}
\Omega :v_{\varepsilon _{1}}\otimes v_{\varepsilon _{2}}\cdots \otimes
v_{\varepsilon _{N}}\mapsto v_{\varepsilon _{N}}\otimes v_{\varepsilon
_{1}}\otimes v_{\varepsilon _{2}}\cdots \otimes v_{\varepsilon _{N-1}},
\label{Omega}
\end{equation}%
then a simple but somewhat tedious computation shows that the following
relations hold true.

\begin{lemma}
\label{prop:braidaction}The matrices $\hat{r}_{j}$ (and $\check{r}_{j}$)
obey the following identities:

\begin{itemize}
\item[(i)] $\hat{r}_{j}^{2}-2\hat{r}_{j}+1=0$

\item[(ii)] $\hat{r}_{j}(t_{j+1},t_{j+2})\hat{r}_{j+1}(t_{j},t_{j+2})\hat{r}%
_{j}(t_{j},t_{j+1})=\hat{r}_{j+1}(t_{j},t_{j+1})\hat{r}_{j}(t_{j},t_{j+2})%
\hat{r}_{j+1}(t_{j+1},t_{j+2})$

\item[(iii)] For $|i-j|\neq 1,N-1$ we have $s_{i}\hat{r}_{j}=\hat{r}_{j}s_{i}
$ and otherwise%
\begin{equation*}
s_{j}\hat{r}_{j}=\hat{r}_{j}^{-1}s_{j},\qquad s_{j\pm 1}\hat{r}_{j}=\hat{r}%
_{j}s_{j\pm 1}+\Omega ^{\mp 1}(\hat{r}_{j\pm 1}-1)\Omega ^{\pm 1}s_{j\pm 1}~.
\end{equation*}
\end{itemize}

In all these identities the indices are understood modulo $N$.
\end{lemma}

\begin{proof}
A direct computation using the explicit form (\ref{ybesolns}) of the
solution $r$. From the latter we derive the following action of $\hat{r}_{j}$
on a basis vector,%
\begin{multline}
\hat{r}_{j}v_{\varepsilon _{1}}\otimes \cdots \otimes v_{\varepsilon _{N}}=
\label{braidaction} \\
\left\{ 
\begin{array}{cc}
v_{\varepsilon _{1}}\otimes \cdots \otimes v_{\varepsilon
_{N}}-(t_{j}-t_{j+1})v_{\varepsilon _{1}}\otimes \cdots v_{\varepsilon
_{j+1}}\otimes v_{\varepsilon _{j}}\cdots \otimes v_{\varepsilon _{N}}, & 
\varepsilon _{j}<\varepsilon _{j+1} \\ 
v_{\varepsilon _{1}}\otimes \cdots \otimes v_{\varepsilon _{N}}, & \text{else%
}%
\end{array}%
\right. \;.
\end{multline}%
Making repeatedly use of the above formula one now easily verfies the
assertions.
\end{proof}

Employing the above lemma we obtain (i) in Theorem \ref{SWduality}.

\begin{proposition}
\label{prop:symmgroupaction}The operators $\boldsymbol{s}_{j}=s_{j}\hat{r}%
_{j}$ define a representation of the affine symmetric group $\mathbb{\hat{S}}%
_{N}$ on the quantum space $\mathcal{V}$.
\end{proposition}

\begin{proof}
Consider a basis vector $f\otimes v=f\otimes v_{\varepsilon _{1}}\otimes
v_{\varepsilon _{2}}\cdots \otimes v_{\varepsilon _{N}}$ in $\mathcal{V}$.
Then it follows from the preceding lemma that%
\begin{equation*}
\boldsymbol{s}_{j}^{2}=s_{j}\hat{r}_{j}s_{j}\hat{r}_{j}=\hat{r}%
_{j}^{-1}s_{j}^{2}\hat{r}_{j}=\hat{r}_{j}^{-1}\hat{r}_{j}=1_{\mathcal{V}}\;.
\end{equation*}%
To prove the braid relation verify from (\ref{braidaction}) that $\hat{r}%
_{j}=p_{j+1,j}r_{j+1,j}=1-t_{jj+1}\delta _{j}^{\vee }$ with $%
t_{ij}=t_{i}-t_{j}$ and $\delta _{j}^{\vee }=\sigma _{j}^{+}\sigma
_{j+1}^{-} $. Then 
\begin{eqnarray*}
\boldsymbol{s}_{j}\boldsymbol{s}_{j+1}\boldsymbol{s}_{j} &=&s_{j}\hat{r}%
_{j}s_{j+1}\hat{r}_{j+1}s_{j}\hat{r}_{j} \\
&=&s_{j}s_{j+1}s_{j}(1-t_{j+1j+2}\delta _{j}^{\vee })(1-t_{jj+2}\delta
_{j+1}^{\vee })(1-t_{jj+1}\delta _{j}^{\vee })
\end{eqnarray*}%
and%
\begin{eqnarray*}
\boldsymbol{s}_{j+1}\boldsymbol{s}_{j}\boldsymbol{s}_{j+1} &=&s_{j+1}\hat{r}%
_{j+1}s_{j}\hat{r}_{j}s_{j+1}\hat{r}_{j+1} \\
&=&s_{j+1}s_{j}s_{j+1}(1-t_{jj+1}\delta _{j+1}^{\vee })(1-t_{jj+2}\delta
_{j}^{\vee })(1-t_{j+1j+2}\delta _{j+1}^{\vee })\;.
\end{eqnarray*}%
Multiplying out the brackets in each of the expressions and using that $%
\delta _{j}^{\vee }\delta _{j}^{\vee }=\delta _{j+1}^{\vee }\delta
_{j}^{\vee }=0$, $s_{j+1}s_{j}s_{j+1}=s_{j}s_{j+1}s_{j}$ one finds that both
expressions are equal and, thus, the assertion follows.
\end{proof}

We now prove statement (ii) from Theorem \ref{SWduality}, namely that the
corresponding $\mathbb{S}_{N}$-action commutes with the action of the row
Yang-Baxter algebra (\ref{row_yba}) and the $\mathbb{\hat{S}}_{N}$-action
(dependent on $q$ via (\ref{qr})) with the transfer matrices.

\begin{proposition}
The matrix elements (\ref{row_yba}) of the row monodromy matrix (\ref{rowmom}%
) commute with the above $\mathbb{S}_{N}$-action, i.e.%
\begin{equation}
\boldsymbol{s}_{j}M_{ab}=M_{ab}\boldsymbol{s}_{j},\qquad j=1,\ldots ,N-1\;.
\label{SactionM}
\end{equation}%
In contrast, the row-to-row transfer matrices $H=H(t),~E=E(t)$ obey%
\begin{equation}
\boldsymbol{s}_{j}H=H\boldsymbol{s}_{j}\qquad \text{and}\qquad \boldsymbol{s}%
_{j}E=E\boldsymbol{s}_{j},\qquad j=1,\ldots ,N\;.  \label{SactionHE}
\end{equation}%
These relation continue to hold true for the $q$-deformed transfer matrices.
However, only when setting $q=1$, we have in addition%
\begin{equation}
\Omega ~H(t)~\Omega ^{-1}=H(\varpi t)\qquad \text{and}\qquad \Omega
~E(t)~\Omega ^{-1}=E(\varpi t)\;.  \label{RotHE}
\end{equation}
\end{proposition}

\begin{proof}
From (\ref{ybe}) and (\ref{rowmom}) it follows that 
\begin{eqnarray*}
\hat{r}_{j}M(x|t_{1},\ldots ,t_{N}) &=&L_{N}(x;t_{N})\cdots \hat{r}%
_{j}L_{j+1}(x;t_{j+1})L_{j}(x;t_{j})\cdots L_{1}(x;t_{1}) \\
&=&L_{N}(x;t_{N})\cdots L_{j+1}(x;t_{j})L_{j}(x;t_{j+1})\hat{r}_{j}\cdots
L_{1}(x;t_{1}) \\
&=&M(x|t_{1},\ldots ,t_{j+1},t_{j},\ldots ,t_{N})\hat{r}_{j}
\end{eqnarray*}%
This proves the first assertion. Exploiting (\ref{qybe}) and the cyclicity
of the trace in (\ref{Zdef}) with $n=1$ one obtains the remaining assertions
for $H,E$.
\end{proof}

Replacing the simple Weyl reflection $s_{j}$ in the divided difference
operators (\ref{DBGG}) with the braid matrix $\hat{r}_{j}$ we now attain two
other representations of the affine nil-Coxeter algebra $\mathbb{A}_{N}$
over $\mathcal{V}_{n}$ which restrict to $V^{\otimes N}$ and, hence, commute
with the actions of $\mathbb{\tilde{S}}_{N},\mathbb{A}_{N},\mathbb{H}_{N}$
on $\Lambda $.

\begin{proposition}
\label{prop:nilcoxrep}The maps%
\begin{equation}
a_{j}\mapsto \delta _{j}^{\vee }=\frac{1_{j+1,j}-\hat{r}_{j}}{t_{j}-t_{j+1}}%
=\sigma _{j}^{+}\sigma _{j+1}^{-},\qquad j=1,2,\ldots ,N  \label{rhat}
\end{equation}%
and%
\begin{equation}
a_{j}\mapsto \delta _{j}=\frac{1_{j+1,j}-\check{r}_{j}}{t_{j+1}-t_{j}}%
=\sigma _{j}^{-}\sigma _{j+1}^{+},\qquad j=1,2,\ldots ,N  \label{rcheck}
\end{equation}%
are representations $\mathbb{A}_{N}\rightarrow \limfunc{End}V^{\otimes N}$
with $\mathcal{P}\delta _{N-j}\mathcal{P}=\delta _{j}^{\vee }$. Here all
indices are understood modulo $N$. We have the identities%
\begin{eqnarray}
\delta _{j}^{\vee }v_{\varepsilon _{1}}\otimes \cdots \otimes v_{\varepsilon
_{N}} &=&\left\{ 
\begin{array}{cc}
v_{\varepsilon _{1}}\otimes \cdots v_{\varepsilon _{j+1}}\otimes
v_{\varepsilon _{j}}\cdots \otimes v_{\varepsilon _{N}}, & \varepsilon
_{j}<\varepsilon _{j+1} \\ 
0, & \text{else}%
\end{array}%
\right.  \label{delta} \\
\delta _{j}v_{\varepsilon _{1}}\otimes \cdots \otimes v_{\varepsilon _{N}}
&=&\left\{ 
\begin{array}{cc}
v_{\varepsilon _{1}}\otimes \cdots v_{\varepsilon _{j+1}}\otimes
v_{\varepsilon _{j}}\cdots \otimes v_{\varepsilon _{N}}, & \varepsilon
_{j}>\varepsilon _{j+1} \\ 
0, & \text{else}%
\end{array}%
\right. \;.  \label{dual_delta}
\end{eqnarray}%
Setting $q_{1}=q$ and $q_{2}=\cdots =q_{N}=1$ we obtain $\delta _{N}^{\vee
}(q)=q^{-1}\delta _{N}^{\vee }$ and $\delta _{N}(q)=q\delta _{N}$ instead
and both representations extend to $\mathbb{Z}[q^{\pm 1}]\otimes V^{\otimes
N}$.
\end{proposition}

\begin{remark}
We will see in Section 5 how the operators (\ref{delta}) are related to the
generalised divided difference operators first introduced by Peterson \cite[%
Lectures 6, 13, 14]{Peterson}. The latter have also been discussed in \cite%
{KnutsonTao} and \cite{Tymoczko} in the setting of GKM theory.
\end{remark}

\begin{remark}
The representations (\ref{delta}), (\ref{dual_delta}) factor through the 
\emph{affine nil Temperley-Lieb algebra}, i.e. one has the additional
relation%
\begin{equation}
\delta _{j}\delta _{j+1}\delta _{j}=\delta _{j+1}\delta _{j}\delta
_{j+1}=0\;.
\end{equation}%
The affine nil Temperley-Lieb algebra has been used previously in \cite%
{Postnikov,KorffStroppel,VicOsc} to describe the non-equivariant quantum
cohomology ring.
\end{remark}

\begin{proof}
As already mentioned previously in the proof of Prop \ref%
{prop:symmgroupaction} one verifies from (\ref{ybesolns}) that $\hat{r}%
_{j}=p_{j+1,j}r_{j+1,j}=1-(t_{j}-t_{j+1})\delta _{j}^{\vee }$ and $\check{r}%
_{j}=p_{j,j+1}r_{j,j+1}=1+(t_{j}-t_{j+1})\delta _{j}$ with 
\begin{equation*}
\delta _{j}^{\vee }=\left( 
\begin{array}{cccc}
0 & 0 & 0 & 0 \\ 
0 & 0 & 1 & 0 \\ 
0 & 0 & 0 & 0 \\ 
0 & 0 & 0 & 0%
\end{array}%
\right) _{j+1,j}\qquad \text{and\qquad }\delta _{j}=\left( 
\begin{array}{cccc}
0 & 0 & 0 & 0 \\ 
0 & 0 & 1 & 0 \\ 
0 & 0 & 0 & 0 \\ 
0 & 0 & 0 & 0%
\end{array}%
\right) _{j,j+1}
\end{equation*}%
The assertion about the action (\ref{delta}) and the $\mathbb{A}_{N}$%
-relations are then verified by a straightforward computation.
\end{proof}

\begin{corollary}[Leibniz rules]
The row monodromy matrix (\ref{rowmom}) and its dual (\ref{dual_mom})
satisfy the following cross-relations,%
\begin{eqnarray}
\delta _{j}^{\vee }M &=&(s_{j}M)\delta _{j}^{\vee }+(\partial _{j}M)
\label{Leibniz} \\
\delta _{j}M^{\vee } &=&(s_{j}^{\vee }M^{\vee })\delta _{j}-(\partial
_{j}^{\vee }M^{\vee })  \notag
\end{eqnarray}%
for $j=1,\ldots ,N-1$. Here $s_{j}^{\vee }=s_{N+1-j},\partial _{j}^{\vee
}=-\partial _{N+1-j}$ denote the simple Weyl reflection and divided
difference operator with respect to the $T_{i}=t_{N+1-i}$ parameters.

Replacing $M$ with $H=\limfunc{Tr}_{V}M$ the relation (\ref{Leibniz}) holds
also true for $j=N$, i.e. the affine nil-Coxeter algebra $\mathbb{A}_{N}$.
The identical equations apply to $M^{\prime }$ and $E$. Moreover, we have
the explicit formulae,%
\begin{eqnarray}
\partial _{j}H(x) &=&\frac{\sigma _{j}^{+}\sigma _{j}^{-}H(x)\sigma
_{j+1}^{-}\sigma _{j+1}^{+}}{1-xt_{j+1}}-\frac{\sigma _{j}^{-}\sigma
_{j}^{+}H(x)\sigma _{j+1}^{+}\sigma _{j+1}^{-}}{1-xt_{j}}  \notag \\
\partial _{j}E(x) &=&\frac{\sigma _{j+1}^{-}\sigma _{j+1}^{+}E(x)\sigma
_{j}^{+}\sigma _{j}^{-}}{1+xt_{j}}-\frac{\sigma _{j+1}^{+}\sigma
_{j+1}^{-}E(x)\sigma _{j}^{-}\sigma _{j}^{+}}{1+xt_{j+1}}\;.  \label{dellHE}
\end{eqnarray}
\end{corollary}

\begin{proof}
The cross-relation (\ref{Leibniz}) follows from (\ref{ybe}) and (\ref{rhat}%
), (\ref{rcheck}). Note that 
\begin{equation*}
M_{i}^{\vee }=L_{i1}(x_{i};T_{1})L_{i2}(x_{i};T_{2})\cdots
L_{iN}(x_{i};T_{N})
\end{equation*}%
whence we expressed the second relation in the divided difference operators
for the $T_{i}$'s instead of the $t_{i}$'s.

To derive (\ref{dellHE}) first note that both operators act by
multiplication on $\Lambda $ and the operators $\partial _{j}M,\partial
_{j}^{\vee }M^{\vee },\partial _{j}H,\partial _{j}E$ are therefore
well-defined. Consider the matrix elements of these operators in quantum
space, then one finds from the explicitly given weights in Figure \ref%
{fig:5vmodels} by a case-by-case discussion the given formulae. Note that
despite first appearances $\partial _{j}H,\partial _{j}E$ map polynomial
factors to polynomial factors because the projection operators $\sigma
_{j}^{+}\sigma _{j}^{-},\sigma _{j}^{-}\sigma _{j}^{+}$ ensure that either
the result is zero or the matrix elements of $\sigma _{j}^{+}\sigma
_{j}^{-}H(x)\sigma _{j+1}^{-}\sigma _{j+1}^{+}$ etc. contain $(1-xt_{j+1})$
etc. as a factor.
\end{proof}

\subsection{Quantum Knizhnik-Zamolodchikov equations}

Another way to express the relation between the representation of the affine
braid group on $\mathcal{V}$ and the representations of the affine
nil-Coxeter algebra on $\Lambda $ and $V^{\otimes N}$ is the following
observation.

\begin{lemma}[quantum KZ equations]
\label{lem:qKZ}Let $\Psi =\Psi (t_{1},\ldots ,t_{N})\in \mathcal{V}$ and
consider the actions of the braid matrices $\hat{r}_{j}$ from \ref%
{braidaction}, the actions of $\mathbb{\hat{S}}_{N},~\mathbb{A}_{N}$ on $%
\Lambda $ and the action from Prop \ref{prop:nilcoxrep}. Then the following
identities are equivalent:%
\begin{equation}
\hat{r}_{j}\Psi =s_{j}\Psi \qquad \Leftrightarrow \qquad \delta _{j}^{\vee
}\Psi =\partial _{j}\Psi ,\qquad j=1,2,\ldots ,N\;.  \label{qKZ}
\end{equation}
\end{lemma}

\begin{proof}
Recall the formula $\hat{r}_{j}:=p_{j+1,j}r_{j+1,j}=1-(t_{j}-t_{j+1})\delta
_{j}^{\vee }$. Then%
\begin{equation*}
0=[\hat{r}_{j}-s_{j}]\Psi =[1-s_{j}-(t_{j}-t_{j+1})\delta _{j}^{\vee }]\Psi
\;.
\end{equation*}%
Dividing by $(t_{j}-t_{j+1})$ yields the asserted equivalence.
\end{proof}

\begin{remark}
We shall refer to (\ref{qKZ}) as \emph{quantum Knizhnik-Zamolodchikov} (qKZ)
equations and the space of $\Psi $'s satisfying (\ref{qKZ}) as solution
space. N.B. (\ref{qKZ}) is similar to the equation satisfied by certain
sine-Gordon QFT correlation functions as first pointed out by Smirnov \cite%
{Smirnov}. This is a simpler version of the now more common difference
version due to Frenkel and Reshetikhin \cite{FR} which is often referred to
as qKZ equation in the literature.
\end{remark}

\begin{corollary}
The transfer matrices $H,E$ preserve the solution space of the qKZ equation (%
\ref{qKZ}).
\end{corollary}

\begin{proof}
Let $\Psi $ be a solution of (\ref{qKZ}). Then $\hat{r}_{j}(H\Psi )=(s_{j}H)%
\hat{r}_{j}\Psi =s_{j}(H\Psi )$. The same computation applies to $E$.
\end{proof}

\subsection{Commutation relations of the Yang-Baxter algebras}

The nil-Coxeter algebra allows us to express the commutation relations of
the column Yang-Baxter algebra (\ref{col_yba}) in a simplified form.

\begin{lemma}
\label{lem:cyba_BGG}We have the following commutation relations of the
column Yang-Baxter algebra in terms of divided difference operators%
\begin{equation}
\left\{ 
\begin{array}{l}
a_{j+1}b_{j}=-\partial _{j}b_{j+1}a_{j}, \\ 
c_{j+1}a_{j}=\partial _{j}a_{j+1}c_{j}, \\ 
d_{j+1}b_{j}=\partial _{j}b_{j+1}d_{j}, \\ 
c_{j+1}d_{j}=-\partial _{j}d_{j+1}c_{j}, \\ 
c_{j+1}b_{j}=\partial _{j}a_{j+1}d_{j}=-\partial _{j}d_{j+1}a_{j}%
\end{array}%
\right.
\end{equation}%
and%
\begin{equation}
\mathfrak{x}_{j+1}\mathfrak{x}_{j}=\mathfrak{x}_{j}\mathfrak{x}_{j+1},\qquad 
\mathfrak{x}=a,b,c,d\;.
\end{equation}
\end{lemma}

\begin{proof}
By making use of (\ref{ybe}) we find that 
\begin{equation*}
\hat{r}_{j}m_{j+1}(x|t_{j+1})m_{j}(x|t_{j})=m_{j+1}(x|t_{j})m_{j}(x|t_{j+1})%
\hat{r}_{j}\;.
\end{equation*}%
Suppose $F=1\otimes v\in \mathcal{V}$ with $v\in V^{\otimes N}$. Exploiting (%
\ref{rhat}) we deduce%
\begin{multline*}
\frac{1-s_{j}}{t_{j}-t_{j+1}}m_{j+1}(x|t_{j+1})m_{j}(x|t_{j})F= \\
\delta^\vee
_{j}m_{j+1}(x|t_{j+1})m_{j}(x|t_{j})F-m_{j+1}(x|t_{j})m_{j}(x|t_{j+1})%
\delta^\vee _{j}F
\end{multline*}%
Choosing $v=v_{\varepsilon _{1}}\otimes \cdots \otimes v_{\varepsilon
_{j}}\otimes v_{\varepsilon _{j+1}}\cdots \otimes v_{\varepsilon _{N}}$ with 
$\varepsilon _{j},\varepsilon _{j+1}=0,1$ we obtain the asserted commutation
relations.
\end{proof}

In contrast the $R$-matrices from (\ref{ybe}) for the row monodromy matrix (%
\ref{rowmom}) induce representations of the nil-Hecke algebra.

\begin{proposition}
Set $\hat{R}_{i}=p_{i+1,i}R_{i+1,i}$. Then the map 
\begin{equation}
\pi _{i}\mapsto \frac{1_{i+1,i}-\hat{R}_{i}}{1-x_{i}/x_{i+1}}=\sigma
_{i}^{+}\sigma _{i}^{-}-\sigma _{i}^{+}\sigma _{i+1}^{-},\qquad i=1,\ldots
,n-1  \label{tildepi}
\end{equation}%
is an algebra homomorphism $\mathbb{H}_{n}^{\prime }\rightarrow \limfunc{End}%
V^{\otimes n}$. Similarly, we have that%
\begin{equation}
\pi _{i}\mapsto \frac{1_{i+1,i}-\hat{R}_{i}^{\prime }}{1-x_{i+1}/x_{i}}%
=\sigma _{i+1}^{+}\sigma _{i+1}^{-}-\sigma _{i}^{+}\sigma _{i+1}^{-},\qquad
i=1,\ldots ,k-1  \label{tildepi'}
\end{equation}%
is an algebra homomorphism $\mathbb{H}_{k}^{\limfunc{fin}}\rightarrow 
\limfunc{End}V^{\otimes k}$.
\end{proposition}

\begin{proof}
A straightforward computation.
\end{proof}

Analogous to the column case considered above, we now express also the
commutation relations of the row Yang-Baxter algebra (\ref{row_yba}) in
terms of divided difference operators.

Instead of $\Lambda $ we now consider the polynomial ring $\mathcal{X}_{n}=%
\mathbb{Z}[x_{1},\ldots ,x_{n}]$ where the $x_{i}$'s are the so-called
spectral variables. As before there is a natural action of the symmetric
group $S_{n}$ on $\mathcal{X}_{n}$ by permuting the $x_{i}$'s and we denote
by $\{\tilde{s}_{i}\}_{i=1}^{n-1}$ the elementary transpositions which
exchange $x_{i}$ and $x_{i+1}$. Let $\tilde{\nabla}_{i}$ be the BGG or
Demazure operators with respect to the $x_{i}$ variables. We now have the
following simple but important lemma.

\begin{lemma}
\label{Lem:ybaDemazure}Let $f\in \mathcal{X}_{n}\otimes V^{\otimes N}$ and
suppose $\tilde{s}_{i}f=f$ for all $i=1,\ldots ,n-1$. Then we have the
commutation relations%
\begin{eqnarray}
A(x_{i+1})B(x_{i})f &=&\tilde{\nabla}_{i}B(x_{i+1})A(x_{i})f  \notag \\
D(x_{i+1})B(x_{i})f &=&-\tilde{\nabla}_{i}B(x_{i+1})A(x_{i})f  \notag \\
C(x_{i+1})B(x_{i})f &=&\tilde{\nabla}_{i}D(x_{i+1})A(x_{i})f\;.
\label{ybaDemazure}
\end{eqnarray}

\begin{proof}
The proof is analogous to the one for the column Yang-Baxter algebra using (%
\ref{tildepi}) instead.
\end{proof}
\end{lemma}

\begin{lemma}
The column monodromy matrix (\ref{colmom}) obeys the cross relations%
\begin{equation}
\tilde{\pi}_{i}m_{j}=\left( \tilde{\nabla}_{i}m_{j}\right) +(\tilde{s}%
_{i}m_{j})\tilde{\pi}_{i}  \label{Leibnizx}
\end{equation}%
with $\tilde{\nabla}_{i}=-\tilde{s}_{i}-\tilde{\partial}_{i}x_{i+1}$. The
analogous cross-relation holds for $m_{j}^{\prime }$ but with $\tilde{\nabla}%
_{i}^{\prime }=\tilde{s}_{i}-\tilde{\partial}_{i}x_{i}$.
\end{lemma}

\begin{proof}
The cross-relation (\ref{Leibnizx}) is a consequence of (\ref{mom_ybe}),%
\begin{equation*}
\hat{R}%
_{i}L_{i+1j}(x_{i+1}|t)L_{ij}(x_{i}|t)=L_{i+1j}(x_{i}|t)L_{ij}(x_{i+1}|t)%
\hat{R}_{i}
\end{equation*}%
and the explicit form of the matrices $R,R^{\prime }$ stated earlier in Prop %
\ref{prop:ybe}; see (\ref{ybesolns}).
\end{proof}

\subsection{Representations of the nil-Hecke algebra in quantum space}

We now define a multi-parameter representation of the nil-Hecke algebra in
the extended quantum space $\mathcal{\tilde{V}}:=\tilde{\Lambda}[q]\otimes
V^{\otimes N}$.

\begin{proposition}
The map $\rho _{t}:\mathbb{H}_{N}\rightarrow \limfunc{End}\mathcal{\tilde{V}}
$ given by 
\begin{equation}
\pi _{j}\mapsto \rho _{t}(\pi _{j}):=t_{j}^{-1}\sigma _{j}^{-}\sigma
_{j+1}^{+}-\sigma _{j}^{-}\sigma _{j}^{+}  \label{pi}
\end{equation}%
is an algebra homomorphism. We have the following action on basis vectors,%
\begin{equation}
\rho _{t}(\pi _{j})v_{\varepsilon _{1}}\otimes \cdots \otimes v_{\varepsilon
_{N}}=\left\{ 
\begin{array}{cc}
t_{j}^{-1}v_{\varepsilon _{1}}\otimes \cdots v_{\varepsilon _{j+1}}\otimes
v_{\varepsilon _{j}}\otimes \cdots \otimes v_{\varepsilon _{N}}, & 
\varepsilon _{j}>\varepsilon _{j+1} \\ 
-v_{\varepsilon _{1}}\otimes \cdots \otimes v_{\varepsilon _{N}}, & 
\varepsilon _{j}=0 \\ 
0, & \text{else}%
\end{array}%
\right.  \label{piaction}
\end{equation}%
where all indices are understood modulo $N$.
\end{proposition}

\begin{proof}
A straightforward computation using that $v_{i}=\sigma _{i}^{-}\sigma
_{i}^{+}$ and $u_{i}=t_{i}^{-1}\sigma _{i}^{-}\sigma _{i+1}^{+}$ obey the
relations%
\begin{gather*}
v_{i}^{2} =v_{i},\qquad v_{i}v_{j}=v_{j}v_{i} \\
u_{i}^{2} =u_{i}u_{i+1}u_{i}=u_{i+1}u_{i}u_{i+1}=0 \\
v_{i}u_{i} =u_{i}v_{i+1}=u_{i},\qquad u_{i}v_{i}=v_{i+1}u_{i}=0 \\
u_{i}u_{j} =u_{j}u_{i},\qquad v_{i}u_{j}=u_{j}v_{i}\text{\quad for\quad }%
|i-j|>1\;
\end{gather*}%
which can be easily checked.
\end{proof}

We introduce two additional representations $\rho _{T}^{\vee },\rho
_{t}^{\prime }:\mathbb{H}_{N}\rightarrow \limfunc{End}\mathcal{\tilde{V}}$
which will occur because of Poincar\'{e} and level-rank duality:%
\begin{equation}
\rho _{T}^{\vee }(\pi _{j}):=\mathcal{P}\rho _{t}(\pi _{j})\mathcal{P\quad }%
\text{and}\mathcal{\quad }\rho _{t}^{\prime }(\pi _{j}):=-\Theta \rho _{-T}(%
\bar{\pi}_{N-j})\Theta \;,  \label{dual_pis}
\end{equation}%
where we recall that $\bar{\pi}_{j}=\pi _{j}+1,~T_{i}=t_{N+1-i}$ and we have
used the natural extension of the maps (\ref{parity}) and (\ref{charge}) to $%
\mathcal{V}$ with $\Theta =\mathcal{PC}$. By abuse of notation we will often
denote $\rho _{t}(\pi _{j}),\rho _{t}^{\prime }(\pi _{j}),\rho _{T}^{\vee
}(\pi _{j})$ simply by $\pi _{j},\pi _{j}^{\prime },\pi _{j}^{\vee }$ in
what follows if no confusion can arise. Explicitly, we have $\pi _{j}^{\vee
}=T_{j+1}^{-1}\sigma _{j}^{+}\sigma _{j+1}^{-}-\sigma _{j+1}^{-}\sigma
_{j+1}^{+}$ and $\pi _{j}^{\prime }=t_{j+1}^{-1}\sigma _{j}^{-}\sigma
_{j+1}^{+}-\sigma _{j+1}^{-}\sigma _{j+1}^{+}$.

Consider the following decomposition of the quantum space (\ref%
{quantum_space})%
\begin{equation}
\mathcal{\tilde{V}}=\tbigoplus_{n=0}^{N}\mathcal{\tilde{V}}_{n},\qquad 
\mathcal{\tilde{V}}_{n}=\tilde{\Lambda}[q]\otimes V_{n}~,
\end{equation}%
where $V_{n}$ is spanned by the basis vectors $v_{\varepsilon _{1}}\otimes
\cdots \otimes v_{\varepsilon _{N}}$ with exactly $n$ of the $\varepsilon
_{i}$'s equal to one, i.e. $\sum_{i=1}^{N}\varepsilon _{i}=n$. In each
subspace $\mathcal{V}_{n}$ we label the basis vectors $v_{\varepsilon
_{1}}\otimes \cdots \otimes v_{\varepsilon _{N}}$ in terms of partitions $%
\lambda $ in the set 
\begin{equation}
\Pi _{n,k}:=\{\lambda :k\geq \lambda _{1}\geq \cdots \geq \lambda _{n}\geq
0\}
\end{equation}%
by using the same bijection between 01-words and partitions as in \cite%
{KorffStroppel,VicOsc}. Let $\ell _{j}(\lambda )=\lambda _{n+1-j}+j$ be the
position of the one-letters in the 01-word $\varepsilon _{1}\cdots
\varepsilon _{N}$ then we set%
\begin{equation}
|\lambda \rangle =v_{\varepsilon _{1}}\otimes \cdots \otimes v_{\varepsilon
_{N}}\in V^{\otimes N}\,,  \label{basis}
\end{equation}%
where $\varepsilon _{j}=1$ if $j=\lambda _{n+1-j}+j$ for $j=1,\ldots ,n$ and 
$\varepsilon _{j}=0$ else. Note that the Young diagram of $\lambda $ fits
into the bounding box of height $n$ and width $k$. The following lemma is a
direct consequence of these definitions.

\begin{lemma}
The action (\ref{piaction}) of the nil-Hecke algebra leaves each subspace $%
\mathcal{\tilde{V}}_{n}$ invariant. The same applies to $\rho _{t}^{\prime }$
and $\rho _{T}^{\vee }$.
\end{lemma}

The following result is contained in \cite[Lemma 32, p.1577]{Lam} where in 
\emph{loc. cit.} cylindric shapes are used which we will discuss in a later
section. Here we simply associate with the cylindric shape $\lambda \lbrack
d]$ the vector $q^{d}\otimes |\lambda \rangle \in \mathbb{C}[q^{\pm
1}]\otimes \mathcal{\tilde{V}}_{n}$.

\begin{lemma}[Lam]
\label{lem:lam_perm}Let $\lambda ,\mu\in\Pi_{n,k}$ and denote by $|\lambda
\rangle ,|\mu \rangle \in V_{n}$ the corresponding basis vectors (\ref{basis}%
). For fixed $d\in \mathbb{Z}$ there exists at most one affine permutation $%
w(\lambda ,\mu ,d)\in \mathbb{\hat{S}}_{N}$ such that $q^{d}|\lambda \rangle
=\delta _{w(\lambda ,\mu ,d)}|\mu \rangle $. If $\mu =\emptyset $ then for
each $\lambda $ there exists a unique (finite) permutation $w_{\lambda }\in 
\mathbb{S}_{N}\subset \mathbb{\hat{S}}_{N}$ such that $|\lambda \rangle
=\delta _{w_{\lambda }}|\emptyset \rangle $.
\end{lemma}

\subsection{Transfer matrices as cyclic words in the nil-Hecke algebra}

We now expand the elements of the monodromy matrices of the vicious and
osculating walker models according to $\mathcal{O}(x_{i})=\sum_{r\geq 0}%
\mathcal{O}_{r}x_{i}^{r}$ with $\mathcal{O}=A,B,C,D$ and $\mathcal{O}%
=A^{\prime },B^{\prime },C^{\prime },D^{\prime }$. The coefficients $%
\mathcal{O}_{r}$ have simple expressions in the multi-parameter
representation of the nil-Hecke algebra discussed above.

Define the following ordered sums in terms of the nil-Hecke algebra,%
\begin{eqnarray}
\Sigma _{r}^{N,>} &=&\sum_{0<j_{1}<\cdots <j_{r}<N}(\delta _{j_{r}}-\hat{t}%
_{j_{r}})\cdots (\delta _{j_{1}}-\hat{t}_{j_{1}})  \notag \\
&=&\sum_{0<j_{1}<\cdots <j_{r}<N}t_{j_{1}}t_{j_{2}}\cdots t_{j_{r}}~\pi
_{j_{r}}\cdots \pi _{j_{2}}\pi _{j_{1}}
\end{eqnarray}%
and%
\begin{eqnarray}
\Sigma _{r}^{N,<} &=&\sum_{0<j_{1}<\cdots <j_{r}<N}(\delta _{j_{1}}+\check{t}%
_{j_{1}})\cdots (\delta _{j_{r}}+\check{t}_{j_{r}})  \notag \\
&=&\sum_{0<j_{1}<\cdots <j_{r}<N}t_{j_{1}}t_{j_{2}}\cdots t_{j_{r}}~\bar{\pi}%
_{j_{1}}\bar{\pi}_{j_{2}}\cdots \bar{\pi}_{j_{r}}
\end{eqnarray}%
where we recall that $\delta _{j}=\sigma _{j}^{-}\sigma _{j+1}^{+}$ and we
have set $\hat{t}_{j}=t_{j}\sigma _{j}^{-}\sigma _{j}^{+},\ \check{t}%
_{j}=t_{j}\sigma _{j}^{+}\sigma _{j}^{-}.$

\begin{lemma}
We have the following identities%
\begin{eqnarray*}
\delta _{j}\hat{t}_{j+1} &=&\hat{t}_{j}\delta _{j}+1-\check{r}_{j},\qquad
\delta _{j}\hat{t}_{j}=\delta _{j}\hat{t}_{j+1}=0 \\
\delta _{j}^{\vee }\hat{t}_{j} &=&\hat{t}_{j+1}\delta _{j}^{\vee }+1-\hat{r}%
_{j},\qquad \delta _{j}^{\vee }\hat{t}_{j+1}=\hat{t}_{j}\delta _{j}^{\vee }=0
\end{eqnarray*}%
and%
\begin{eqnarray*}
\delta _{j}^{\vee }\check{t}_{j+1} &=&\check{t}_{j}\delta _{j}^{\vee }+1-%
\check{r}_{j},\qquad \delta _{j}^{\vee }\check{t}_{j}=\check{t}_{j+1}\delta
_{j}^{\vee }=0 \\
\delta _{j}\check{t}_{j} &=&\check{t}_{j+1}\delta _{j}+1-\hat{r}_{j},\qquad
\delta _{j}\check{t}_{j+1}=\check{t}_{j}\delta _{j}=0
\end{eqnarray*}
\end{lemma}

\begin{proof}
A simple computation.
\end{proof}

The last set of relations should be thought of as a generalisation of the
cross relation or Leibniz rule (\ref{LeibnizBGG}) for the braid group
representation.

\begin{proposition}
We have the following explicit expressions for the monodromy matrix elements
of the vicious and osculating walker models in terms of the nil Hecke
algebra representation (\ref{pi}),%
\begin{equation}
\left\{ 
\begin{array}{l}
A_{r}=\Sigma _{r}^{N,>}-\hat{t}_{N}\Sigma _{r-1}^{N,>} \\ 
B_{r}=A_{r-1}\sigma _{1}^{+} \\ 
C_{r}=\sigma _{N}^{-}A_{r} \\ 
D_{r}=\sigma _{N}^{-}A_{r-1}\sigma _{1}^{+}%
\end{array}%
\right. \qquad \text{and}\qquad \left\{ 
\begin{array}{l}
A_{r}^{\prime }=\Sigma _{r}^{N,<}+\Sigma _{r-1}^{N,<}\check{t}_{N} \\ 
B_{r}^{\prime }=\sigma _{1}^{+}A_{r-1}^{\prime } \\ 
C_{r}^{\prime }=A_{r}^{\prime }\sigma _{N}^{-} \\ 
D_{r}^{\prime }=\sigma _{1}^{+}A_{r-1}^{\prime }\sigma _{N}^{-}%
\end{array}%
\right. \;.  \label{YBA2Hecke}
\end{equation}
\end{proposition}

\begin{proof}
Proceed by induction in $N$ using the coproduct $\Delta =\Delta ^{\func{col}%
} $ in (\ref{cops}) of the row Yang-Baxter algebra (\ref{row_yba}). The
latter leads to the formulae%
\begin{eqnarray*}
A^{(N)}(x) &=&\Delta (A^{(N-1)}(x))=A^{(N-1)}(x)\otimes (1_{N}-x\hat{t}%
_{N})+C^{(N-1)}(x)\otimes x\sigma _{N}^{+} \\
B^{(N)}(x) &=&\Delta (B^{(N-1)}(x))=B^{(N-1)}(x)\otimes (1_{N}-x\hat{t}%
_{N})+D^{(N-1)}(x)\otimes x\sigma _{N}^{+} \\
C^{(N)}(x) &=&\Delta (C^{(N-1)}(x))=A^{(N-1)}(x)\otimes \sigma
_{N}^{-}+C^{(N-1)}(x)\otimes x\sigma _{N}^{-}\sigma _{N}^{+} \\
D^{(N)}(x) &=&\Delta (D^{(N-1)}(x))=B^{(N-1)}(x)\otimes \sigma
_{N}^{-}+D^{(N-1)}(x)\otimes x\sigma _{N}^{-}\sigma _{N}^{+}
\end{eqnarray*}%
which are a direct consequence of (\ref{L}) and (\ref{rowmom}). The
analogous formulae hold for the Yang-Baxter algebra of the osculating walker
model.
\end{proof}

\begin{example}
For $N=1$ the row monodromy matrix is simply the $L$-operator and we have,%
\begin{equation*}
A^{(1)}=1-\hat{t}_{1}x,\;B^{(1)}=x\sigma _{1}^{+},\;C^{(1)}=\sigma
_{1}^{-},\;D^{(1)}=x\sigma _{1}^{-}\sigma _{1}^{+}\;.
\end{equation*}%
Thus, for $N=2$ we find from the coproduct%
\begin{eqnarray*}
A^{(2)} &=&\left( 1-\hat{t}_{1}x\right) \left( 1-\hat{t}_{2}x\right)
+x\delta _{1} \\
B^{(2)} &=&x\left( 1-\hat{t}_{2}x\right) \sigma _{1}^{+}+x^{2}\delta
_{1}\sigma _{1}^{+} \\
C^{(2)} &=&\sigma _{2}^{-}\left( 1-\hat{t}_{1}x\right) +x\sigma
_{2}^{-}\delta _{1} \\
D^{(2)} &=&x\sigma _{2}^{-}\sigma _{1}^{+}+x^{2}\sigma _{2}^{-}\delta
_{1}\sigma _{1}^{+}\;.
\end{eqnarray*}%
Thus, for $N=3$ we end up with the following formula for the $A,D$ operators%
\begin{eqnarray*}
A^{(3)} &=&\left( 1-\hat{t}_{3}x\right) \left( 1-\hat{t}_{2}x\right) \left(
1-\hat{t}_{1}x\right) \\
&&+x\delta _{2}\left( 1-\hat{t}_{1}x\right) +x\delta _{1}\left( 1-\hat{t}%
_{3}x\right) +x^{2}\delta _{2}\delta _{1}
\end{eqnarray*}%
and%
\begin{equation*}
D^{(3)}=x\sigma _{3}^{-}\left( 1-\hat{t}_{2}x\right) \sigma
_{1}^{+}+x^{2}\sigma _{3}^{-}(\delta _{1}+\delta _{2})\sigma
_{1}^{+}+x^{3}\sigma _{3}^{-}\delta _{2}\delta _{1}\sigma _{1}^{+}\;.
\end{equation*}%
This is of the stated form (\ref{YBA2Hecke}) when taking into account the
various cancellation of terms due to the relations $\hat{t}_{j+1}\sigma
_{j+1}^{+}=\hat{t}_{j+1}\delta _{j}=\delta _{j}\hat{t}_{j}=\sigma _{j}^{-}%
\hat{t}_{j}=0$ etc.
\end{example}

We will now deduce from these results that after taking partial traces of
the respective monodromy matrices leads to simple expressions of the
transfer matrices in terms of the affine nil-Hecke algebra.

Given an affine permutation $w\in \mathbb{\hat{S}}_{N},$ a \emph{reduced word%
} of $w$ is a minimal length decomposition of $w$ as a product $%
s_{j_{1}}\cdots s_{j_{r}}$ of simple reflections. Identify $w$ with the word 
$j_{1}\cdots j_{r}\in \lbrack N]^{r}$ given by its reduced expression and
set $\pi _{w}=\pi _{j_{1}}\cdots \pi _{j_{r}}$, $t_{w}=t_{j_{1}}t_{j_{2}}%
\cdots t_{j_{r}}$.

\begin{definition}
A word $w=j_{1}\cdots j_{r}$ with letters in $[N]$ is called
(anti-)clockwise ordered if each letter occurs at most once and the letter $%
(j+1)\func{mod}N$ succeeds (precedes) the letter $j$ in case both are
present.
\end{definition}

In the literature clockwise ordered words are also called \emph{cyclically
increasing} and anticlockwise ordered ones \emph{cyclically decreasing}. An
affine permutation is called cyclically decreasing or increasing if its
associated reduced word expression has the respective property.

\begin{corollary}
We have the following expressions of the row-to-row transfer matrices 
\begin{equation}
H_{r}=A_{r}+qD_{r}=\sum_{|w|=r}^{\circlearrowleft }(\delta _{j_{1}}-\hat{t}%
_{j_{1}})\cdots (\delta _{j_{r}}-\hat{t}_{j_{r}})=\sum_{|w|=r}^{%
\circlearrowleft }t_{w}\rho _{t}(\pi _{w})~  \label{HHecke}
\end{equation}%
and 
\begin{equation}
E_{r}=A_{r}^{\prime }+qD_{r}^{\prime }=\sum_{|w|=r}^{\circlearrowright
}(\delta _{j_{1}}+\check{t}_{j_{1}})\cdots (\delta _{j_{r}}+\check{t}%
_{j_{r}})=\sum_{|w|=r}^{\circlearrowright }t_{w}\rho _{t}(\bar{\pi}_{w})~,
\label{EHecke}
\end{equation}%
were the sums run respectively over all anti-clockwise (cyclically
decreasing) and clockwise (cyclically increasing) ordered words $%
w=j_{1}\ldots j_{r}$ of length $r$ with $1\leq j_{1},\ldots ,j_{r}\leq N$.
\end{corollary}

\begin{remark}
Note that while the representation $\rho _{t}$ is defined only over the
Laurent polynomials $\tilde{\Lambda}[q]$ the transfer matrices are already
defined over the original quantum space (\ref{quantum_space}), i.e. the
transfer matrices are polynomial in the equivariant parameters. In fact, if
we define $\upsilon _{j}:=t_{j}\pi _{j}=\delta _{j}+\hat{t}_{j}$ for $%
j=1,2,\ldots ,N$ the operators $\upsilon _{j}:\mathcal{V}_{n}\rightarrow 
\mathcal{V}_{n}$ obey the deformed nil-Hecke algebra relations%
\begin{equation*}
\left\{ 
\begin{array}{c}
\upsilon _{j}^{2}=-t_{j}\upsilon _{j} \\ 
\upsilon _{j}\upsilon _{j+1}\upsilon _{j}=t^{\alpha _{j}}\upsilon
_{j+1}\upsilon _{j}\upsilon _{j+1}%
\end{array}%
\right. ,
\end{equation*}%
where $\alpha _{j}=e_{j}-e_{j+1}$ is the $j$th simple root of type $A_{N}$.
Setting $t_{1}=\cdots =t_{N}=0$ we recover the nil-Coxeter algebra
relations, while setting $t_{1}=\cdots =t_{N}=1$ we obtain the nil-Hecke
algebra relations. Similar algebras have been considered before in the
context of the Fomin-Stanley algebra \cite{FominKirillov}.
\end{remark}

\begin{corollary}
Under the involution $\Theta =\mathcal{PC}:\mathcal{V}\rightarrow \mathcal{V}
$ we have the transformations%
\begin{eqnarray}
\Theta E_{r}(t)\Theta &=&H_{r}(-T)=\sum_{w}\varpi ^{-1}(T_{w})~\rho
_{-T}^{\prime }(\pi _{w}), \\
\Theta H_{r}(t)\Theta &=&E_{r}(-T)=\sum_{w}\varpi ^{-1}(T_{w})~\rho
_{-T}^{\prime }(\bar{\pi}_{w}^{\prime })\;,
\end{eqnarray}%
which result in the alternative expressions%
\begin{equation}
H_{r}=\sum_{w,~\circlearrowleft }\varpi ^{-1}(t_{w})\rho _{t}^{\prime }(\pi
_{w})\qquad \text{and}\qquad E_{r}=\sum_{w,~\circlearrowright }\varpi
^{-1}(t_{w})\rho _{t}^{\prime }(\bar{\pi}_{w})\;.
\end{equation}
\end{corollary}

\begin{proof}
First we compute that%
\begin{eqnarray}
\Theta \rho _{t}(\pi _{j})\Theta &=&t_{j}^{-1}\sigma _{N+1-j}^{+}\sigma
_{N-j}^{-}-\sigma _{N+1-j}^{+}\sigma _{N+1-j}^{-}=-\rho _{-T}^{\prime }(\bar{%
\pi}_{N-1})  \notag \\
\Theta \rho _{t}(\bar{\pi}_{j})\Theta &=&t_{j}^{-1}\sigma _{N+1-j}^{+}\sigma
_{N-j}^{-}+\sigma _{N+1-j}^{-}\sigma _{N+1-j}^{+}=-\rho _{-T}^{\prime }(\pi
_{N-1})
\end{eqnarray}%
Replacing $t\rightarrow -T$ we obtain the desired formulae.
\end{proof}

\subsection{Affine stable Grothendieck polynomials}

We first recall the definition of these polynomials which were introduced in 
\cite{Lam}. Let $(\cdot ,\cdot ):\mathbb{H}_{N}\times \mathbb{H}%
_{N}\rightarrow \mathbb{Z}$ be the bilinear form defined on the standard
basis $\{\pi _{w}\}$ by setting $(\pi _{w},\pi _{w^{\prime }})=\delta
_{w,w^{\prime }}$ for any $w,w^{\prime }\in \mathbb{\hat{S}}_{N}$. Set $%
\boldsymbol{h}_{r}=\sum_{\omega }^{\circlearrowleft }\pi _{\omega }$ where
the sum is over all cyclically anticlockwise ordered (decreasing) words $%
\omega =i_{1}\cdots i_{r}$ of length $r$. Then the \emph{affine stable
Grothendieck polynomial} is defined as 
\begin{equation}
\tilde{G}_{w}(x_{1},\ldots ,x_{n})=\sum_{\alpha }(\boldsymbol{h}_{\alpha
},\pi _{w})x^{\alpha },
\end{equation}%
where the sum is over all compositions $\alpha =(\alpha _{1},\ldots ,\alpha
_{n})$ of $\ell (w)$ and $\boldsymbol{h}_{\alpha }=\boldsymbol{h}_{\alpha
_{n}}\cdots \boldsymbol{h}_{\alpha _{1}}$. Because of the definition of the
bilinear form $(\cdot ,\cdot )$ the above sum over compositions $\alpha $
effectively restricts to a sum over all decompositions of $w$ into
cyclically decreasing subwords \cite{Lam}. The following corollary states
that we if replace the $\boldsymbol{h}_{r}$'s with our $H_{r}$'s given in (%
\ref{HHecke}) then we can identify for $t_{1}=\cdots =t_{N}=1$ the partition
function of our integrable model with an affine stable Grothendieck
polynomial in the representation (\ref{pi}).

\begin{corollary}
For $\lambda ,\mu \in \Pi _{n,k}$ and $d\in \mathbb{Z}$ fixed let $%
w=w(\lambda ,\mu ,d)$ be the unique affine permutation from Lemma \ref%
{lem:lam_perm} such that $\delta _{w}|\mu \rangle =q^{d}|\lambda \rangle $.
Then%
\begin{equation}
\langle \lambda |Z_{n}(x|t)|\mu \rangle =\sum_{\alpha }\langle \lambda
|H_{\alpha }|\mu \rangle x^{\alpha }=\sum_{a}t_{w}\langle \lambda |\rho
_{t}(\pi _{w})|\mu \rangle x^{a},
\end{equation}%
where the second sum runs over all decompositions $a=(a_{1},\ldots ,a_{n})$
of $w$ into cyclically anti-clockwise ordered words.
\end{corollary}

\begin{remark}
The non-affine versions, \emph{stable} Grothendieck polynomials, have been
considered by Buch in the context of the $K$-theory of Grassmannians \cite%
{BuchKtheory}. The affine polynomials (for $w$ being an affine Grassmannian
permutation) have been identified with the Schubert basis in the $K$%
-cohomology of the affine Grassmannian \cite{LSS} and for 321-avoiding $w$'s
they have been conjectured \cite{Lam} to be related to the quantum $K$%
-theory of Grassmannians \cite{BM}.
\end{remark}

\section{Inhomogeneous Vicious and Osculating Walkers}


\begin{figure}[tbp]
\begin{equation*}
\includegraphics[scale=0.32]{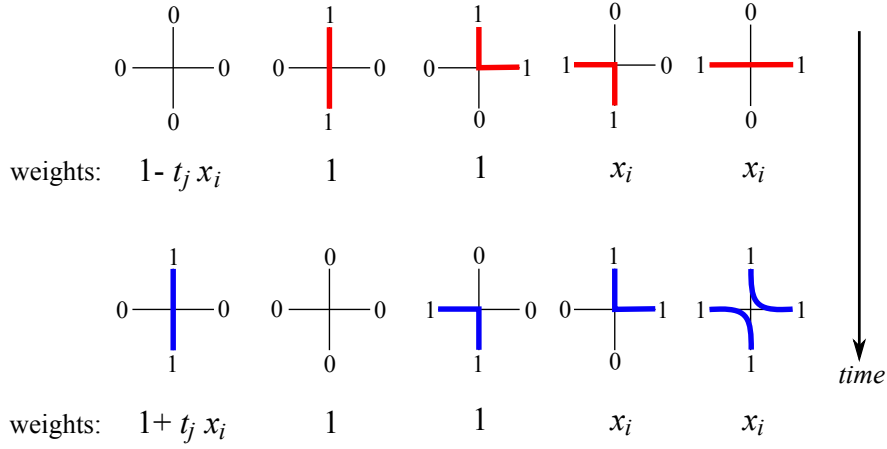}
\end{equation*}%
\caption{The weights of the five allowed vertex configurations for the
vicious (top) and osculating walker (bottom) model. Here $x_{i}$ and $t_{j}$
are commuting indeterminates with $i$ being the row index and $j$ the column
index of the square lattice.}
\label{fig:5vmodels}
\end{figure}

Consider an $n\times N$ square lattice ($n$ rows and $N$ columns) where we
assign statistical variables $\varepsilon =0,1$ to the internal lattice
edges $\mathbb{E}$. be We call a map $\mathbb{E}\rightarrow \{0,1\}$ a
lattice configuration $\mathcal{C}$. Fix two partitions $\lambda ,\mu $
whose Young diagrams fit into the bounding box of height $n$ and width $k$.
Let $\varepsilon (\lambda )$ and $\varepsilon (\mu )$ be the corresponding
01-words of length $N$ with $n$ 1-letters at positions $\ell _{i}(\lambda
)=\lambda _{n+1-i}+i$ and $\ell _{i}(\mu )=\mu _{n+1-i}+i$. We shall only
consider configurations $\mathcal{C}=\mathcal{C}(\lambda ,\mu )$ where the
values on the outer edges on the top and bottom are fixed by the 01-words $%
\varepsilon (\mu )$ and $\varepsilon (\lambda )$, while imposing periodic
boundary conditions in the horizontal direction; this setup is the same as
in \cite{VicOsc}.


\begin{figure}[tbp]
\begin{equation*}
\includegraphics[scale=0.24]{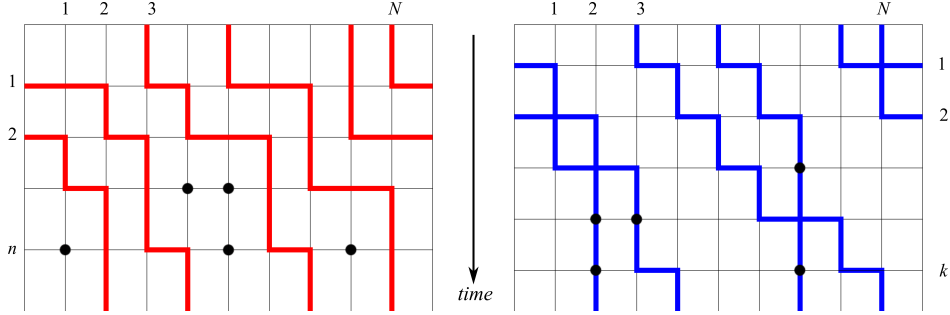}
\end{equation*}%
\caption{Examples of a vicious (left) and osculating (right) walker
configuration.}
\label{fig:walks}
\end{figure}

For each such lattice configuration $\mathcal{C}$ define a weight in $%
\mathbb{C}[x_{1},\ldots ,x_{n}]\otimes \Lambda \lbrack q]$ through the
product over the vertex weights depicted at the top in Figure \ref%
{fig:5vmodels},%
\begin{equation}
\limfunc{wt}(\mathcal{C}):=q^{d(\mathcal{C})}\prod_{\text{vertices }v\in 
\mathcal{C}}\limfunc{wt}(v)\;.  \label{wtC}
\end{equation}%
If a given configuration $\mathcal{C}$ contains vertices which are not shown
in Figure \ref{fig:5vmodels} then we assign those vertices weight zero,
meaning that $\limfunc{wt}(\mathcal{C})=0$ and we call such configurations 
\emph{forbidden}. Finally, $d(\mathcal{C})$ in (\ref{wtC}) denotes the
number of paths transgressing the boundary where the lattice is glued
together to give the cylinder. Namely, each configuration can be identified
with $n$ non-intersecting paths on the cylinder with start positions $\ell
_{i}(\mu )$ and end positions $\ell _{i}(\lambda )$ by connecting lattice
edges which have value 1; see Figure \ref{fig:walks}. Following Fisher \cite%
{Fisher} we call this statistical model the (inhomogeneous) vicious walker
model, although our choice of weights constitutes a non-trivial extension of
his model.

Similarly, we define an inhomogeneous extension of Brak's osculating walker
model \cite{Brak} on a $k\times N$ square lattice ($k$ rows and $N$ columns)
using the weights depicted at the bottom in Figure \ref{fig:5vmodels}. Again
our choice of weights is special and differs from the one in \emph{loc. cit.}

\begin{lemma}
The partition functions of the inhomogeneous vicious and osculating walker
models%
\begin{equation}
Z_{\lambda ,\mu }(x|t)=\sum_{\mathcal{C}}\limfunc{wt}(\mathcal{C})=\sum_{%
\mathcal{C}}q^{d(\mathcal{C})}\prod_{v\in \mathcal{C}}\limfunc{wt}(v)
\label{partition_function}
\end{equation}%
is given in terms of the matrix elements of the operator (\ref{Zdef}). In
particular, the partition functions of a single lattice row for the vicious
walker model are given by the matrix elements of the operator,%
\begin{equation}
H(x|t)=A(x)+qD(x)=\sum_{r\geq 0}x^{r}H_{r},  \label{H}
\end{equation}%
and for the osculating model by the matrix elements of the operator,%
\begin{equation}
E(u|t)=A^{\prime }(x)+qD^{\prime }(x)=\sum_{r\geq 0}x^{r}E_{r}\;.  \label{E}
\end{equation}
\end{lemma}

\begin{proof}
Set $Z(x|t)|\mu \rangle =\sum_{\lambda \in \Pi _{n,k}}Z_{\lambda ,\mu
}(x|t)|\lambda \rangle $. Then the partition function with fixed start and
end positions is given by 
\begin{equation*}
Z_{\lambda ,\mu }(x|t)=\sum_{\mathcal{C}}q^{d(\mathcal{C})}\prod_{v\in 
\mathcal{C}}(L_{ij})_{a(v)b(v)}^{c(v)d(v)}=\sum_{\mathcal{C}}q^{d(\mathcal{C}%
)}\prod_{v\in \mathcal{C}}\limfunc{wt}(v)\;,
\end{equation*}%
where $a,b,c,d=0,1$ are the values of the W, N, E, and S edge of the vertex $%
v$ for the configuration $\mathcal{C}$.
\end{proof}

\subsection{Combinatorial formulae}

Given a partition $\lambda $ inside the $n\times k$ bounding box we recall
the definition of its cylindric loop $\lambda \lbrack r]$; see \cite%
{GesselKrattenthaler}, \cite{Postnikov}.

\begin{definition}[cylindric loops]
Let $\lambda =(\lambda _{1},\ldots ,\lambda _{n})\in (n,k)$ and define the
following associated \emph{cylindric loops} $\lambda \lbrack r]$ for any $%
r\in \mathbb{Z}$,%
\begin{equation}
\lambda \lbrack r]:=(\ldots ,\underset{r}{\lambda _{n}+r+k},\underset{r+1}{%
\lambda _{1}+r},\ldots ,\underset{r+n}{\lambda _{n}+r},\underset{r+n+1}{%
\lambda _{1}+r-k},\ldots )\;.  \notag
\end{equation}
\end{definition}

For $r=0$ the cylindric loop can be visualized as a path in $\mathbb{Z}%
\times \mathbb{Z}$ determined by the outline of the Young diagram of $%
\lambda $ which is periodically continued with respect to the vector $(n,-k)$%
. For $r\neq 0$ this line is shifted $r$ times in the direction of the
lattice vector $(1,1)$.

\begin{definition}[cylindric \& toric skew diagrams]
Given two partitions $\lambda ,\mu \in (n,k)$ denote by $\lambda /d/\mu $
the set of squares between the two lines $\lambda \lbrack d]$ and $\mu
\lbrack 0]$ modulo integer shifts by $(n,-k)$, 
\begin{equation*}
\lambda /d/\mu :=\{\langle i,j\rangle \in \mathbb{Z}\times \mathbb{Z}/(n,-k)%
\mathbb{Z}~|~\lambda \lbrack d]_{i}\geq j>\mu \lbrack 0]_{i}\}\;.
\end{equation*}%
We shall refer to $\lambda /d/\mu $ as a \emph{cylindric skew-diagram} of
degree $d$. If the number of boxes in each row does not exceed $k$ then $%
\lambda /d/\mu $ is called \emph{toric}.
\end{definition}

A cylindric skew diagram $\nu /d/\mu $ which has at most one box in each
column will be called a (cylindric/toric) \emph{horizontal strip} and one
which has at most one box in each row a (cylindric/toric) \emph{vertical
strip}. The \emph{length} of such strips will be the number of boxes within
the skew diagram.

\begin{proposition}[vicious walkers]
Let $\mu \in \Pi _{n,k}$ and $\tilde{H}(x)=x^{k}H(x^{-1})|_{\mathcal{V}_{n}}$%
. We have the following combinatorial action of the vicious walker transfer
matrix,%
\begin{equation}
\tilde{H}(x)|\mu \rangle =\sum_{d=0,1}q^{d}\sum_{\lambda /d/\mu \text{ hor
strip}}\prod_{j\in J_{\lambda /d/\mu }}(x-t_{j})|\lambda \rangle \;,
\label{combH}
\end{equation}%
where the second sum runs over all partitions $\lambda $ (inside the same
bounding box as $\mu $) such that $\lambda /d/\mu $ is a cylindric
horizontal strip and the set $J_{\lambda /d/\mu }$ consists of the diagonals 
$j-i+n$ of the bottom squares $s=(i,j)$ in each column which does not
intersect $\lambda /d/\mu $. If the column does not contain any boxes add $%
j+n$ to $J_{\lambda /d/\mu }$.
\end{proposition}


\begin{figure}[tbp]
\begin{equation*}
\includegraphics[scale=0.5]{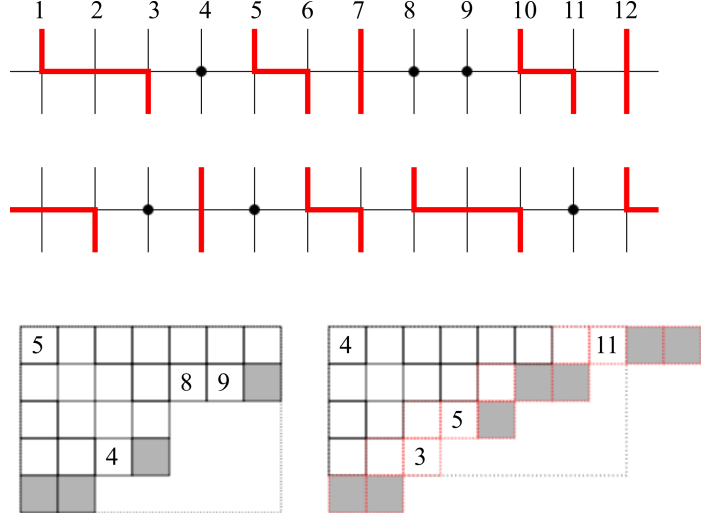}
\end{equation*}%
\caption{A row configuration corresponding to the $A$ (top) and $D$-operator
(bottom) of the vicious walker model. Displayed below are the corresponding
skew Young diagrams, left and right, respectively.}
\label{fig:ADopVic}
\end{figure}

\begin{example}
We explain (\ref{combH}) on an example; the reader should refer to Figure %
\ref{fig:ADopVic} and recall that $H=A+qD$. Set $N=12$ and $n=6$. Consider
the row configuration shown on top in Figure \ref{fig:ADopVic}. Draw the
Young diagram of $\mu =(7,6,4,3,0)$ which corresponds to the 01-word fixing
the values of the vertical lattice edges on top. Starting with the leftmost
path add a box for each horizontal path edge in the bottom row of the Young
diagram. Then do the same for the next path in the row above, etc. If there
are no horizontal path edges do not add any boxes. The number in the boxes
of the resulting skew diagram give the diagonals $+~n$ for each bottom
square in those columns which do not intersect the horizontal strip.

To obtain the horizontal strip for the $D$-operator (bottom row
configuration in Figure \ref{fig:ADopVic}) one has to add a boundary ribbon
of $N$ boxes to $\lambda =(6,4,2,1)$ first whose boxes are indicated by a
red dotted line. Note that $n=4$ in this case.
\end{example}

We now state the analogous combinatorial action for the transfer matrix of
the osculating walker model.

\begin{proposition}[osculating walkers]
Let $\mu \in \Pi _{n,k}$ and $\tilde{E}(x)=x^{n}E(x^{-1})|_{\mathcal{V}_{n}}$%
. We have the following combinatorial action of the osculating walker
transfer matrix,%
\begin{equation}
\tilde{E}(x)|\mu \rangle =\sum_{d=0,1}q^{d}\sum_{\lambda ^{\prime }/d/\mu
^{\prime }\text{ hor strip}}\prod_{j\in J_{\lambda ^{\prime }/d/\mu ^{\prime
}}^{\prime }}(x+T_{j})|\lambda \rangle \;,  \label{combE}
\end{equation}%
where $\lambda ^{\prime },\mu ^{\prime }$ denote the conjugate partitions of 
$\lambda ,\mu $, $T_{j}=t_{N+1-j}$ and $J_{\lambda ^{\prime }/d/\mu ^{\prime
}}^{\prime }$ is defined analogously to $J_{\lambda ^{\prime }/d/\mu
^{\prime }}$ but with $n$ replaced by $k$.
\end{proposition}


\begin{figure}[tbp]
\begin{equation*}
\includegraphics[scale=0.5]{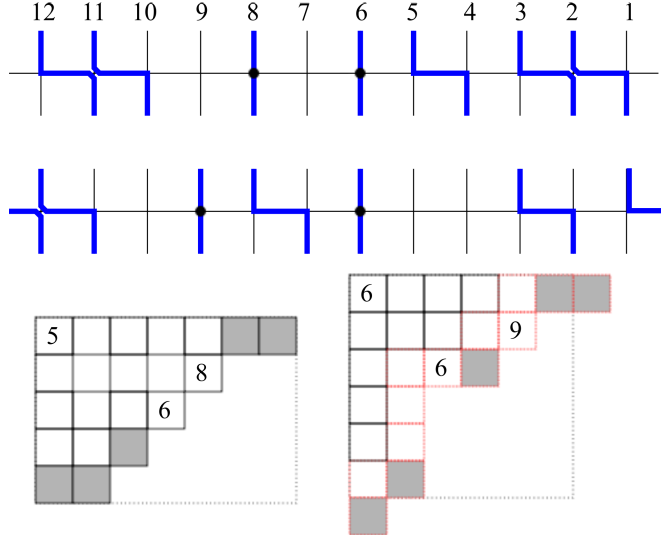}
\end{equation*}%
\caption{Examples of row configurations for the osculating walker model. The
(transposed) Young diagrams on the bottom (top $\rightarrow $ left and
bottom $\rightarrow $\ right) show the corresponding horizontal strips.}
\label{fig:ADopOsc}
\end{figure}

\begin{example}
We now explain (\ref{combE}) on a concrete example. Set again $N=12$ and
refer to Figure \ref{fig:ADopOsc} for two sample row configurations with $%
n=7 $ (top) and $n=6$ (bottom). One uses the same bijection as in the
vicious walker case, i.e. adding boxes in the Young diagrams for each
horizontal path edge, but now for the transposed Young diagrams. The
numbering is now from right to left to facilitate the reading of the
equivariant parameters $T_{j}$ occurring in each row configuration: for each
straight path one collects a factor $(1+uT_{j})$; compare with the vertex
weights in Figure \ref{fig:5vmodels}.
\end{example}

\begin{definition}
A cylindric/toric tableau of shape $\lambda /d/\mu $ is a map $\mathcal{T}%
:\lambda \lbrack d]/\mu \lbrack 0]\rightarrow \mathbb{N}$ such that $%
\mathcal{T}_{i,j}\leq \mathcal{T}_{i,j+1}$ and $\mathcal{T}_{i,j}<\mathcal{T}%
_{i+1,j}$ where $(i,j)$ are the coordinates of the squares between the two
cylindric loops $\lambda \lbrack d]$ and $\mu \lbrack 0]$ in the $\mathbb{Z}%
_{2}$-plane.
\end{definition}

\begin{lemma}
Each toric tableau $\mathcal{T}:\lambda \lbrack d]/\mu \lbrack 0]\rightarrow
\{1,\ldots ,\ell \}$ determines uniquely a sequence of cylindric loops 
\begin{equation}
\Lambda \lbrack d_{1},\ldots ,d_{\ell }=d]=(\lambda ^{(0)}[0]=\mu \lbrack
0],\lambda ^{(1)}[d_{1}],\ldots ,\lambda ^{(\ell )}[d_{\ell }])
\label{torictab2seq}
\end{equation}%
where $d_{i}=d_{1}+d_{2}+\cdots +d_{i-1}$ and $\lambda
^{(i+1)}/(d_{i+1}-d_{i})/\lambda ^{(i)}$ is a horizontal strip. Conversely,
each such sequence of cylindric loops defines uniquely a toric tableau.
\end{lemma}


\begin{figure}[tbp]
\begin{equation*}
\includegraphics[scale=0.7]{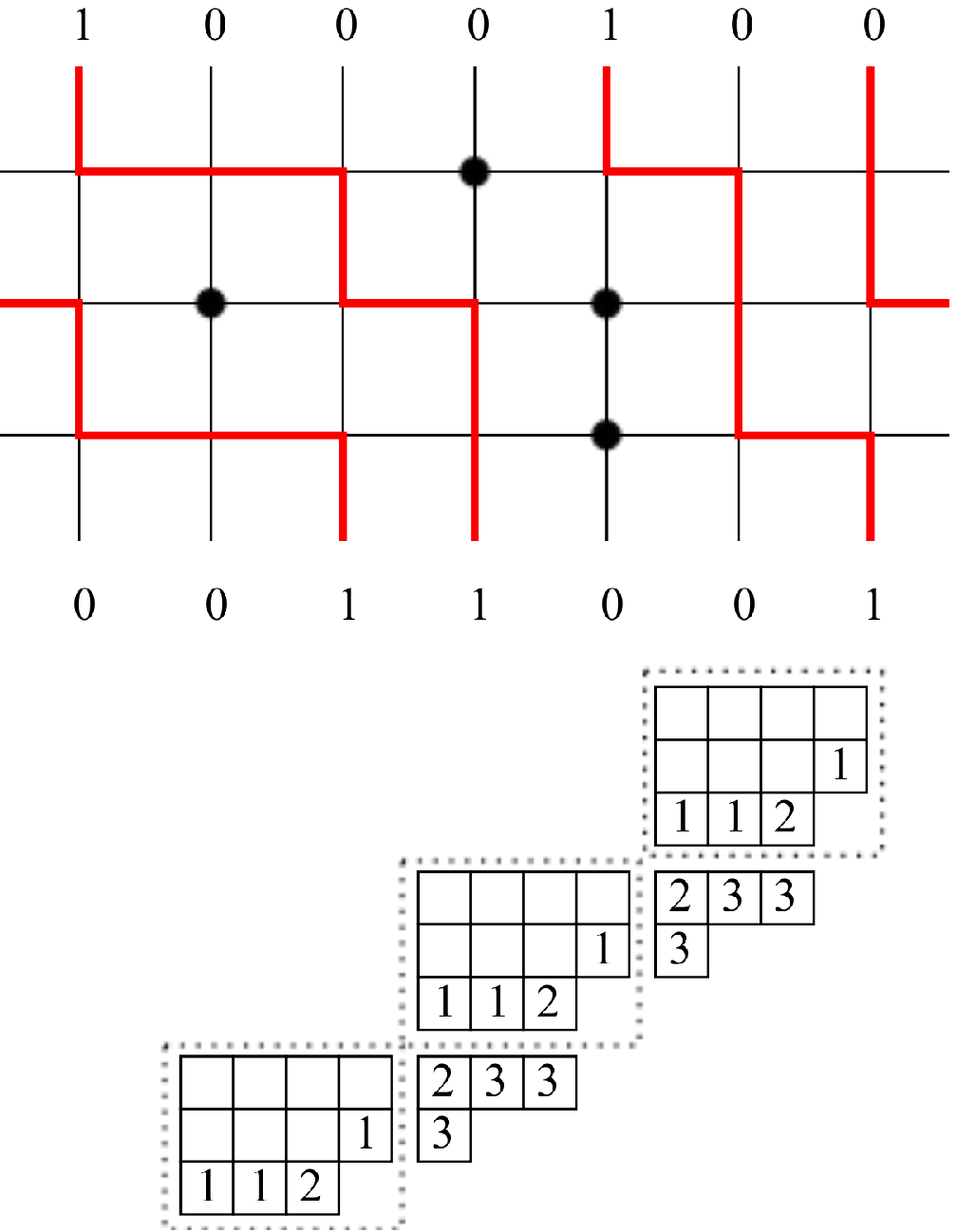}
\end{equation*}%
\caption{The top of the figure shows a vicious walker lattice configuration
for $N=7$ and $n=3$. Black bullets mark the vertex configurations which
contribute factors $(x_{i}-t_{j})$ where $i$ is the row index and $j$ the
column index. Below is the corresponding toric tableau from which the same
factors can be obtained as detailed in formula (\protect\ref{combZ}).}
\label{fig:lattice_config}
\end{figure}

Our two previous combinatorial formulae for the row-to-row transfer matrices
allow us to state the following explicit combinatorial expression for the
partition functions of the vicious and osculating walker models in terms of
toric tableaux.

Given two partitions $\lambda ,\mu $ fix the boundary conditions on the top
and bottom of the square lattice in terms of their corresponding 01-words.
Denote by $Z_{\lambda ,\mu }(x|t)$ the corresponding partition function,
i.e. the weighted sum over lattice path configurations which start at
position $\ell _{i}(\mu )$ and end at position $\ell _{i}(\lambda )$ and let 
$\tilde{Z}_{\lambda ,\mu }(x_{1},\ldots ,x_{n}|t):=(x_{1}\cdots
x_{n})^{n}Z_{\lambda ,\mu }(x_{1}^{-1},\ldots ,x_{n}^{-1}|t)$.

\begin{corollary}[partition functions]
We have the following explicit expression for the partition functions of the
vicious walker model,%
\begin{equation}
\tilde{Z}_{\lambda ,\mu }(x|t)=\sum_{d=0}^{n}q^{d}\sum_{|\mathcal{T}%
|=\lambda /d/\mu }\prod_{1\leq i,j\leq n}\prod_{\ell =\lambda
_{i+1}^{(j)}[d_{j}]+\rho _{i+1}+1}^{\lambda _{i}^{(j-1)}[d_{j-1}]+\rho
_{i}-1}(x_{j}-t_{\ell }),  \label{combZ}
\end{equation}%
where the sum runs over all toric tableaux $\mathcal{T}=(\lambda
^{(0)}[0],\lambda ^{(1)}[d_{1}],\ldots ,\lambda ^{(n)}[d_{n}])$ and $\rho
_{i}=n+1-i$. Here the second product is understood to give one if $\lambda
_{i}^{(j-1)}[d_{j-1}]+\rho _{i}-1<\lambda _{i+1}^{(j)}[d_{j}]+\rho _{i+1}+1$%
. An analogous formulae holds for the osculating walker model exploiting
level-rank duality (\ref{levelrankmom}), $\tilde{Z}_{\lambda ,\mu }^{\prime
}(x|t)=\tilde{Z}_{\lambda ^{\prime },\mu ^{\prime }}(x|-T)$ where $\lambda
^{\prime },\mu ^{\prime }$ are the conjugate partitions of $\lambda ,\mu $.
\end{corollary}

\begin{proof}
A direct consequence of the previous results (\ref{combH}), (\ref{combE})
for the row-to-row transfer matrices and the preceding lemma which allows
one to decompose each toric tableau into a sequence of toric horizontal
strips. The interval $[\lambda _{i}^{(j-1)}[d_{i}]+\rho _{i}-1,\lambda
_{i+1}^{(j)}[d_{i+1}]+\rho _{i+1}+1]$ gives precisely the content of those
squares in the tableau $\mathcal{T}$ which are at the bottom of a column
which does not intersect the $j$th horizontal strip.
\end{proof}

\begin{example}
We explain (\ref{combZ}) on a simple example. Let $n=3$, $k=4$ and take the
01-words $\varepsilon (\mu )=1000101,~\varepsilon (\lambda )=0011001$ as
start and end configurations for vicious walkers with $\mu =(4,3,0)$ and $%
\lambda =(4,2,2)$. Consider the lattice configuration and the corresponding
toric tableau shown in Figure \ref{fig:lattice_config}. As we can see from
the picture there is a horizontal path edge in the second row only, so we
have $d_{0}=d_{1}=0$ and $d_{2}=d_{3}=1$. The sequence of cylindric loops $%
\lambda ^{(0)}[d_{0}]=\mu \lbrack 0],\lambda ^{(1)}[d_{1}],\lambda
^{(2)}[d_{2}],\lambda ^{(3)}[d_{3}]=\lambda \lbrack 1],$ corresponding to
the toric tableau and the associated integers $\ell $ appearing in the
second product of (\ref{combZ}) are detailed in the table below,%
\begin{equation*}
\begin{tabular}{c||c|c|c|c}
$j$ & $0$ & $1$ & $2$ & $3$ \\ \hline
$d_{j}$ & $0$ & $0$ & $1$ & $1$ \\ \hline
$\lambda ^{(j)}[d_{j}]$ & $\ldots ,4,3,0,\ldots $ & $\ldots ,4,4,2,\ldots $
& $\ldots ,5,4,3,\ldots $ & $\ldots ,7,6,3,\ldots $ \\ \hline
$\ell $ & - & $4$ & $5,2$ & $5$%
\end{tabular}%
\end{equation*}%
For instance, we find for $i=1,2,3$ and $j=2$ that $\lambda
^{(2)}[d_{2}]+\rho +1=(\ldots ,5+4,4+3,3+2,1+1,\ldots )$ and $\lambda
^{(1)}[d_{1}]+\rho -1=(\ldots ,4+2,4+1,2+0,\ldots )$. According to formula (%
\ref{combZ}) the second product therefore runs over $\ell =5,2$ for $i=2,3$.
There is no contribution for $i=1$, since $\lambda _{2}^{(2)}[d_{2}]+\rho
_{2}+1=7$ and $\lambda _{1}^{(1)}[d_{1}]+\rho _{1}-1=6$. Similarly, one
computes the other contributions and values of $\ell $. The resulting factor
is%
\begin{equation*}
(x_{1}-t_{4})(x_{2}-t_{5})(x_{2}-t_{2})(x_{3}-t_{5})
\end{equation*}%
which coincides with the weight assigned to the vicious walker lattice
configuration according to (\ref{partition_function}) using the weights in
Fig \ref{fig:5vmodels}.
\end{example}

\subsection{Factorial powers and quantum Pieri rules}

In this section we relate the vicious and osculating walker models to the
equivariant quantum cohomology ring. The row partition functions or transfer
matrices play the central role and should be thought of as noncommutative
analogues of the generating functions for the complete symmetric and
elementary symmetric functions.

Prompted by the previous combinatorial formulae (\ref{combH}), (\ref{combE})
for the row-to-row transfer matrices we now expand them into so-called \emph{%
factorial powers} with respect to the equivariant parameters. On each
subspace $\mathcal{V}_{n}\subset \mathcal{V}$ define a family of operators $%
\{\tilde{H}_{r,n}\}_{r=0}^{k}$ via the following expansion of the transfer
matrix in terms of the factorial powers $(x|T)^{p}:=(x-T_{1})(x-T_{2})\cdots
(x-T_{p})$,%
\begin{equation}
x^{k}H(x^{-1})|_{\mathcal{V}_{n}}=\sum_{r=0}^{k}x^{k-r}H_{r}|_{\mathcal{V}%
_{n}}=\sum_{r=0}^{k}(x|T)^{k-r}\tilde{H}_{r,n}\;.  \label{facH}
\end{equation}%
Note that it follows from the definition of $H(x)$ that $H(x)|_{\mathcal{V}%
_{n}}$ is at most of degree $k$ in $u$, since each allowed row configuration
in $\mathcal{V}_{n}$\ contains at most $k$ vertices with weights $(1-t_{j}u)$
or $u$; see Figure \ref{fig:5vmodels}. Hence the above expansion is
well-defined.

\begin{lemma}
We have the following identities%
\begin{equation}
H_{r}|_{\mathcal{V}_{n}}=\sum_{j=0}^{r}(-1)^{j}e_{j}(T_{1},\ldots ,T_{k-r+j})%
\tilde{H}_{r-j,n}  \label{H2facH}
\end{equation}%
and%
\begin{equation}
\tilde{H}_{r,n}=\sum_{j=0}^{r}\det (e_{1-a+b}(T_{1},\ldots
,T_{k-r+b}))_{1\leq a,b\leq j}~H_{r-j}|_{\mathcal{V}_{n}}  \label{facH2H}
\end{equation}%
where we set $\tilde{H}_{0,n}=1|_{\mathcal{V}_{n}}$ and $e_{r}$ are the
elementary symmetric polynomials. Acting with $\tilde{H}_{r,n}$ on the empty
partition $\emptyset $ in $\mathcal{V}_{n}$, i.e. the vector 
\begin{equation}
|\emptyset \rangle =~\underset{n}{\underbrace{v_{1}\otimes \cdots \otimes
v_{1}}}\otimes \underset{k}{\underbrace{v_{0}\otimes \cdots \otimes v_{0}}}%
\in V^{\otimes N}~,
\end{equation}%
we find that%
\begin{equation}
\tilde{H}_{r,n}|\emptyset \rangle =|(r,0,\ldots ,0)\rangle \;.
\label{PieriId}
\end{equation}
\end{lemma}

\begin{proof}
A straightforward computation. Using that%
\begin{equation*}
(u|T)^{k-r}=\sum_{j=0}^{k-r}(-1)^{j}e_{j}(T_{1},\ldots ,T_{k-r})u^{k-r-j}
\end{equation*}%
we infer that the coefficient of $u^{k-r}$ in $u^{k}H(u^{-1})$ gives the
first identity. This is a linear system of equations with a lower triangular
matrix and its solutions give the second identity. Employing%
\begin{eqnarray}
\tilde{H}_{r,n} &=&H_{r}|_{\mathcal{V}_{n}}-%
\sum_{j=1}^{r}(-1)^{j}e_{j}(T_{1},\ldots ,T_{k-r+j})\tilde{H}_{r-j,n} \\
&=&H_{r}|_{\mathcal{V}_{n}}-\sum_{j=1}^{r}(-1)^{j}\sum_{n+r+1-j\leq
i_{1}<\cdots <i_{j}\leq N}t_{i_{1}}\cdots t_{i_{j}}\tilde{H}_{r-j,n}
\end{eqnarray}%
one then easily deduces the action on the empty partition.
\end{proof}

Explicitly, we have for the first few elements,%
\begin{eqnarray*}
\tilde{H}_{1,n} &=&H_{1}|_{\mathcal{V}_{n}}+T_{1}+\cdots +T_{k} \\
\tilde{H}_{2,n} &=&H_{2}|_{\mathcal{V}_{n}}+\sum_{j=1}^{k-1}T_{j}~H_{1}|_{%
\mathcal{V}_{n}}+\sum_{i=1}^{k-1}T_{i}\sum_{j=1}^{k}T_{j}-\sum_{1\leq
i<j\leq k}T_{i}T_{j} \\
&&\vdots
\end{eqnarray*}%
The result (\ref{PieriId}) is a needed compatibility condition as we wish to
realise the expansion of the product $s_{\lambda }\ast s_{\mu }$ of two
Schubert classes in $QH_{T}^{\ast }(\limfunc{Gr}_{n,N})$ by acting with an
appropriate operator $\tilde{S}_{\lambda ,n}$ on the vector $|\mu \rangle $.
Starting with $\lambda =(r)$ and setting $\tilde{S}_{\lambda ,n}=\tilde{H}%
_{r,n}$ we need to ensure that $\tilde{S}_{\lambda ,n}|\emptyset \rangle
=|\lambda \rangle $ as $s_{\emptyset }$ is the unit in $QH_{T}^{\ast }(%
\limfunc{Gr}_{n,N})$.

\begin{corollary}[equivariant quantum Pieri-Chevalley rule]
The action of $\tilde{H}_{1,n}$ on a basis vector $|\lambda \rangle
=v_{\varepsilon _{1}}\otimes \cdots \otimes v_{\varepsilon _{N}}\in \mathcal{%
V}_{n},$ where $\boldsymbol{\varepsilon }(\lambda )=\varepsilon _{1}\cdots
\varepsilon _{N}$ with $\varepsilon _{i}=0,1$, yields 
\begin{equation}
\tilde{H}_{1,n}|\lambda \rangle =\sum_{\mu -\lambda =(1)}|\mu \rangle
+\left( \sum_{i=1}^{n}T_{k+i-\lambda _{i}}-\sum_{j=k+1}^{N}T_{j}\right)
|\lambda \rangle +q|\lambda ^{-}\rangle \;,  \label{qPieri}
\end{equation}%
where $\lambda ^{-}$ is the partition obtained from $\lambda $ by removing a
boundary rim hook of length $N-1$. If such a rim hook cannot be removed then
this term is missing from the formula.
\end{corollary}

\begin{remark}
Using a result from Mihalcea \cite[Cor 7.1]{MihalceaAIM} the last formula
together with the introduction of an appropriate grading suffices to show
that the commutative ring generated by the endomorphisms $\{\tilde{H}%
_{r,n}\}_{r=1}^{k}$ is canonically isomorphic to $QH_{T}^{\ast }(\limfunc{Gr}%
_{n,N})$. However, here we shall not make use of this fact, but instead will
identify both rings in the last section via their idempotents and, thus,
give an alternative derivation of Mihalcea's result.
\end{remark}

\begin{proof}
This is immediate as 
\begin{equation*}
H_{1}|\lambda \rangle =\sum_{\mu -\lambda =(1)}|\mu \rangle +\left(
\sum_{i=1}^{n}t_{n+1-i+\lambda _{i}}-\sum_{j=1}^{N}t_{j}\right) |\lambda
\rangle +q|\lambda ^{-}\rangle \;,
\end{equation*}%
and $\tilde{H}_{1,n}=H_{1}|\mathcal{V}_{n}+t_{n+1}+\cdots +t_{N}$.
\end{proof}

\begin{example}
We explain the action of the operators (\ref{facH2H}) on a simple example,
setting $N=4$ and $n=2$. Let us for simplicity momentarily drop the $n$%
-dependence in the notation of $\tilde{H}$. Then%
\begin{equation*}
\tilde{H}%
_{2}=H_{2}+t_{4}H_{1}+t_{4}^{2}+t_{3}t_{4}-t_{3}t_{4}=H_{2}+T_{1}H_{1}+T_{1}^{2}
\end{equation*}%
Figure \ref{fig:qPieri} shows the various lattice row configurations of the
vicious walker model when acting with $H_{2}$ and $H_{1}$ on the basis
vector $|2,1\rangle $. Collecting terms we find that%
\begin{equation*}
\tilde{H}_{2}|2,1\rangle =(T_{1}-T_{4})|2,2\rangle
+(T_{1}-T_{2})(T_{1}-T_{4})|2,1\rangle +q|1,0\rangle
+q(T_{1}-T_{2})|0,0\rangle \;.
\end{equation*}%
In a similar manner one can compute the action of $\tilde{H}_{2}$ on any
other basis vector in $\{|0,0\rangle ,|1,0\rangle ,|2,0\rangle ,|1,1\rangle
,|2,1\rangle ,|2,2\rangle \}$. Identifying each basis vector $|\lambda
\rangle $ with a Schubert class $\sigma _{\lambda }$, the coefficients in
the above expansion formula for $\tilde{H}_{2}|2,1\rangle $ match the
coefficients in the product expansion of $\sigma _{(2)}\ast \sigma _{(2,1)}$
in the quantum cohomology ring $QH_{T}^{\ast }(\limfunc{Gr}_{2,4})$; see the
multiplication table in \cite[8.2]{Mihalcea06}. We state the precise ring
isomorphism in the last Section.
\end{example}


\begin{figure}[tbp]
\begin{equation*}
\includegraphics[scale=0.4]{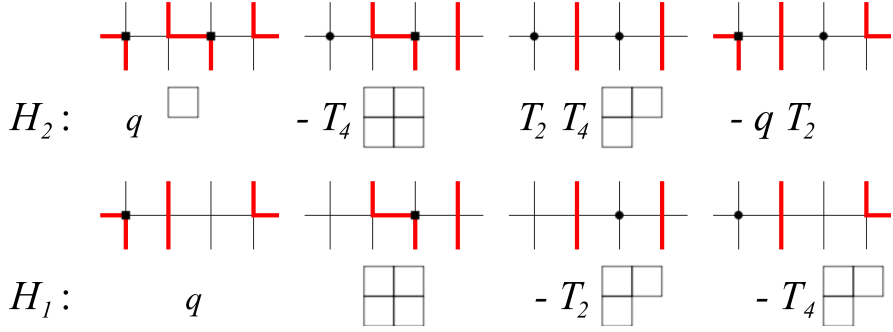}
\end{equation*}%
\caption{The action of the transfer matrix coefficients $H_{2},H_{1}$ on the
partition $(2,1)$ in terms of lattice row configurations for $N=4$ and $n=2$%
. }
\label{fig:qPieri}
\end{figure}

As in the case of the vicious walker model we also introduce for osculating
walkers a different set of operators in terms of the transfer matrix by
expanding the latter in factorial powers,%
\begin{equation}
x^{n}E(x^{-1})|_{\mathcal{V}_{n}}=\sum_{r=0}^{n}x^{n-r}E_{r}|_{\mathcal{V}%
_{n}}=\sum_{r=0}^{n}(x|-t)^{n-r}\tilde{E}_{r,n}\;.  \label{facE}
\end{equation}%
Again one checks by consulting the allowed vertex configurations in Figure %
\ref{fig:5vmodels} that $E(x)|_{\mathcal{V}_{n}}$ is at most of degree $n$,
whence the above expansion is well-defined.

\begin{lemma}
We now have the identities,%
\begin{gather}
E_{r}|_{\mathcal{V}_{n}}=\sum_{j=0}^{r}e_{j}(t_{1},\ldots ,t_{n-r+j})\tilde{E%
}_{r-j,n}  \label{E2facE} \\
\tilde{E}_{r,n}=\sum_{j=0}^{r}(-1)^{j}\det (e_{1-a+b}(t_{1},\ldots
,t_{n-r+b}))_{1\leq a,b\leq j}~E_{r-j}|_{\mathcal{V}_{n}}\;.  \label{facE2E}
\end{gather}%
And we find when acting on the empty partition that%
\begin{equation}
\tilde{E}_{r,n}|\emptyset \rangle =|1^{r}\rangle \;.
\end{equation}
\end{lemma}

\begin{proof}
A direct consequence of level-rank duality (\ref{levelrankmom}) and our
previous result for the $H_{r}$'s.
\end{proof}

\section{The Bethe ansatz}

We now relate the eigenvalue problem of the transfer matrices to the
equivariant quantum cohomology ring. Starting point is a particular guess or
ansatz for the algebraic form of the eigenvectors which we relate to
factorial Schur functions. The eigenvectors will be shown to be identical
with the idempotents of the quantum cohomology ring; see \cite{VicOsc} for
the non-equivariant case.

\subsection{Factorial Schur functions}

To keep this article self-contained we collect several known facts about
factorial Schur functions which we will use repeatedly in what follows.
Given a partition $\lambda =(\lambda _{1},\ldots ,\lambda _{n})$ and some
commuting indeterminates $x=(x_{1},\ldots ,x_{n})$ recall the definition of
the factorial Schur function (see e.g. \cite[Chap I.3, Ex. 20]{Macdonald}
and references therein),%
\begin{equation}
s_{\lambda }(x|a):=\frac{\det [(x_{j}|a)^{\lambda _{i}+n-i}]_{1\leq i,j\leq
n}}{\det [(x_{j}|a)^{n-i}]_{1\leq i,j\leq n}}~,\qquad
(x_{j}|a)^{r}:=\prod_{i=1}^{r}(x_{j}-a_{i}),  \label{facSchur}
\end{equation}%
where $a=(a_{i})_{i\in \mathbb{Z}}$ is an infinite sequence which we choose
to be $a_{i}=t_{i}=T_{N+1-i}$ for $1\leq i\leq N$ and $a_{i}=0$ else. For
convenience we abuse notation and write $s_{\lambda }(x|t)$ for $s_{\lambda
}(x|a)$ under this specialisation. Define $e_{r}(x|a)=s_{(1^{r})}(x|a)$ and $%
h_{r}(x|a)=s_{(r)}(x|a)$ to be the factorial elementary and complete
factorial Schur functions. Note that for $t_{1}=\cdots =t_{N}=0$ one
recovers the ordinary Schur functions $s_{\lambda }(x)=s_{\lambda }(x|0)$.
One can express (\ref{facSchur}) in terms of the latter according to the
expansion \cite[Eqn (6.18)]{Macdonald}%
\begin{equation}
s_{\lambda }(x|a)=s_{\lambda }(x)+\sum_{\mu \varsubsetneq \lambda
}(-1)^{|\lambda /\mu |}\det (e_{\lambda _{i}-\mu _{j}-i+j}(a_{1},\ldots
,a_{n+\lambda _{i}-i}))_{1\leq i,j\leq n}~s_{\mu }(x)\;.  \label{facs2s}
\end{equation}%
We also adopt the notation from \emph{loc. cit.} for the shift operator $%
\tau $ acting on factorial Schur functions via $\tau s_{\lambda
}(x|a):=s_{\lambda }(x|\tau a)$ with $(\tau a)_{i}=a_{i+1}$. Using the
latter one has the following generalisation of the Jacobi-Trudi and N\"{a}%
gelsbach-Kostka determinant formulae,%
\begin{equation}
s_{\lambda }(x|a)=\det (h_{\lambda _{i}-i+j}(x|\tau ^{1-j}a))_{1\leq i,j\leq
n}=\det (e_{\lambda _{i}^{\prime }-i+j}(x|\tau ^{j-1}a))_{1\leq i,j\leq n}
\label{facs2det}
\end{equation}

We also need the tableau-definition of the factorial Schur function. Fix $%
\lambda $, then a semi-standard tableau $\mathcal{T}:Y(\lambda )\rightarrow 
\mathbb{N}$ is a filling of the Young diagram $Y$ of $\lambda $ with
positive integers such that the numbers are weakly increasing in each row
and strictly increasing in each column. Then one has the following
definition as a sum over semi-standard tableaux, 
\begin{equation}
s_{\lambda }(x|a)=\sum_{\mathcal{T}}\prod_{(i,j)\in Y(\lambda )}(x_{\mathcal{%
T}(i,j)}-a_{\mathcal{T}(i,j)+j-i})\;.  \label{facSchurtab}
\end{equation}%
Finally there exists the following generalisation of Cauchy's identity (see
e.g. \cite[Eqn (6.17)]{Macdonald}),%
\begin{equation}
\prod_{i=1}^{n}\prod_{j=1}^{k}(x_{i}+y_{j})=\sum_{\lambda }s_{\lambda
}(x|t)s_{(\lambda ^{\vee })^{\prime }}(y|-t)\;.  \label{Cauchy}
\end{equation}%
Below we will generalise all these formulae to a noncommutative setting by
replacing the elementary and complete factorial Schur functions with the
operator coefficients in the expansions (\ref{facE}), (\ref{facH}).

\subsection{Bethe vectors}

For $y=(y_{1},\ldots ,y_{n})$ some indeterminates introduce the so-called
off-shell Bethe vector in $\mathcal{V}_{n}$%
\begin{equation}
|y_{1},\ldots ,y_{n}\rangle =(y_{1}\cdots y_{n})^{N}\hat{B}%
(y_{n}^{-1}|t)\cdots \hat{B}(y_{1}^{-1}|t)v_{0}\otimes \cdots \otimes v_{0},
\label{Bethev}
\end{equation}%
where $\hat{B}(y_{i}|t)=B(y_{i}|t)y_{i}^{\hat{n}}$ with $\hat{n}%
=\sum_{j=1}^{N}\sigma _{j}^{+}\sigma _{j}^{-}$. Similarly, we define for a $%
k $-tuple $z=(z_{1},\ldots ,z_{k})$ its counterpart under level-rank duality
in $\mathcal{V}_{n}$ as%
\begin{equation}
|z_{1},\ldots ,z_{k}\rangle =\hat{C}^{\prime }(z_{k}^{-1}|t)\cdots \hat{C}%
^{\prime }(z_{1}^{-1}|t)v_{1}\otimes \cdots \otimes v_{1},
\label{dualBethev}
\end{equation}%
where $\tilde{C}^{\prime }(z_{i}|t)=C^{\prime }(z_{i}|t)z_{i}^{-\hat{n}}$.
The reason for introducing the operators $\hat{B}$ and $\hat{C}^{\prime }$
is that the latter commute $\hat{B}(y_{i})\hat{B}(y_{j})=\hat{B}(y_{j})\hat{B%
}(y_{i})$, $\hat{C}^{\prime }(z_{i})\hat{C}^{\prime }(z_{j})=\hat{C}^{\prime
}(z_{j})\hat{C}^{\prime }(z_{i})$ while the $B,C^{\prime }$-operators do
not. Hence, we can conclude that the vectors (\ref{Bethev}) and (\ref%
{dualBethev}) are symmetric in the $y$'s and $z$'s.

We now identify the coefficients of the Bethe vectors with factorial Schur
functions. This identification is a special case of the discussion in \cite%
{Bumpetal}; here we state explicit bijections between our vicious and
osculating walker configurations and semi-standard tableaux which are not
contained in \emph{loc. cit. }Moreover, the formulation in terms of the
Yang-Baxter algebra operators is new.


\begin{figure}[tbp]
\begin{equation*}
\includegraphics[scale=0.5]{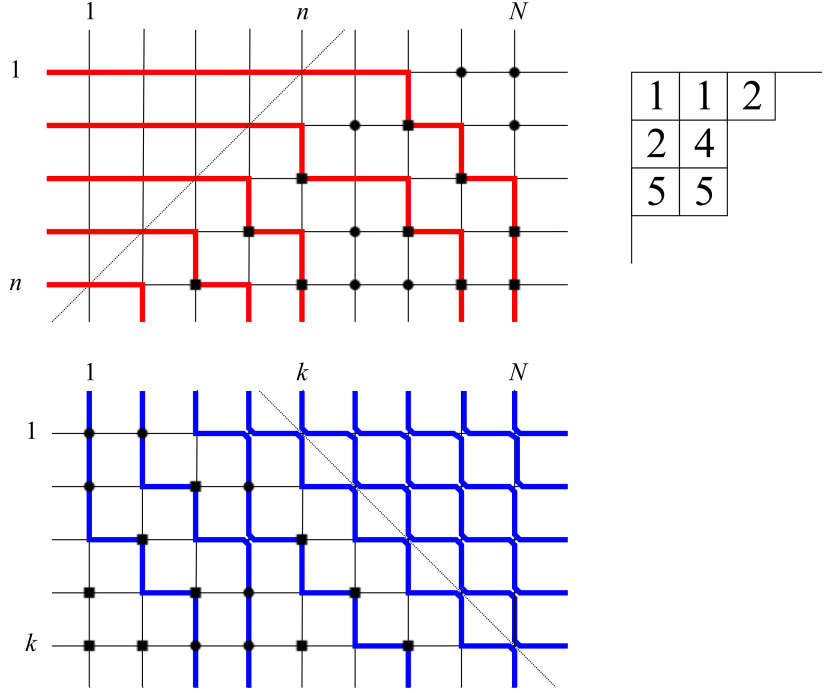}
\end{equation*}%
\caption{Two sample lattice configurations for the $B$ (top) and $C$%
-operator (bottom) of the vicious and osculating model, respectively. Shown
on the right is the tableau obtained under the bijection explained in the
text.}
\label{fig:BCop}
\end{figure}

\begin{proposition}[\protect\cite{Bumpetal}]
Let $\lambda ^{\vee }=(k-\lambda _{n},\ldots ,k-\lambda _{1})$. Then we have
the identities%
\begin{eqnarray}
|y_{1},\ldots ,y_{n}\rangle &=&\sum_{\lambda \in (n,k)}s_{\lambda ^{\vee
}}(y_{1},\ldots ,y_{n}|T)|\lambda \rangle ,  \label{BethefacSchur} \\
|z_{1},\ldots ,z_{k}\rangle &=&\sum_{\lambda \in (n,k)}s_{(\lambda ^{\vee
})^{\prime }}(z_{1},\ldots ,z_{k}|-t)|\lambda \rangle ,
\label{dualBethefacSchur}
\end{eqnarray}%
where we have set $T_{i}=t_{N+1-i}$ as before.
\end{proposition}

\begin{proof}
We only sketch the proof; compare with \cite{Bumpetal}. Consider the matrix
elements 
\begin{eqnarray*}
&&(y_{1}\cdots y_{n})^{N}\langle \lambda |B(y_{n}^{-1})\cdots
B(y_{1}^{-1})|0\rangle \\
&&(z_{1}\cdots z_{k})^{N}\langle \lambda |C^{\prime }(z_{k}^{-1})\cdots
C^{\prime }(z_{1}^{-1})|N\rangle
\end{eqnarray*}%
with $|0\rangle :=v_{0}\otimes \cdots \otimes v_{0}$ and $|N\rangle
:=v_{1}\otimes \cdots \otimes v_{1}$. Each can be identified with a weighted
sum over path configurations with certain fixed boundary conditions; see
Figure \ref{fig:BCop} for examples with $N=9$, $n=5$ (top) and $N=9,$ $k=5$
(bottom). Each vertex with a bullet contributes a factor $(x_{i}-T_{N+1-j})$
for the vicious walker model and a factor $(x_{i}+t_{j})$ for the osculating
model. Here$\ i,j$ are the lattice row and column numbers where the vertex
occurs. Each vertex with a square contributes a factor $x_{i}$. To obtain
the tableau on the right we make use of the following bijections which hold
true in general, but it is instructive to verify them on the example shown
in Figure \ref{fig:BCop}.

\begin{itemize}
\item \emph{Vicious walkers}. Starting from the top, place for each vertex
labelled with a bullet in lattice row $i$ a box labelled with $i$ in the $j$%
th row of the Young tableau where $j$ is the total number of paths crossing
the row to the left of the vertex. The resulting tableau has shape $\lambda
^{\vee }$.

\item \emph{Osculating walkers}. Consider the leftmost path and write in the
first column (counting from left to right) of the Young diagram of $(\lambda
^{\vee })^{\prime }$ the lattice row numbers where a vertex with a bullet
occurs. Then do the same for the next path writing the lattice row numbers
now in the second column etc. If there are no vertices with a bullet leave
the column empty.
\end{itemize}

Set $\rho =(n-1,\ldots ,1,0)$, $\rho ^{\prime }=(k-1,\ldots ,1,0)$ and
denote by $w_{0},w_{0}^{\prime }$ are the longest elements in the symmetric
groups $\mathbb{S}_{n},\mathbb{S}_{k}$. Then using the above bijections we
arrive at the identities%
\begin{eqnarray*}
&&\langle \lambda |B(y_{n}^{-1})\cdots B(y_{1}^{-1})|0\rangle =\frac{%
w_{0}(y)^{\rho }}{(y_{1}\cdots y_{n})^{N}}\sum_{\mathcal{T}}\prod_{\langle
i,j\rangle \in \lambda ^{\vee }}(x_{\mathcal{T}(i,j)}-T_{\mathcal{T}%
(i,j)+j-i}), \\
&&\langle \lambda |C^{\prime }(z_{k}^{-1})\cdots C^{\prime
}(z_{1}^{-1})|N\rangle =\frac{w_{0}^{\prime }(z)^{\rho ^{\prime }}}{%
(z_{1}\cdots z_{k})^{N}}\sum_{\mathcal{T}}\prod_{\langle i,j\rangle \in
(\lambda ^{\vee })^{\prime }}(x_{\mathcal{T}(i,j)}+t_{\mathcal{T}(i,j)+j-i}),
\end{eqnarray*}%
where the sum in the first identity runs over all semistandard tableau $%
\mathcal{T}$ of shape $\lambda ^{\vee }$ and in the second over all
semistandard tableau $\mathcal{T}$ of shape $(\lambda ^{\vee })^{\prime }$.
The assertion now follows from the identities%
\begin{eqnarray*}
\hat{B}(y_{n}^{-1})\cdots \hat{B}(y_{1}^{-1}) &=&w_{0}(y)^{-\rho
}B(y_{n}^{-1})\cdots B(y_{1}^{-1})(y_{1}\cdots y_{n})^{-\hat{n}} \\
\hat{C}^{\prime }(z_{k}^{-1})\cdots \hat{C}^{\prime }(z_{1}^{-1})
&=&w_{0}^{\prime }(z)^{-\rho ^{\prime }}C^{\prime }(z_{k}^{-1})\cdots
C^{\prime }(z_{1}^{-1})(z_{1}\cdots z_{n})^{\hat{n}}\;.
\end{eqnarray*}
\end{proof}

\begin{remark}
Note that the bijection between lattice configurations for the $B,C^{\prime
} $-operators and tableaux used in \cite{Bumpetal} is different from the one
used in the case of the transfer matrices (\ref{combH}), (\ref{combE});
compare also with \cite{VicOsc}. These different bijections arise because in
the non-equivariant case, $t_{j}=0$ for all $j=1,\ldots ,N$, one can employ
the (generally valid) Schur function identity%
\begin{equation*}
s_{\lambda ^{\vee }}(x_{1}^{-1},\ldots ,x_{n}^{-1})=\frac{s_{\lambda
}(x_{1},\ldots ,x_{n})}{s_{(k^{n})}(x_{1},\ldots ,x_{n})}
\end{equation*}%
which is linked with the so-called curious duality in the non-equivariant
quantum cohomology ring \cite{Postnikov}. This relation is no longer valid
for factorial Schur functions and, hence, does not apply to the equivariant
case.
\end{remark}

\begin{corollary}
We have the following identity for the factorial Schur functions under a
permutation of the equivariant parameters,%
\begin{equation}
s_{\lambda }(x|T)=s_{\lambda }(x|\ldots ,T_{N+1-j},T_{N-j},\ldots
)+(T_{N-j}-T_{N+1-j})s_{\delta _{j}^{\vee }\lambda }(x|T)
\label{braidfacschur}
\end{equation}%
for all $j=1,\ldots ,N-1$. Here we set $s_{\delta _{j}^{\vee }\lambda
}(x|T)=0$ if $\delta _{j}^{\vee }|\lambda \rangle =0$.
\end{corollary}

\begin{proof}
Inserting the identity $1=\hat{r}_{j}\hat{r}_{j}^{-1}$ into the matrix
element $\langle \lambda |B(y_{n}^{-1}|t)\cdots B(y_{1}^{-1}|t)|0\rangle $
we obtain from (\ref{mom_ybe}), 
\begin{equation*}
\langle \lambda |B(y_{n}^{-1}|t)\cdots B(y_{1}^{-1}|t)|0\rangle =\langle
\lambda |\hat{r}_{j}B(y_{n}^{-1}|s_{j}t)\cdots B(y_{1}^{-1}|s_{j}t)\hat{r}%
_{j}^{-1}|0\rangle \;.
\end{equation*}%
Exploiting the previous result (\ref{BethefacSchur}) and the explicit action
of the braid matrix (\ref{braidaction}) the assertion follows.
\end{proof}

\subsection{The Bethe ansatz equations}

We call the Bethe vectors (\ref{Bethev}) \textquotedblleft
on-shell\textquotedblright\ if the indeterminates $\{y_{i}\}_{i=1}^{n}$ --
called \emph{Bethe roots} -- are solutions to the following set of Bethe
ansatz equations,%
\begin{equation}
\prod_{j=1}^{N}(y_{i}-t_{j})+(-1)^{n}q=0,\qquad i=1,\ldots ,n\;.  \label{BAE}
\end{equation}%
Similarly, the dual Bethe vectors (\ref{dualBethev}) are on-shell provided
that%
\begin{equation}
\prod_{j=1}^{N}(z_{i}+T_{j})+(-1)^{k}q=0,\qquad i=1,\ldots ,k\;.
\label{dualBAE}
\end{equation}%
We now wish to discuss properties of the solutions of the above equations.
Let $\mathbb{F}:=\mathbb{C}\{\!\{t_{1},\ldots ,t_{N}\}\!\}$ be the
algebraically closed field of Puiseux series in the equivariant parameters
and -- assuming that $q^{\pm 1/N}$ exists -- set $\mathbb{F}_{q}=\mathbb{C}%
[q^{\pm 1/N}]\widehat{\otimes }\mathbb{F}$ to be the completed tensor
product. Note that by a simple rescaling $y_{i}\rightarrow q^{\frac{1}{N}%
}y_{i},$ $z_{i}\rightarrow q^{\frac{1}{N}}z_{i}$ and $t_{j}\rightarrow q^{%
\frac{1}{N}}t_{j}$ we can eliminate $q$ from the above equations and we
therefore often will set $q=1$ temporarily to work with $\mathbb{F}$ only.
Denote by $Y=Y(t_{1},\ldots ,t_{N})\subset \mathbb{F}^{N}$ the set of $N$
solutions to (\ref{BAE}) with $q=1$ and by $Y_q$ the solutions in $\mathbb{F}%
_q^{N}$. The following observations are immediate.

\begin{lemma}
\ \newline

\begin{itemize}
\item[(i)] Permutation invariance: for any $w \in \mathbb{S}_{N}$ we have $%
Y_q(w t)=Y_q(t)$.

\item[(ii)] Level-rank duality: if $y_{i}\in Y_{q}(w_{0}t)$ then $%
z_{i}=-y_{i}$ is a solution of (\ref{dualBAE}).

\item[(iii)] $\mathbb{Z}_{N}$-covariance: let $\eta ^{N}=1$ then $\eta
Y(t)=Y(\eta ^{-1}t)$. In particular, if we set $t_{j}=\eta ^{j}$ then $\eta
Y(t_{1},\ldots ,t_{N})=Y(t_{N},t_{1},t_{2},\ldots ,t_{N-1})$.
\end{itemize}
\end{lemma}

Note that property (i) does not mean that each individual solution $y_{i}$
stays invariant under permutations, especially we emphasise that the
factorial Schur functions in the expansions (\ref{BethefacSchur}), (\ref%
{dualBethefacSchur}) of the Bethe vectors are \emph{not }symmetric in the
equivariant parameters $t$. (ii) reflects that there are two alternative
ways of writing the same eigenvector using (\ref{Bethev}) and (\ref%
{dualBethev}). (iii) is related to the transformation property (\ref{RotHE})
under the map (\ref{Omega}).

For our purposes, it will be convenient to parametrise $n$-tuples $%
y=(y_{1},\ldots ,y_{n})$ of solutions in $Y_{q}$ in terms of 01-words or
partitions $\alpha \in \Pi _{n,k}$. For this purpose we fix a numbering of
the $N$ solutions in $Y_{q}$ by exploiting that the Bethe roots are
explicitly known for $q=0$.

Let $y_{j}=y_{j}(q)$ be the solution which maps on $y_{j}(0)=t_{j}$ 
when setting $q=0$ with $j=1,\ldots ,N$. We then simply write $y_{\lambda }$
for the $n$-tuple $(y_{\ell _{1}},\ldots ,y_{\ell _{n}})$ where $\ell
_{i}(\lambda )=\lambda _{n+1-i}+i$ and $k\geq \lambda _{1}\geq \cdots
\lambda _{n}\geq 0$ as before. Property (ii) in the above Lemma states that
we only need to consider (\ref{BAE}) and fixes the convention for numbering
the solutions of (\ref{dualBAE}): given a particular solution $y_{i}(q)$ of (%
\ref{BAE}) act with the longest permutation $w_{0}$ and then multiply it
with minus one.

\subsection{Spectral decomposition of the transfer matrices}

\begin{theorem}
The on-shell Bethe vectors and their dual vectors both form orthogonal
eigenbases of the transfer matrix $H$ in each subspace $\mathcal{V}%
_{n}\otimes \mathbb{F}$ with eigenvalue equations%
\begin{gather}
H(x_{i}|t)|y_{1},\ldots ,y_{n}\rangle =\left(
\prod_{j=1}^{N}(1-x_{i}t_{j})+(-1)^{n}qx_{i}^{N}\right) \prod_{l=1}^{n}\frac{%
1}{1-x_{i}~y_{l}}~|y_{1},\ldots ,y_{n}\rangle ,  \label{specH} \\
H(x_{i}|t)|z_{1},\ldots ,z_{k}\rangle =\prod_{l=1}^{k}\left(
1+x_{i}~z_{l}\right) ~|z_{1},\ldots ,z_{k}\rangle \;.  \label{dualspecH}
\end{gather}%
In contrast, the operator $E$ satisfies the identities%
\begin{gather}
E(x_{i}|t)|y_{1},\ldots ,y_{n}\rangle =\prod_{l=1}^{n}\left(
1+x_{i}~y_{l}\right) ~|y_{1},\ldots ,y_{n}\rangle  \label{specE} \\
E(x_{i}|t)|z_{1},\ldots ,z_{k}\rangle =\left(
\prod_{j=1}^{N}(1+x_{i}T_{j})+(-1)^{k}qx_{i}^{N}\right) \prod_{l=1}^{k}\frac{%
1}{1-x_{i}~z_{l}}~|z_{1},\ldots ,z_{k}\rangle \;.  \label{dualspecE}
\end{gather}
\end{theorem}

\begin{proof}
The proof is analogous to the one used in the non-equivariant case \cite%
{VicOsc}. One employs the following commutation relations of the row
Yang-Baxter algebra (\ref{row_yba}) encoded in the Yang-Baxter equation (\ref%
{mom_ybe}),%
\begin{equation*}
\left\{ 
\begin{array}{c}
A(x)B(y)=\frac{x}{x-y}B(x)A(y)-\frac{y}{x-y}B(y)A(x) \\ 
D(y)B(x)=\frac{y}{x-y}B(y)D(x)-\frac{x}{x-y}B(x)D(y)%
\end{array}%
\right.
\end{equation*}%
and the analogous relations for the osculating walker model which are easily
obtained from level-rank duality (\ref{levelrankmom}). The latter in
conjunction with (\ref{QQ}) then allow one to derive all four eigenvalue
expressions as well as the equations (\ref{BAE}) as necessary conditions for
(\ref{Bethev}), (\ref{dualBethev}) to be eigenvectors.

There are $\dim \mathcal{V}_{n}=\binom{N}{n}=\binom{N}{k}$ solutions to (\ref%
{BAE}) since each of the equations is of degree $N$ and $\mathbb{F}$ is
algebraically closed.

One can show that both matrices $H,E$ are normal and that the eigenvalues
separate points, whence the eigenvectors (\ref{Bethev}), (\ref{dualBethev})
have to be orthogonal and, thus, form each an eigenbasis.
\end{proof}

We employ the involution $\Theta $ to obtain transformation properties for
the Bethe vectors.

\begin{lemma}
Let $\lambda ^{\ast }=(\lambda ^{\vee })^{\prime }=(\lambda ^{\prime
})^{\vee }$ where $\lambda ^{\prime }$ is the conjugate partition of $%
\lambda $ and $\lambda ^{\vee }$ the complement of $\lambda $ in the $%
n\times k$ bounding box. Given a solution $y(t)=(y_{1},\ldots ,y_{n})$ of (%
\ref{BAE}) there exists a solution $z(-T)=(z_{1},\ldots ,z_{n})$ of (\ref%
{dualBAE}) with $k\rightarrow n$ such that%
\begin{equation}
\Theta |y_{1},\ldots ,y_{n}\rangle =|z_{1},\ldots ,z_{n}\rangle \in \mathcal{%
V}_{k}\;.  \label{LR_duality_Bethev}
\end{equation}%
That is, under level-rank duality the $y$-Bethe vectors in $\mathcal{V}_{n}$
are mapped on to the $z$-Bethe vectors in $\mathcal{V}_{k}$.
\end{lemma}

\begin{proof}
Employing level-rank duality (\ref{levelrankmom}) of the monodromy matrices
we have and 
\begin{eqnarray*}
\Theta |y_{1},\ldots ,y_{n}\rangle &=&C^{\prime }(y_{1}|-T)\cdots C^{\prime
}(y_{n}|-T)|N\rangle =\sum_{\lambda \in (k,n)}s_{\lambda ^{\ast
}}(y|T)|\lambda \rangle \\
\Theta |z_{1},\ldots ,z_{k}\rangle &=&B(z_{1}|-T)\cdots B(z_{k}|-T)|N\rangle
=\sum_{\lambda \in (n,k)}s_{\lambda ^{\ast }}(z|-t)|\lambda \rangle \;.
\end{eqnarray*}%
According to our previous results $\Theta |y_{1},\ldots ,y_{n}\rangle $ must
be an eigenvector of both $E(u|-T)$ and $H(u|-T)$ in $\mathcal{V}_{k}$ and
therefore proportional to $|z_{1},\ldots ,z_{n}\rangle $ for some solution $%
z=z(y)$ of (\ref{dualBAE}) since $\Theta H(u|t)\Theta =E(u|-T)$ and the
eigenvalues separate points.
\end{proof}

\begin{corollary}
We have the additional eigenvalue equations 
\begin{eqnarray}
&&\left\{ 
\begin{array}{c}
\tilde{E}_{r,n}|y_{1},\ldots ,y_{n}\rangle =e_{r}(y|t)|y_{1},\ldots
,y_{n}\rangle \\ 
\tilde{E}_{r,n}|z_{1},\ldots ,z_{k}\rangle =h_{r}(z|-T)|z_{1},\ldots
,z_{k}\rangle%
\end{array}%
\right. ,\qquad  \label{facEspec} \\
&&\left\{ 
\begin{array}{c}
\tilde{H}_{r,n}|z_{1},\ldots ,z_{k}\rangle =e_{r}(z|-T)|z_{1},\ldots
,z_{k}\rangle \\ 
\tilde{H}_{r,n}|y_{1},\ldots ,y_{n}\rangle =h_{r}(y|t)|y_{1},\ldots
,y_{n}\rangle%
\end{array}%
\right. ,\qquad  \label{facHspec}
\end{eqnarray}%
for the operators (\ref{facH}), (\ref{facE}) defined in terms of factorial
powers.
\end{corollary}

\begin{proof}
The first eigenvalue equation in (\ref{facEspec}) is a direct consequence of
the definition (\ref{facE}) and the identity 
\begin{equation*}
\prod_{i=1}^{n}(u-x_{i})=\sum_{r=0}^{n}(-1)^{r}e_{r}(x|t)(u|t)^{n-r}
\end{equation*}%
or equivalently%
\begin{equation*}
\prod_{i=1}^{n}(u+x_{i})=\sum_{r=0}^{n}e_{r}(x|t)(u|-t)^{n-r}
\end{equation*}%
where $e_{r}(x|t)$ is the factorial elementary symmetric function.

The functional equation (\ref{QQ}) in terms of the expansions (\ref{facH})
and (\ref{facE}) reads%
\begin{equation*}
\left( \sum_{r=0}^{n}(-1)^{r}\tilde{E}_{r,n}(u|t)^{n-r}\right) \left(
\sum_{r=0}^{k}\tilde{H}_{r,n}(u|T)^{k-r}\right) =(u|t)^{N}+(-1)^{n}q
\end{equation*}%
Noting the trivial identities%
\begin{equation*}
(u|T)^{r}=(u|\tau ^{N-r}t)^{r},\qquad (u|T)^{k-r}=(u|\tau
^{n+r}t)^{k-r},\qquad (u|t)^{n}(u|T)^{k}=(u|t)^{N}
\end{equation*}%
the latter becomes%
\begin{equation}
\sum_{r=0}^{n+k}\sum_{s=0}^{r}(-1)^{s}\tilde{E}_{s,n}\tilde{H}%
_{r-s,n}(u|t)^{n-s}(u|\tau ^{n-s+r}t)^{k-r+s}=(u|t)^{N}+(-1)^{n}q\;.
\end{equation}%
We have $(u|t)^{a+b}=(u|t)^{a}(u|\tau ^{a}t)^{b}$ and the generally valid
identity \cite{Macdonald}%
\begin{equation}
\sum_{s=0}^{r}(-1)^{s}e_{s}(x|\tau
^{r-1}a)h_{r-s}(x|a)=\sum_{s=0}^{r}(-1)^{s}e_{s}(x|a)h_{r-s}(x|\tau
^{1-r}a)=0
\end{equation}%
from which we deduce the first eigenvalue equation in (\ref{facHspec}) by
comparing coefficients. The remaining identities then follow from level rank
duality, $\Theta H(u|t)\Theta =E(u|-T)$ and $\Theta |y_{1},\ldots
,y_{n}\rangle =|z_{1},\ldots ,z_{n}\rangle $.
\end{proof}

The alternative descriptions of the spectrum of the same operators in terms
of Bethe roots and dual Bethe roots implies the following identities.

\begin{lemma}
Given a solution $y_{\alpha }$ of (\ref{BAE}) there exists a solution $%
z_{\alpha }=-y_{\alpha ^{\ast }}$ of (\ref{dualBAE}) such that 
\begin{eqnarray}
e_{r}(y_{1},\ldots ,y_{n}) &=&\sum_{s=0}^{r}e_{s}(T_{1},\ldots
,T_{N})h_{r-s}(z_{1},\ldots ,z_{k})  \label{baeid1} \\
e_{r}(y_{1},\ldots ,y_{n}|t) &=&h_{r}(z_{1},\ldots ,z_{k}|-T),\qquad
r=1,\ldots ,n  \label{baeid2}
\end{eqnarray}%
and, hence, one deduces the identities%
\begin{eqnarray}
h_{r}(z_{1},\ldots ,z_{k}) &=&\sum_{s=0}^{r}(-1)^{s}h_{s}(T_{1},\ldots
,T_{N})e_{r-s}(y_{1},\ldots ,y_{n})  \label{LRdualityspec} \\
s_{\lambda ^{\prime }}(y_{1},\ldots ,y_{n}|t) &=&s_{\lambda }(z_{1},\ldots
,z_{k}|-T)  \label{LRdualityFacSchur}
\end{eqnarray}
\end{lemma}

\begin{proof}
The first two statements are a direct consequence of the Bethe ansatz
computations and the alternative eigenvalue formulae for (\ref{Bethev}) and (%
\ref{dualBethev}). The identities (\ref{baeid1}), (\ref{baeid2}) form a
linear system of equations which can be easily solved to yield (\ref%
{LRdualityspec}). To obtain (\ref{LRdualityFacSchur}) from (\ref{baeid2})
one uses the N\"{a}gelsbach-Kostka and Jacobi-Trudi formula (\ref{facs2det})
for factorial Schur functions.
\end{proof}

\begin{lemma}
The Bethe ansatz equations (\ref{BAE})\ are equivalent to the identities%
\begin{gather}
\sum_{r=0}^{j}(-1)^{r}e_{r}(t_{1},\ldots ,t_{N})h_{j-r}(y_{1},\ldots
,y_{n})=0,\quad j=k+1,\ldots ,N-1  \notag \\
\sum_{r=0}^{N}(-1)^{r}e_{r}(t_{1},\ldots ,t_{N})h_{N-r}(y_{1},\ldots
,y_{n})=(-1)^{n-1}q  \label{M2'}
\end{gather}
\end{lemma}

\begin{proof}
It follows from the allowed vertex configurations in Figure \ref%
{fig:5vmodels} that if $|y_{1},\ldots ,y_{n}\rangle $ is a joint eigenvector
of $H(x|t)$ and $E(x|t)$ then the corresponding eigenvalues%
\begin{eqnarray}
H(x|t)|y_{1},\ldots ,y_{n}\rangle &=&\left( 1+b_{1}x+\cdots
+b_{k}x^{k}\right) |y_{1},\ldots ,y_{n}\rangle  \notag \\
E(-x|t)|y_{1},\ldots ,y_{n}\rangle &=&\left( 1+a_{1}x+\cdots
+a_{n}x^{n}\right) |y_{1},\ldots ,y_{n}\rangle  \label{spec}
\end{eqnarray}%
are at most of degree $k$ and $n$ in $x$, respectively. Together with (\ref%
{specH}) this proves that the equations in (\ref{BAE}) imply (\ref{M2'}).

For the converse implication, assume that (\ref{M2'}) hold true and compute
the residue of the eigenvalue in (\ref{specH}) at $x=y_{i}^{-1}$.
\end{proof}

\subsection{Left eigenvectors \& Poincar\'{e} duality}

For convenience we use the Dirac notation and denote the dual basis of the
ket-vectors $\{|\lambda \rangle \}_{\lambda \in (n,k)}$ by the bra-vectors $%
\{\langle \lambda |\}_{\lambda \in (n,k)}\subset V_{n}^{\ast }$.

\begin{proposition}[Poincar\'{e} duality]
The left eigenvectors of the transfer matrices (dual eigenbasis) are given by%
\begin{equation}
\langle y_{\alpha }|~=\sum_{\lambda \in (n,k)}\frac{s_{\lambda }(y_{\alpha
}|t)}{\mathfrak{e}(y_{\alpha })}\langle \lambda |,\qquad \alpha \in \Pi
_{n,k},  \label{leftBethev}
\end{equation}%
where 
\begin{equation}
\mathfrak{e}(y_{\alpha })=\prod_{\substack{ i\in I(\alpha )  \\ j\in
I(\alpha ^{\ast })}}(y_{i}-y_{j})\;.  \label{BetheEuler}
\end{equation}%
We shall refer to the isomorphism $\mathcal{V}_{n}^{\mathbb{F}}\rightarrow
\left( \mathcal{V}_{n}^{\mathbb{F}}\right) ^{\vee }:=\mathbb{F}_{q}\otimes
V_{n}^{\ast }$ given by the mapping $|y_{\alpha }\rangle \mapsto \langle
y_{\alpha }|$ for all $\alpha \in \Pi _{n,k}$ as \emph{Poincar\'{e} duality}.
\end{proposition}

\begin{proof}
Recall the involution $\mathcal{P}:\mathcal{V}_{n}^{\mathbb{F}}\rightarrow 
\mathcal{V}_{n}^{\mathbb{F}}$ defined by $\mathcal{P}|\lambda \rangle
=|\lambda ^{\vee }\rangle $. Then one verifies that the transpose of the
matrices 
\begin{equation*}
\boldsymbol{H}(x|t)=(\langle \lambda |H(x|t)|\mu \rangle )_{\lambda ,\mu \in
(n,k)},\qquad \boldsymbol{E}(x|t)=(\langle \lambda |E(x|t)|\mu \rangle
)_{\lambda ,\mu \in (n,k)}
\end{equation*}%
are given by 
\begin{equation*}
\boldsymbol{H}(x|t)^{T}=(\langle \lambda |\mathcal{P}H(x|T)\mathcal{P}|\mu
\rangle )_{\lambda ,\mu }\quad \text{and\quad }\boldsymbol{E}%
(x|t)^{T}=(\langle \lambda |\mathcal{P}E(x|T)\mathcal{P}|\mu \rangle
)_{\lambda ,\mu }~.
\end{equation*}%
Thus, the first assertion, that $\langle y_{\alpha }|$ is a left
eigenvector, follows from the explicit expansion (\ref{Bethev}) and the
previous result (\ref{specH}), (\ref{specE}) that the Bethe vectors are
right eigenvectors. Since the eigenvalues of the transfer matrices separate
points we can conclude that $\langle y_{\alpha }|y_{\beta }\rangle $ must be
proportional to the Kronecker function $\delta _{\alpha \beta }$. Let $%
\mathfrak{e}(y_{\alpha })$ denote the proportionality factor which is
computed as follows:%
\begin{eqnarray*}
\mathfrak{e}(y_{\alpha }) &=&\mathfrak{e}(y_{\alpha })\langle y_{\alpha
}|y_{\alpha }\rangle =\sum_{\lambda \in (n,k)}s_{\lambda }(y_{\alpha
}|t)s_{\lambda ^{\vee }}(y_{\alpha }|T) \\
&=&\sum_{\lambda \in (n,k)}s_{\lambda }(y_{\alpha }|t)s_{(\lambda ^{\vee
})^{\prime }}(z_{\alpha }|-t)=\sum_{\lambda \in (n,k)}s_{\lambda }(y_{\alpha
}|t)s_{\lambda ^{\ast }}(-y_{\alpha ^{\ast }}|-t) \\
&=&\prod_{\substack{ i\in I(\alpha )  \\ j\in I(\alpha ^{\ast })}}%
(y_{i}-y_{j}),
\end{eqnarray*}%
where in the second line we have made use of (\ref{LRdualityFacSchur}) and
in the last step we have used the Cauchy identity for factorial Schur
functions (\ref{Cauchy}).
\end{proof}

\begin{remark}
To motivate our definition of Poincar\'{e} duality recall that this mapping
reduces to the known Poincar\'{e} duality in the non-equivariant setting, $%
t_{j}=0$, and in the classical equivariant setting, $q=0$.
\end{remark}

Since the Bethe vectors (\ref{Bethev}) and (\ref{leftBethev}) form each an
eigenbasis they give rise to a resolution of the identity $\boldsymbol{1}%
=\sum_{\alpha \in (n,k)}|y_{\alpha }\rangle \langle y_{\alpha }|$ which
translates into the following identity for factorial Schur functions when
evaluated at solutions of the Bethe ansatz equations.

\begin{corollary}[orthogonality]
For all $\lambda ,\mu \in (n,k)$ we have the identity%
\begin{equation}
\sum_{\alpha \in (n,k)}\frac{s_{\lambda ^{\vee }}(y_{\alpha }|T)s_{\mu
}(y_{\alpha }|t)}{\mathfrak{e}(y_{\alpha })}=\delta _{\lambda \mu }\;.
\label{res of 1}
\end{equation}
\end{corollary}

\subsection{Bethe vectors and GKM theory}

Consider the extension of the $\mathbb{A}_{N}$ and $\mathbb{\hat{S}}_{N}$%
-actions (\ref{affsymgroup}) on $\Lambda $ to $\mathbb{F}$ and set $\mathcal{%
V}_{n}^{\mathbb{F}}:=\mathcal{V}_{n}\otimes \mathbb{F}$ as before.

\begin{proposition}[GKM conditions]
The $\mathbb{S}_{N}$-action on $\mathcal{V}_{n}$ given by $\{\boldsymbol{s}%
_{j}\}_{j=1}^{N-1}$ in Prop \ref{prop:symmgroupaction} permutes the Bethe
vectors $|y_{\alpha }\rangle $ according to the natural $\mathbb{S}_{N}$%
-action on the 01-words $\alpha $, i.e. $\boldsymbol{s}_{j}|y_{\alpha
}\rangle =|y_{p_{j}\alpha }\rangle $ where $p_{j}$ permutes the $j$th and $%
(j+1)$th letter in $\alpha $. In particular, we have that 
\begin{equation}
s_{\lambda }(y_{\alpha }|t)-s_{j}\cdot s_{\lambda }(y_{p_{j}\alpha
}|t)=(t_{j}-t_{j+1})s_{\delta _{j}\lambda }(y_{\alpha }|t)\,,\quad
j=1,\ldots ,N-1,  \label{GKM}
\end{equation}%
where we set $s_{\delta _{j}\lambda }(y|T)\equiv 0$ if $\delta _{j}|\lambda
\rangle =0$.
\end{proposition}

\begin{remark}
The last proposition states that we can identify in our setting the
coefficients of the vector $|\lambda \rangle $ in the basis of on-shell
Bethe vectors with a localised Schubert class. From (\ref{Bethev}), (\ref%
{leftBethev}) and (\ref{res of 1}) it follows that 
\begin{equation}
|\lambda \rangle =\sum_{\alpha }\frac{s_{\lambda }(y_{\alpha }|t)}{\mathfrak{%
e}(y_{\alpha })}|y_{\alpha }\rangle \;.  \label{basischange}
\end{equation}%
Define $\xi _{\lambda }:\mathbb{S}_{N}/\mathbb{S}_{n}\times \mathbb{S}%
_{k}\rightarrow \mathbb{F}_{q}$ by setting $\alpha \mapsto \xi _{\lambda
}(\alpha ):=s_{\lambda }(y_{\alpha }|t)$ where $\alpha $ fixes uniquely a
minimal length representative of a coset in $\mathbb{S}_{N}/\mathbb{S}%
_{n}\times \mathbb{S}_{k}$. For $q=0$ we have that $y_{\alpha }=t_{\alpha }$
and the equalities (\ref{GKM}) become the familiar GKM conditions \cite{GKM}
which fix a localised Schubert class with values in $\Lambda $. If $q\neq 0$
the solutions of (\ref{BAE}) cease to be polynomial in the equivariant
parameter in general and one therefore has to work over the algebraically
closed field $\mathbb{F}_{q}$ instead.
\end{remark}

\begin{remark}
Note that it suffices to consider the $\mathbb{S}_{N}$-action, since the $%
\mathbb{\hat{S}}_{N}$-actions from Prop \ref{prop:symmgroupaction} and (\ref%
{affsymgroup}) are both of level-0 for $q=1$. That is, the affine simple
Weyl reflection acts via the Weyl reflection associated with the negative
highest root. The case $q\neq 1$ can be recovered by rescaling Bethe roots
and equivariant parameters by the same common factor $q^{-1/N}$ and, hence,
the relations (\ref{GKM}) remain unchanged.
\end{remark}

\begin{proof}
Let $|y_{\alpha }\rangle =|y_{1},\ldots ,y_{n}\rangle $ be a Bethe vector
with $H(x|t)|y_{\alpha }\rangle =h_{\alpha }(x|t)|y_{\alpha }\rangle $. We
infer from (\ref{SactionHE}) that%
\begin{equation*}
s_{j}\hat{r}_{j}H(x|t)|y_{\alpha }\rangle =s_{j}H(x|s_{j}t)\hat{r}%
_{j}|y_{\alpha }\rangle =(s_{j}h_{\alpha }(x|t))s_{j}\hat{r}_{j}|y_{\alpha
}\rangle
\end{equation*}%
and, hence, that $s_{j}\hat{r}_{j}|y\rangle $ is an eigenvector of $H(x|t)$
with eigenvalue $s_{j}h_{\alpha }(x|t)$. Thus, there must exist another
Bethe vector $|y_{\alpha ^{\prime }}\rangle $ with eigenvalue $h_{\alpha
^{\prime }}(x|t)=s_{j}h_{y}(x|t)$. Since the Bethe vectors form an
eigenbasis and the eigenvalues of $H$ separate points we must have that the
braided eigenvector $s_{j}\hat{r}_{j}|y_{\alpha }\rangle $ is proportional
to this second Bethe vector,%
\begin{equation*}
s_{j}\hat{r}_{j}|y_{\alpha }\rangle =\eta _{j}\cdot |y_{\alpha ^{\prime
}}\rangle ,\qquad \eta _{j}\in \mathbb{F}\;.
\end{equation*}%
It follows from (\ref{mom_ybe}) for $j=1,2,\ldots ,N-1$ 
\begin{eqnarray*}
s_{j}\hat{r}_{j}|y_{\alpha }\rangle &=&s_{j}\hat{r}_{j}(y_{1}\cdots
y_{n})^{N}\hat{B}(y_{n}|t)\cdots \hat{B}(y_{1}|t)|0\rangle \\
&=&s_{j}(y_{1}\cdots y_{n})^{N}\hat{B}(y_{n}|s_{j}t)\cdots \hat{B}%
(y_{1}|s_{j}t)|0\rangle =\eta _{j}~|y_{\alpha ^{\prime }}\rangle
\end{eqnarray*}%
Since $s_{j}y_{\alpha }$ is also a solution of (\ref{BAE}) the left hand
side is of the form (\ref{Bethev}) and we can conclude that $\eta _{j}=1$.
To show that $\alpha ^{\prime }=p_{j}\alpha $ we first note that the claim
is correct for $q=0,$ since then $y_{\alpha }=t_{\alpha }$ is a solution as
discussed earlier. Expanding $y_{\alpha }$ at $q=0$ it follows that this
stays true in the vicinity of $q=0$. By similar arguments one shows that $%
s_{j}\mathfrak{e}(y_{\alpha })=\mathfrak{e}(y_{p_{j}\alpha })$.

To derive (\ref{GKM}) we start from the expansion (\ref{basischange}) and
act with $\boldsymbol{s}_{j}$ on both sides. The identities (\ref{GKM}) then
follow from (\ref{rhat}) and (\ref{BethefacSchur}) by comparing coefficients
in the basis of Bethe vectors.
\end{proof}

\begin{corollary}
Each Bethe vector $|y_{\alpha }\rangle $ obeys the qKZ equations (\ref{qKZ})
with respect to the subgroup $\mathbb{S}_{n}\times \mathbb{S}_{k}\subset 
\mathbb{S}_{N}$ where the latter embedding depends on $\alpha $.
\end{corollary}

\begin{proof}
If $j$ is such that $p_{j}\alpha =\alpha $ it follows from Lemma \ref%
{lem:qKZ} that $\partial _{j}|y_{\alpha }\rangle =\delta _{j}^{\vee
}|y_{\alpha }\rangle $ with 
\begin{equation*}
\delta _{j}^{\vee }|y\rangle =\sum_{\lambda \in (n,k)}s_{\lambda ^{\vee
}}(y|T)\delta _{j}^{\vee }|\lambda \rangle =\sum_{\lambda \in
(n,k)}s_{(\delta _{j}\lambda )^{\vee }}(y|T)|\lambda \rangle \;.
\end{equation*}
\end{proof}

\section{Combinatorial construction of $QH_{T}^{\ast }(\limfunc{Gr}_{n,N})$}

Using the results from the Bethe ansatz we now prove that the row-to-row
transfer matrices generate a symmetric Frobenius algebra. We will identify
the latter with the presentation of the equivariant quantum cohomology ring
as Jacobi algebra put forward in \cite{Gepner} and \cite{GP}.

\subsection{Noncommutative factorial Schur functions}

We employ the identity (\ref{facs2s}) for factorial Schur functions to
define for any $\lambda \in \Pi _{n,k}$ the following operator $\tilde{S}%
_{\lambda }^{(n)}:\mathcal{V}_{n}\rightarrow \mathcal{V}_{n}$%
\begin{equation}
\tilde{S}_{\lambda }^{(n)}=\sum_{\mu \subseteq \lambda }(-1)^{|\lambda /\mu
|}\det (e_{\lambda _{i}-\mu _{j}-i+j}(t_{1},\ldots ,t_{n+\lambda
_{i}-i}))_{1\leq i,j\leq n}~S_{\mu }^{(n)},  \label{facS}
\end{equation}%
where $S_{\lambda }^{(n)}:\mathcal{V}_{n}\rightarrow \mathcal{V}_{n}$ is
given by the analogue of the N\"{a}gelsbach-Kostka determinant formula,%
\begin{equation}
S_{\lambda }^{(n)}=\det (E_{\lambda _{i}^{\prime }-i+j})_{1\leq i,j\leq k}\;.
\label{S}
\end{equation}%
Note that both operators are well-defined since the $E_{r}$'s commute
pairwise. We define $\tilde{S}_{\lambda },S_{\lambda }:\mathcal{V}%
\rightarrow \mathcal{V}$ to be the unique operators which restrict to $%
\tilde{S}_{\lambda }^{(n)},S_{\lambda }^{(n)}$ on each $\mathcal{V}_{n}$
with $0<n<N$ and to the identity for $n=0,N$. We will often omit the
superscript if it is clear on which subspace $\tilde{S}_{\lambda
}^{(n)},S_{\lambda }^{(n)}$ are acting.

It might be helpful to the reader to illustrate the definition on a concrete
example.

\begin{example}
We find from formula (\ref{facS}) that the operator $\tilde{S}_{\lambda }$
with $\lambda =(2,1)$ reads explicitly%
\begin{multline*}
\tilde{S}_{2,1}=S_{2,1}-T_{4}S_{2,0}-(T_{2}+T_{3}+T_{4})S_{1,1} \\
+(T_{2}+T_{3}+T_{4})T_{4}S_{1,0}-(T_{2}+T_{3})T_{4}^{2}S_{0,0}\;.
\end{multline*}%
We now demonstrate on a simple example that we can compute the product
expansions $\sigma _{\lambda }\ast \sigma _{\mu }$ in the quantum cohomology
ring by identifying the basis vectors in the expansion of $\tilde{S}%
_{\lambda }|\mu \rangle $ with the respective Schubert classes in $%
QH_{T}^{\ast }(\limfunc{Gr}{}_{2,4})$. For instance, let us compute the
matrix elements $\langle 2,1|\tilde{S}_{\lambda }|\lambda \rangle ,\langle
2,2|\tilde{S}_{\lambda }|\lambda \rangle $ with $\lambda =(2,1)$ as before.
Using that $S_{2,1}=E_{2}E_{1}-E_{3}$, $S_{2,0}=E_{1}^{2}-E_{2}$, $%
S_{1,1}=E_{2}$, $S_{1,0}=E_{1}$ and $S_{0,0}=1$ we find from (\ref{combE})
the matrix elements%
\begin{equation*}
\begin{tabular}{c||c|c|c|c}
$\mu $ & $2,1$ & $2,0$ & $1,1$ & $1,0$ \\ \hline
$\langle 2,1|S_{\mu }|\lambda \rangle $ & $T_{1}T_{3}(T_{1}+T_{3})$ & $%
2T_{1}T_{3}+T_{1}^{2}+T_{3}^{2}-T_{1}T_{3}$ & $T_{1}T_{3}$ & $T_{1}+T_{3}$
\\ \hline
$\langle 2,2|S_{\mu }|\lambda \rangle $ & $T_{1}(T_{1}+T_{2})+T_{1}T_{3}$ & $%
T_{1}+T_{2}+T_{3}$ & $T_{1}$ & $1$%
\end{tabular}%
\ .
\end{equation*}%
N.B. that $E_{3}$ does not give any contribution, since we are acting on a
01-word with two 1-letters. Inserting the results into the above expansion
formula for $\tilde{S}_{\lambda }$ and cancelling terms we obtain%
\begin{equation}
\langle 2,1|\tilde{S}_{\lambda }|\lambda \rangle =T_{12}T_{14}T_{34}\quad \;%
\text{and\ }\quad \langle 2,2|S_{\mu }|\lambda \rangle =T_{14}^{2}
\end{equation}%
which are the equivariant Gromov-Witten invariants $C_{\lambda \lambda
}^{(2,1),d=0}(T)$ and $C_{\lambda \lambda }^{(2,2),d=0}(T)$ for $%
QH_{T}^{\ast }(\limfunc{Gr}{}_{2,4})$. The latter can - for example - be
computed using Knutson-Tao puzzles; see Figure \ref{fig:KTpuzzles} and \cite%
{KnutsonTao} for their definition. In general, $d\neq 0$, Knutson-Tao
puzzles cannot be used to compute the product in $QH_{T}^{\ast }(\limfunc{Gr}%
_{n,N})$. To see a genuine quantum product we observe that $\tilde{S}_{(1)}=%
\tilde{E}_{1}=\tilde{H}_{1}$ in which case we derived previously the
equivariant quantum Pieri rule (\ref{qPieri}).
\end{example}


\begin{figure}[tbp]
\begin{equation*}
\includegraphics[scale=1]{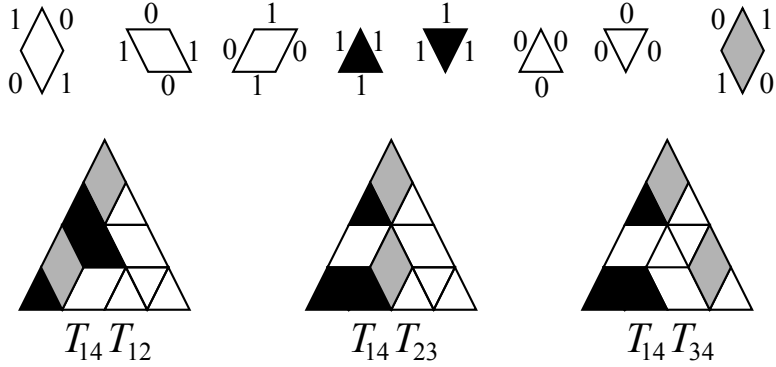}
\end{equation*}%
\caption{Equivariant Knutson-Tao puzzles \protect\cite{KnutsonTao} can be
used to compute Gromov-Witten invariants for $q=0$, i.e. the product in $%
H_{T}^{\ast }(\limfunc{Gr}_{n,N})$. On top the allowed puzzle tiles are
shown, on the bottom three puzzles whose weighted sum gives the
Gromov-Witten invariant $%
C_{(2,1),(2,1)}^{(2,2),d=0}(T)=T_{14}(T_{12}+T_{23}+T_{34})=T_{14}^{2}$.}
\label{fig:KTpuzzles}
\end{figure}

The following lemma will be instrumental in proving that the matrix elements
of (\ref{facS}) are Gromov-Witten invariants.

\begin{lemma}
\label{lem:facS}Consider the (unique) extension of $\tilde{S}_{\lambda }$ to 
$\mathcal{V}_{n}^{\mathbb{F}}:=\mathcal{V}_{n}\otimes \mathbb{F}$. (i) The
Bethe vectors (\ref{Bethev}) are eigenvectors of $\tilde{S}_{\lambda }$ and
on each $\mathcal{V}_{n}^{\mathbb{F}}$ we have the eigenvalue equation $%
\tilde{S}_{\lambda }|y_{\alpha }\rangle =s_{\lambda }(y_{\alpha
}|t)|y_{\alpha }\rangle $ where $s_{\lambda }(x|t)$ is the factorial Schur
function. (ii) Let $|\emptyset \rangle =v_{1}\otimes \cdots \otimes
v_{1}\otimes v_{0}\otimes \cdots \otimes v_{0}\in \mathcal{V}_{n}$ be the
unique basis vector which corresponds to the empty partition. Then $\tilde{S}%
_{\lambda }|\emptyset \rangle =|\lambda \rangle $ for all $\lambda \in \Pi
_{n,k}$.
\end{lemma}

\begin{proof}
Statement (i) is a direct consequence of (\ref{specE}) and that the
expansion (\ref{facS}) applies to the factorial Schur function; see \cite[%
Eqn (6.18)]{Macdonald}. To prove (ii) recall the expansion of $|\emptyset
\rangle $ and $|\lambda \rangle $ into Bethe vectors,%
\begin{equation*}
\tilde{S}_{\lambda }|\emptyset \rangle =\sum_{\alpha \in \Pi _{n,k}}%
\mathfrak{e}(y_{\alpha })^{-1}\tilde{S}_{\lambda }|y_{\alpha }\rangle
=\sum_{\alpha \in \Pi _{n,k}}\mathfrak{e}(y_{\alpha })^{-1}s_{\lambda
}(y_{\alpha }|t)|y_{\alpha }\rangle =|\lambda \rangle \;.
\end{equation*}
\end{proof}

Using this last result we can derive two alternative determinant formulae
for (\ref{facS}). Recall the expansions (\ref{facH}) and (\ref{facE}) of the
transfer matrices on the subspace $\mathcal{V}_{n}$. We now introduce
so-called \emph{shifted factorial powers}: let $\tau $ denote the
shift-operator and set $(u|\tau ^{j}T)^{p}:=(u-T_{j+1})(u-T_{j+2})\cdots
(u-T_{j+p})$ for $j=1,2,\ldots ,n$. Then we define $\tau ^{j}\tilde{H}_{r,n}:%
\mathcal{V}_{n}\rightarrow \mathcal{V}_{n}$ to be the coefficient of $%
(u|\tau ^{j}T)^{k-r}$ in the expansion of $u^{k}H(u^{-1})|_{\mathcal{V}_{n}}$
into shifted factorial powers analogous to (\ref{facH}). Similarly, let $%
\tau ^{j}\tilde{E}_{r,n}:\mathcal{V}_{n}\rightarrow \mathcal{V}_{n}$ be the
coefficient of $(u|-\tau ^{j}t)^{n-r}$ when expanding $u^{n}E(u^{-1})|_{%
\mathcal{V}_{n}}$ into the shifted factorial powers $(u|-\tau
^{j}t)^{p}:=(u+t_{j+1})(u+t_{j+2})\cdots (u+t_{j+p})$ for $j=1,2,\ldots ,k$;
compare with (\ref{facE}).

\begin{lemma}[Jacobi-Trudi and N\"{a}gelsbach-Kostka formulae]
We have the following identities for the operator (\ref{facS}) on each
subspace $\mathcal{V}_{n}$,%
\begin{equation}
\tilde{S}_{\lambda }^{(n)}=\det \left( \tau ^{j-1}\tilde{H}_{\lambda
_{i}-i+j,n}\right) _{1\leq i,j\leq n}=\det \left( \tau ^{j-1}\tilde{E}%
_{\lambda _{i}^{\prime }-i+j,n}\right) _{1\leq i,j\leq k}~,
\label{det_formulae}
\end{equation}%
where $\lambda ^{\prime }$ denotes the conjugate partition of $\lambda $.
\end{lemma}

\begin{proof}
Consider first the extension of $\tilde{S}_{\lambda }$ to $\mathcal{V}_{n}^{%
\mathbb{F}}$. Because the Bethe vectors form an eigenbasis we can use the
previous lemma and the known analogous determinant formulae (\ref{facs2det})%
\begin{equation*}
s_{\lambda }(x|t)=\det \left( h_{\lambda _{i}-i+j}(x|\tau ^{1-j}t)\right)
_{1\leq i,j\leq n}=\det \left( e_{\lambda _{i}^{\prime }-i+j}(x|\tau
^{j-1}t)\right) _{1\leq i,j\leq k}
\end{equation*}%
for factorial Schur functions to deduce the asserted identities on $\mathcal{%
V}_{n}^{\mathbb{F}}$. But since the operators $\tau ^{j}\tilde{H}_{r,n},\tau
^{j}\tilde{E}_{r,n}$ in both identities restrict to maps $\mathcal{V}%
_{n}\rightarrow \mathcal{V}_{n}$ according to their definition the assertion
follows. N.B. that the shifted powers $\tau ^{j}\tilde{H}_{r,n}$ have been
defined with respect to the $T_{i}=t_{N+1-i}$ parameters, hence the shift in
the first identity in (\ref{det_formulae}) is positive.
\end{proof}

We demonstrate the alternative formulae by computing the Gromov-Witten
invariants from the last example using (\ref{det_formulae}).

\begin{example}
We return to our previous example of $QH_{T}^{\ast }(\limfunc{Gr}_{2,4})$
with $k=n=2$. To calculate the matrix elements $\langle 2,1|\tilde{S}%
_{\lambda }|\lambda \rangle ,~\langle 2,2|\tilde{S}_{\lambda }|\lambda
\rangle $ with $\lambda =(2,1)$ via (\ref{det_formulae}), we first observe
that $\tilde{S}_{\lambda }=\tilde{H}_{2}~\tau \tilde{H}_{1}=\tilde{E}%
_{2}~\tau \tilde{E}_{1}$, where we have neglected the contribution from $%
\tilde{H}_{3},\tilde{E}_{3}$ as both identically vanish on $\mathcal{V}_{2}$
for $N=4$. Compute from (\ref{combH}), (\ref{combE}) the matrix elements $%
\langle 2,2|\tilde{H}(x_{1})\tilde{H}(x_{2})|\lambda \rangle ,~\langle 2,2|%
\tilde{E}(x_{1})\tilde{E}(x_{2})|\lambda \rangle $ of the transfer matrices
- see Figure \ref{fig:gr24ex} - and then expand the resulting polynomials
into (shifted) factorial powers $(x_{1}|\tau T)^{\alpha
_{1}}(x_{2}|T)^{\alpha _{2}}$ and $(x_{1}|-\tau t)^{\alpha
_{1}}(x_{2}|-t)^{\alpha _{2}}$. One finds for the vicious walker model that%
\begin{eqnarray*}
\langle 2,2|\tilde{H}(x_{1})\tilde{H}(x_{2})|\lambda \rangle
&=&(x_{1}-T_{4})(x_{2}-T_{4})(x_{2}-T_{3})+(x_{1}-T_{4})(x_{1}-T_{2})(x_{2}-T_{4})
\\
&=&(T_{13}T_{12}+T_{13}T_{24})(x_{1}-T_{2})+~\ldots
~+T_{14}T_{34}(x_{1}-T_{2})+~\ldots
\end{eqnarray*}%
where we have only listed the terms corresponding to the shifted factorial
power $\alpha =(k-1,k-2)=(1,0)$. Hence, we have $\langle 2,2|\tilde{H}%
_{2}~\tau \tilde{H}_{1}|\lambda \rangle =T_{14}^{2}$ as required. Similarly
one finds for the osculating walker model that%
\begin{equation*}
\langle 2,2|\tilde{E}(x_{1})\tilde{E}(x_{2})|\lambda \rangle
=(x_{1}+t_{4})(x_{2}+t_{3})(x_{2}+t_{4})+(x_{1}+t_{4})(x_{1}+t_{2})(x_{2}+t_{4})
\end{equation*}%
and, thus, by a simple substitution, we recover the previous result, $%
\langle 2,2|\tilde{E}_{2}~\tau \tilde{E}_{1}|\lambda \rangle
=t_{41}^{2}=T_{14}^{2}$. Note that in general both computations will be
different if $n\neq k$ and $\lambda \neq \lambda ^{\prime }$.
\end{example}


\begin{figure}[tbp]
\begin{equation*}
\includegraphics[scale=0.7]{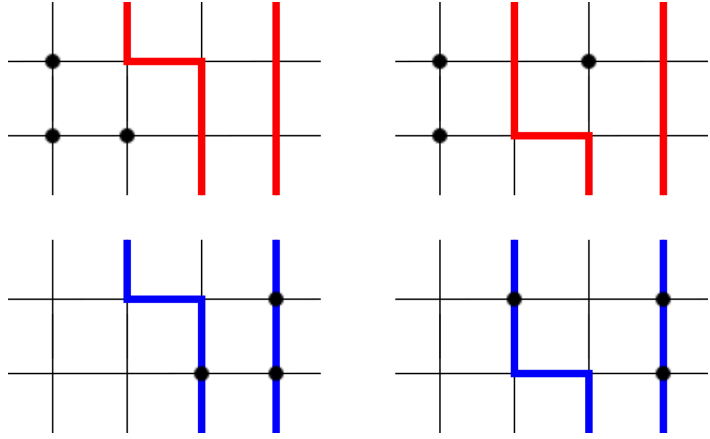}
\end{equation*}%
\caption{Shown on top are the two possible vicious paths which connect the
01-words for $\protect\mu =(2,1)$ and $\protect\lambda =(2,2)$. On the
bottom we see that osculating paths are the same. However, both pairs yield
different weighted sums.}
\label{fig:gr24ex}
\end{figure}

Motivated by the previous examples we now define a purely combinatorial
product on $\mathcal{V}_{n}$. The resulting ring will be identified in the
next section with the equivariant quantum cohomology of the Grassmannian.

\begin{theorem}
Define a product on $\mathcal{V}_{n}$ by setting 
\begin{equation}
|\lambda \rangle \circledast |\mu \rangle :=\tilde{S}_{\lambda }|\mu \rangle
\;.  \label{comb_product}
\end{equation}%
Then $(\mathcal{V}_{n},\circledast )$ is a commutative ring over $\Lambda
\lbrack q]$.
\end{theorem}

\begin{proof}
Consider once more the extension of $\tilde{S}_{\lambda }$ to $\mathcal{V}%
_{n}^{\mathbb{F}}:=\mathcal{V}_{n}\otimes \mathbb{F}$ which is defined over
the field of Puiseux series in the equivariant parameters. Computing the
action of $\tilde{S}_{\lambda }$ by employing the Bethe vectors we find that%
\begin{eqnarray*}
\tilde{S}_{\lambda }|\mu \rangle &=&\sum_{\alpha \in (n,k)}\frac{s_{\mu
}(y_{\alpha }|t)}{\mathfrak{e}(y_{\alpha })}\tilde{S}_{\lambda }|y_{\alpha
}\rangle =\sum_{\alpha \in (n,k)}\frac{s_{\lambda }(y_{\alpha }|t)s_{\mu
}(y_{\alpha }|t)}{\mathfrak{e}(y_{\alpha })}|y_{\alpha }\rangle \\
&=&\sum_{\nu \in (n,k)}\sum_{\alpha \in (n,k)}\frac{s_{\lambda }(y_{\alpha
}|t)s_{\mu }(y_{\alpha }|t)s_{\nu ^{\vee }}(y_{\alpha }|T)}{\mathfrak{e}%
(y_{\alpha })}~|\nu \rangle \;.
\end{eqnarray*}%
Since the coefficients in the last line are symmetric in $\lambda ,\mu $ the
product is clearly commutative. Recall that it follows from the definition (%
\ref{facS}) and the explicit formula (\ref{EHecke}) together with (\ref%
{piaction}) that the expansion coefficients of $\tilde{S}_{\lambda }|\mu
\rangle $ in the basis $\{|\nu \rangle \}_{\nu \in \Pi _{n,k}}$ must be in $%
\Lambda \lbrack q]$, although this is not obvious from the last equality.

Associativity now follows from $\tilde{S}_{\lambda }\tilde{S}_{\mu }=\tilde{S%
}_{\mu }\tilde{S}_{\lambda }$ which in turn is a consequence of (\ref{facS})
and Cor \ref{integrability}, 
\begin{multline*}
|\lambda \rangle \circledast \left( |\mu \rangle \circledast |\nu \rangle
\right) =\tilde{S}_{\lambda }\tilde{S}_{\mu }|\nu \rangle =\tilde{S}%
_{\lambda }\tilde{S}_{\nu }|\mu \rangle = \\
\tilde{S}_{\nu }\tilde{S}_{\lambda }|\mu \rangle =|\nu \rangle \circledast
\left( |\lambda \rangle \circledast |\mu \rangle \right) =\left( |\lambda
\rangle \circledast |\mu \rangle \right) \circledast |\nu \rangle \;.
\end{multline*}
\end{proof}

An immediate consequence of our proof of the last theorem is the following
expression of the structure constants in terms of the Bethe roots.

\begin{corollary}[residue formula]
The structure constants of the ring $(\mathcal{V}_{n},\circledast )$ are
given in terms of the following residue formula%
\begin{equation}
C_{\lambda \mu }^{\nu ,d}(T):=\langle \nu |\tilde{S}_{\lambda }|\mu \rangle
=\sum_{\alpha \in (n,k)}\frac{s_{\lambda }(y_{\alpha }|t)s_{\mu }(y_{\alpha
}|t)s_{\nu ^{\vee }}(y_{\alpha }|T)}{\mathfrak{e}(y_{\alpha })}\;,
\label{residue_formula}
\end{equation}%
where the sum ranges over all solutions $y_{\alpha }\in \mathbb{F}_{q}^{n}$
of the Bethe ansatz equations (\ref{BAE}).
\end{corollary}

\begin{remark}
Our residue formula is a generalisation of the Bertram-Vafa-Intriligator
formula for Gromov-Witten invariants to the equivariant setting. It holds
also true for $q=0$, where the Bethe roots are explicitly known, $%
y_{i}=T_{i} $.
\end{remark}

Another generalisation from the non-equivariant case is the following
duality; compare with the known duality in $H_{T}^{\ast }(\limfunc{Gr}%
_{n,N}) $.

\begin{corollary}[level-rank duality]
We have the following identity%
\begin{equation}
C_{\lambda \mu }^{\nu ,d}(T)=C_{\lambda ^{\prime }\mu ^{\prime }}^{\nu
^{\prime },d}(-t)  \label{GW_level-rank}
\end{equation}%
where $t=w_{0}T$ as before. In particular, the mapping $\Theta :\mathcal{V}%
_{n}\rightarrow \mathcal{V}_{k}$ with the product in $(\mathcal{V}%
_{k},\circledast ^{\prime })$ defined by replacing $t_{j}\rightarrow
-T_{N+1-j}$ in the definition of $\tilde{S}_{\lambda }$ is a ring
isomorphism.

\begin{proof}
A straightforward computation making use of the identity (\ref%
{LRdualityFacSchur}),%
\begin{eqnarray*}
C_{\lambda \mu }^{\nu ,d}(T) &=&\sum_{\alpha \in (n,k)}\frac{s_{\lambda
}(y_{\alpha }|t)s_{\mu }(y_{\alpha }|t)s_{\nu ^{\vee }}(y_{\alpha }|T)}{%
\mathfrak{e}(y_{\alpha })} \\
&=&\sum_{\alpha \in (n,k)}\frac{s_{\lambda ^{\prime }}(-y_{\alpha ^{\ast
}}|-T)s_{\mu ^{\prime }}(-y_{\alpha ^{\ast }}|-T)s_{(\nu ^{\vee })^{\prime
}}(-y_{\alpha ^{\ast }}|-t)}{\mathfrak{e}(-y_{\alpha ^{\ast }})}=C_{\lambda
^{\prime }\mu ^{\prime }}^{\nu ^{\prime },d}(-t)\;.
\end{eqnarray*}
\end{proof}
\end{corollary}

It is obvious from our discussion that the Bethe ansatz is central to our
discussion and the following result clarifies the role of the Bethe vectors
with regard to the product in quantum cohomology.

\begin{corollary}[idempotents]
The algebra $\mathcal{V}_{n}^{\mathbb{F}}=\mathcal{V}_{n}\otimes \mathbb{F}$
is semisimple and canonically isomorphic to the generalised matrix algebra
defined via $|y_{\alpha }\rangle \circledast |y_{\beta }\rangle =\delta
_{\alpha \beta }\mathfrak{e}(y_{\alpha })|y_{\alpha }\rangle $, i.e. the
on-shell Bethe vectors (\ref{Bethev}) yield a complete set of idempotents
for $\mathcal{V}_{n}^{\mathbb{F}}$.
\end{corollary}

\begin{proof}
A straightforward computation using the fact that the Bethe vectors
diagonalise (\ref{facS}),%
\begin{eqnarray*}
|y_{\alpha }\rangle \circledast |y_{\beta }\rangle &=&\sum_{\lambda \in
(n,k)}s_{\lambda ^{\vee }}(y_{\alpha }|T)\tilde{S}_{\lambda }|y_{\beta
}\rangle =\sum_{\lambda \in (n,k)}s_{\lambda ^{\vee }}(y_{\alpha
}|T)s_{\lambda }(y_{\beta }|t)|y_{\beta }\rangle \\
&=&\mathfrak{e}(y_{\beta })|y_{\beta }\rangle \langle y_{\beta }|y_{\alpha
}\rangle =\delta _{\alpha \beta }\mathfrak{e}(y_{\alpha })|y_{\alpha }\rangle
\end{eqnarray*}%
where we have used (\ref{leftBethev}) to arrive at the second line.
\end{proof}

\subsection{Canonical isomorphisms}

We now identify the combinatorial ring from the previous section with $%
QH_{T}^{\ast }(\limfunc{Gr}_{n,N})$. Starting point is the following
presentation of $QH_{T}^{\ast }(\limfunc{Gr}_{n,N})$ as Jacobi algebra. For
type $A$ this result is originally due to Gepner \cite{Gepner} and for
general flag varieties it is contained in \cite{GP}.

Denote by $p_{r}=\sum_{i=1}^{n}x_{i}^{r}$ the power sums in some commuting
indeterminates $x_{1},\ldots ,x_{n}$. Define the so-called fusion potential
by setting 
\begin{equation}
F(x_{1},\ldots ,x_{n};q)=(-1)^{k}qp_{1}+\sum_{r=0}^{N}\frac{%
(-1)^{N-r}p_{N+1-r}}{N+1-r}e_{r}(t_{1},\ldots ,t_{N})\;,  \label{F}
\end{equation}%
where $e_{r}(t_{1},\ldots ,t_{N})$ is once more the elementary symmetric
polynomial of degree $r$ in the equivariant parameters. Note that this
fusion potential $F=F_{n,k}$ for $n>1$ can be written as the sum of fusion
potentials for the case of projective space $n=1$, that is%
\begin{equation*}
F_{n,k}(x_{1},\ldots ,x_{n};q)=\sum_{i=1}^{n}F_{1,N-1}(x_{i};(-1)^{n-1}q)\;.
\end{equation*}%
We have the following known result; c.f. \cite{Gepner}, \cite{GP}.

\begin{theorem}[Jacobi algebra]
The algebra $QH_{T}^{\ast }(\limfunc{Gr}_{n,N})\otimes \mathbb{F}_{q}$ is
canonically isomorphic to the Jacobi algebra $\mathfrak{J}=\mathbb{F}%
_{q}[e_{1},\ldots ,e_{n}]/\langle \partial F/\partial e_{1},\ldots ,\partial
F/\partial e_{n}\rangle $.
\end{theorem}

\begin{remark}
It might be instructive to remind the reader about the equivalence of the
equations $\partial F/\partial e_{j}=0$ and the Bethe ansatz equations (\ref%
{BAE}). Recall the following identities%
\begin{equation}
\sum_{r=0}^{j}(-1)^{r}e_{r}h_{j-r}=0\qquad \text{and}\qquad \frac{1}{r}\frac{%
\partial p_{r}}{\partial e_{j}}=(-1)^{j-1}h_{r-j}\;,
\end{equation}%
where $e_{r},h_{r}$ are the elementary and complete symmetric functions. For 
$j=1,2,\ldots ,n$ we then have%
\begin{equation*}
0=\frac{\partial F}{\partial e_{j}}=(-1)^{N+j-1}%
\sum_{r=0}^{N}(-1)^{r}e_{r}(t_{1},\ldots ,t_{N})h_{k+1-r+n-j}+\delta
_{j1}(-1)^{k}q
\end{equation*}%
These are the relations (\ref{M2'}) which are equivalent to the Bethe ansatz
equations (\ref{BAE}).
\end{remark}

\begin{proposition}
The Jacobi algebra $\mathfrak{J}$ has a complete set of idempotents 
\begin{equation}
P_{\alpha }(x)=\frac{1}{\mathfrak{e}(y_{\alpha })}\prod_{i=1}^{n}\prod_{j\in
I(\alpha ^{\ast })}(x_{i}-y_{j}),  \label{idempotent}
\end{equation}%
where $\alpha $ ranges over all partitions in the $n\times k$ bounding box.
\end{proposition}

N.B. the idempotents (\ref{idempotent}) are symmetric in the indeterminates $%
x_{i}$ and, hence, can be expressed as a polynomial in the $e_{i}$%
's\thinspace\ although it is not practical to do so. In contrast the
idempotents are not symmetric in the Bethe roots $(y_{1},\ldots ,y_{N})\in 
\mathbb{F}_{q}$ as a particular subset of $k$ roots is chosen. In fact, a
relabelling of the Bethe roots $y=(y_{1},\ldots ,y_{N})$ yields a
relabelling of the idempotents.

\begin{theorem}
The map $\Phi :\mathcal{V}_{n}^{\mathbb{F}}\rightarrow \mathfrak{J}$ given
by $|\lambda \rangle \mapsto s_{\lambda }(x|t)$ for all $\lambda \in (n,k)$
is an algebra isomorphism. In particular, the renormalised Bethe vectors $%
Y_{\alpha }=\mathfrak{e}(y_{\alpha })^{-1}|y_{\alpha }\rangle $ are mapped
onto the idempotents (\ref{idempotent}) and the matrix elements $%
q^{d}C_{\lambda \mu }^{\nu ,d}(T)=\langle \nu |\tilde{S}_{\lambda }|\mu
\rangle $ with $dN=|\lambda |+|\mu |-|\nu |$ are the equivariant
Gromov-Witten invariants.
\end{theorem}

\begin{proof}
Upon mapping $|\lambda \rangle \mapsto s_{\lambda }(x|t)$ the image of the
Bethe vectors is given by%
\begin{equation*}
|y_{\alpha }\rangle \mapsto \sum_{\lambda \in (n,k)}\frac{s_{\lambda ^{\vee
}}(y_{\alpha }|T)}{\mathfrak{e}(y_{\alpha })}s_{\lambda }(x|t)\;.
\end{equation*}

To see that these are the idempotents given by the expression (\ref%
{idempotent}) we use (\ref{LRdualityFacSchur}) with $z=-y_{\alpha ^{\ast }}$
and the Cauchy identity (\ref{Cauchy}). As the set of idempotents is
complete, $\mathcal{V}_{n}^{\mathbb{F}}=\tbigoplus_{\alpha }\mathbb{F}%
|y_{\alpha }\rangle ,$ this fixes the isomorphism uniquely.
\end{proof}

The following identifies the coefficients (\ref{HHecke}) and (\ref{EHecke})
of the transfer matrices as the Givental-Kim generators in the presentation (%
\ref{GK}) of $QH_{T}^{\ast }(\limfunc{Gr}_{n,N})$.

\begin{corollary}
Let $\mathbb{P}_{n}\subset \limfunc{End}\mathcal{V}_{n}$ be the commutative $%
\Lambda \lbrack q]$-algebra generated by the restriction of the operators $%
\{H_{r}\}_{r=1}^{k}$ and $\{E_{r}\}_{r=1}^{n}$ to the subspace $\mathcal{V}%
_{n}$. Then the map $\varphi :\mathbb{P}_{n}\rightarrow QH_{T}^{\ast }(%
\limfunc{Gr}_{n,N})$ defined by%
\begin{equation}
\varphi (E_{i})=(-1)^{i}a_{i}\qquad \text{and}\qquad \varphi (H_{j})=b_{j}
\end{equation}%
with $i=1,\ldots ,n$, $j=1,\ldots ,k$ is an algebra isomorphism. In
particular, $\tilde{S}_{\lambda }$ as defined in (\ref{facS}) is mapped to
the Schubert class $\sigma _{\lambda }$.
\end{corollary}

\begin{proof}
Recall the identity (\ref{QQ}) on the subspace $\mathcal{V}_{n}$,%
\begin{equation}
\left( \tsum_{i=0}^{n}x^{i}E_{i}\right) \left(
\tsum_{j=0}^{k}x^{j}H_{j}\right) =(-1)^{n}qx^{N}+\prod_{r=1}^{N}(1-xT_{r})~.
\end{equation}%
But expanding this polynomial identity with respect to the variable $x$ we
deduce the defining relations (\ref{Idef}) of $I$ in the presentation (\ref%
{GK}). This result combined with the fact that the eigenvalues (\ref{specH}%
), (\ref{specE}) separate points and that the Bethe vectors form an
eigenbasis then gives the canonical algebra isomorphism, since it shows that
the coordinate ring defined via (\ref{QQ}) is just (\ref{GK}).

Because $\{|\lambda \rangle \}_{\lambda \in \Pi _{n,k}}$ form a basis of $%
\mathcal{V}_{n}$ we can conclude that - according to Lemma \ref{lem:facS}
(ii) - the operators $\{\tilde{S}_{\lambda }\}_{\lambda \in \Pi _{n,k}}$ are
linearly independent and, thus, span $\mathbb{P}_{n}$. Since we have
previously identified the matrix elements $\langle \nu |\tilde{S}_{\lambda
}|\mu \rangle =\langle \nu |\tilde{S}_{\lambda }\tilde{S}_{\mu }|\emptyset
\rangle =q^{d}C_{\lambda \mu }^{\nu ,d}(T)$ with the Gromov-Witten
invariants for all $\lambda ,\mu ,\nu \in \Pi _{n,k}$, the image $\varphi (%
\tilde{S}_{\lambda })$ must be the Schubert class $\sigma _{\lambda }$.
\end{proof}

The following presentations are due to Laksov \cite[Examples 7.4 -6]{Laksov}.

\begin{corollary}[Laksov]
There exist isomorphisms such that%
\begin{equation}
QH_{T}^{\ast }(\limfunc{Gr}\nolimits_{n,N})\cong \Lambda \lbrack
q][h_{1},\ldots ,h_{k}]/\langle E_{n+1},\ldots
,E_{N-1},E_{N}+(-1)^{k}q\rangle  \label{L1}
\end{equation}%
with $E_{r}=\sum_{s=0}^{r}(-1)^{s}h_{r-s}(T_{1},\ldots ,T_{N})\det
(h_{1+a-b})_{1\leq a,b\leq r}$ and%
\begin{equation}
QH_{T}^{\ast }(\limfunc{Gr}\nolimits_{n,N})\cong \Lambda \lbrack
q][e_{1},\ldots ,e_{n}]/\langle H_{k+1},\ldots
,H_{N-1},H_{N}+(-1)^{n}q\rangle  \label{L2}
\end{equation}%
with $H_{j}=\sum_{r=0}^{j}(-1)^{s}e_{s}(T_{1},\ldots ,T_{N})\det
(e_{1+a-b})_{1\leq a,b\leq j-r}$.
\end{corollary}

\begin{proof}
Using the alternative form (\ref{M2'}) of the Bethe ansatz equations (\ref%
{BAE}) one obtains by the same arguments as before that the coordinate ring
defined by (\ref{M2'}) is equivalent to the presentation (\ref{GK}). This
yields the presentation (\ref{L2}). The presentation (\ref{L1}) then follows
from level-rank duality.
\end{proof}

Laksov showed that his presentations are equivalent to the following
alternative descriptions of $QH_{T}^{\ast }(\limfunc{Gr}_{n,N})$ due to
Mihalcea \cite{Mihalcea}.

\begin{corollary}[Mihalcea]
There exist canonical isomorphisms such that%
\begin{equation}
QH_{T}^{\ast }(\limfunc{Gr}\nolimits_{n,N})\cong \Lambda \lbrack q][\tilde{h}%
_{1},\ldots ,\tilde{h}_{k}]/\langle \tilde{e}_{n+1},\ldots ,\tilde{e}_{N-1},%
\tilde{e}_{N}+(-1)^{k}q\rangle  \label{M1}
\end{equation}%
and%
\begin{equation}
QH_{T}^{\ast }(\limfunc{Gr}\nolimits_{n,N})\cong \Lambda \lbrack q][\tilde{e}%
_{1},\ldots ,\tilde{e}_{n}]/\langle \tilde{h}_{k+1},\ldots ,\tilde{h}_{N-1},%
\tilde{h}_{N}+(-1)^{n}q\rangle  \label{M2}
\end{equation}%
where $\tilde{e}_{r}=\det (h_{1+j-i}(x|\tau ^{1-j}t))_{1\leq i,j\leq r}$ and 
$\tilde{h}_{r}=\det (e_{1+j-i}(x|\tau ^{j-1}t))_{1\leq i,j\leq r}$ .
\end{corollary}

Thus, we arrive at the following identification of the coefficients (\ref%
{facH}), (\ref{facE}) of the transfer matrices with respect to factorial
powers.

\begin{corollary}
The maps $\mathbb{P}_{n}\rightarrow QH_{T}^{\ast }(\limfunc{Gr}_{n,N})$
defined by $\tilde{H}_{j,n}\mapsto \tilde{h}_{j}$ and $\tilde{E}%
_{i,n}\mapsto \tilde{e}_{i}$ with respect to the presentations (\ref{M1}), (%
\ref{M2}) are both algebra isomorphisms. In particular, $\tilde{H}_{1,n}$
can be identified with the equivariant first Chern class of the $n$th
exterior power of the tautological $n$-plane bundle of the Grassmannian.
\end{corollary}

\subsection{The nil-Coxeter algebra and the quantum product}

We first focus on the simplest case $\tilde{H}_{1}=H_{1}+t_{n+1}+\cdots
+t_{N}$.

\begin{corollary}
We have the following modified Leibniz rule for the first Chern class,%
\begin{equation}
\delta _{j}^{\vee }\tilde{H}_{1}=(s_{j}\tilde{H}_{1})\delta _{j}^{\vee
}+(\partial _{j}\tilde{H}_{1})+(\delta _{jn}-\delta _{jN})\hat{r}_{j}
\end{equation}%
where the $\delta $'s in the last term denote the Kronecker delta. So, in
particular we obtain for the equivariant quantum Pieri formula,%
\begin{equation*}
\delta _{j}^{\vee }(\square \circledast \mu )=s_{j}(\square \circledast
\delta _{j}^{\vee }\mu )+\partial _{j}(\square \circledast \mu )+(\delta
_{jn}-\delta _{jN})\hat{r}_{j}\mu \;.
\end{equation*}
\end{corollary}

\begin{remark}
In contrast, Peterson states the following relation between quantum product
and the action of the affine nil-Coxeter algebra for any two Schubert
classes $\sigma ,\sigma ^{\prime }$ \cite[Prop 14.4]{Peterson}%
\begin{equation*}
(\partial _{j}\sigma )\ast \sigma ^{\prime }=\partial _{j}\left( \sigma \ast
(s_{j}\sigma ^{\prime })\right) +\sigma \ast (\partial _{j}\sigma ^{\prime })
\end{equation*}
\end{remark}

\begin{proposition}
Fix $\lambda \in \Pi _{n,k}$. If there exists a $1\leq j\leq $ $N$ such that
the coefficients in (\ref{facS}) are invariant under exchanging $%
t_{j},t_{j+1},$ then we have the following Leibniz rule for quantum
multiplication,%
\begin{equation}
\delta _{j}^{\vee }(\lambda \circledast \mu )=s_{j}(\lambda \circledast
\delta _{j}^{\vee }\mu )+\partial _{j}(\lambda \circledast \mu ),
\label{quantumLeibniz}
\end{equation}%
which directly relates Peterson's nil-Coxeter algebra action on Schubert
classes to the product in $QH_{T}^{\ast }(\limfunc{Gr}_{n,N})$.
\end{proposition}

\begin{proof}
Since it follows from (\ref{ybe}) that $\hat{r}_{j}E=(s_{j}E)\hat{r}_{j}$ we
must have $\hat{r}_{j}E_{\mu }=(s_{j}E_{\mu })\hat{r}_{j}$ with $E_{\mu
}=E_{\mu _{1}}\cdots E_{\mu _{k}}$ for all $\mu _{i}=1,\ldots ,n$. Hence,
exploiting the definition (\ref{S}) it follows that $\hat{r}_{j}S_{\lambda
}=(s_{j}S_{\lambda })\hat{r}_{j}$. This implies via (\ref{facS}) 
\begin{eqnarray}
\delta _{j}^{\vee }\tilde{S}_{\lambda } &=&(s_{j}\tilde{S}_{\lambda })\delta
_{j}^{\vee }+\sum_{\mu \subseteq \lambda }(c_{\mu }-(s_{j}c_{\mu
}))(s_{j}S_{\mu })\delta _{j}^{\vee }+\sum_{\mu \subseteq \lambda }c_{\mu
}(\partial _{j}S_{\mu })  \notag \\
&=&(s_{j}\tilde{S}_{\lambda })\delta _{j}^{\vee }+(\partial _{j}\tilde{S}%
_{\lambda })+\sum_{\mu \subseteq \lambda }(c_{\mu }-(s_{j}c_{\mu
}))(s_{j}S_{\mu })\delta _{j}^{\vee }-\sum_{\mu \subseteq \lambda }(\partial
_{j}c_{\mu })(s_{j}S_{\mu })  \notag \\
&=&(s_{j}\tilde{S}_{\lambda })\delta _{j}^{\vee }+(\partial _{j}\tilde{S}%
_{\lambda })-\sum_{\mu \subseteq \lambda }(\partial _{j}c_{\mu
})(s_{j}S_{\mu })\hat{r}_{j}
\end{eqnarray}%
where $c_{\mu }$ are the coefficients in (\ref{facS}) and we used once more
that $\hat{r}_{j}=1-(t_{j}-t_{j+1})\delta _{j}^{\vee }$. Under the stated
assumptions we have that $\partial _{j}c_{\mu }=0$ and the relation
simplifies to%
\begin{equation*}
\delta _{j}^{\vee }\tilde{S}_{\lambda }=(s_{j}\tilde{S}_{\lambda })\delta
_{j}^{\vee }+(\partial _{j}\tilde{S}_{\lambda })\;.
\end{equation*}%
Applying both sides of this identity to a basis vector $|\mu \rangle $ and
using the definition (\ref{comb_product}), the assertion follows.
\end{proof}

\begin{example}
Consider once more $QH_{T}^{\ast }(\limfunc{Gr}_{2,4})$ and set $\lambda
=(2,1)$ and $\mu =(2,2)$. Employing $s_{j}\partial _{j}=\partial _{j}$ we
can rewrite (\ref{quantumLeibniz}) as%
\begin{equation*}
(\lambda \circledast \delta _{j}^{\vee }\mu )=s_{j}\delta _{j}^{\vee
}(\lambda \circledast \mu )-\partial _{j}(\lambda \circledast \mu )
\end{equation*}%
which allows us to compute the product $\lambda \circledast \delta
_{j}^{\vee }\mu $ in terms of $\lambda \circledast \mu $ by using the
actions (\ref{DBGG}) and (\ref{delta}) of the nil-Coxeter algebra. For the
case at hand, we have $\delta _{2}^{\vee }\mu =(2,1)$, so we can calculate
the product of the Schubert classes $(2,1)\circledast (2,1)$ in terms of the
product expansion 
\begin{equation*}
(2,1)\circledast
(2,2)=T_{14}T_{13}T_{24}~(2,2)+q~(2,1)+qT_{13}~(2,0)+qT_{24}~(1,1)+qT_{13}T_{24}~(1,0)\ ,
\end{equation*}%
where $T_{ij}:=T_{i}-T_{j}$. The latter can be obtained by alternative means
such as the known recursion relations for Gromov-Witten invariants.
Converting each partition occurring in the product expansion into the
associated 01-word we now easily find%
\begin{equation*}
s_{2}\delta _{2}^{\vee }(\lambda \circledast \mu
)=T_{14}T_{12}T_{34}~(2,1)+qT_{12}T_{34}~(0,0)
\end{equation*}%
all other terms vanish. Noting that $-\partial _{2}=T_{23}^{-1}(1-s_{2})$ we
have in addition the terms%
\begin{eqnarray*}
-\partial _{2}(\lambda \circledast \mu ) &=&T_{14}\tfrac{%
T_{13}T_{24}-T_{12}T_{34}}{T_{23}}~(2,2)+q~(2,0)+q~(1,1)+q\tfrac{%
T_{13}T_{24}-T_{12}T_{34}}{T_{23}}~(1,0) \\
&=&T_{14}^{2}~(2,2)+q~(2,0)+q~(1,1)+qT_{14}~(1,0)\;.
\end{eqnarray*}%
Therefore, after collecting terms we find the product expansion%
\begin{equation*}
(2,1)\circledast
(2,1)=T_{14}^{2}~(2,2)+T_{14}T_{12}T_{34}~(2,1)+q~(2,0)+q~(1,1)+qT_{14}~(1,0)+qT_{12}T_{34}~(0,0)
\end{equation*}%
which can be verified via the recursion relations.
\end{example}

\subsection{Frobenius structures \& partition functions}

We are now ready to prove that the partition functions of the vicious and
osculating walker models are generating functions for equivariant
Gromov-Witten invariants. In fact, the partition functions have a natural
interpretation when looking at the quantum cohomology ring as a Frobenius
algebra.

\begin{proposition}
The Jacobi algebra $\mathfrak{J}\cong QH_{T}^{\ast }(\limfunc{Gr}%
_{n,N})\otimes \mathbb{F}$ with bilinear form%
\begin{equation}
\langle s_{\lambda }|s_{\mu }\rangle =\sum_{\alpha \in (n,k)}\frac{%
s_{\lambda }(y_{\alpha }|t)s_{\mu }(y_{\alpha }|t)}{\mathfrak{e}(y_{\alpha })%
},\qquad y_{\alpha }=\{y_{i}\}_{i\in I(\alpha )}  \label{bilinear_form}
\end{equation}%
is a commutative Frobenius algebra.
\end{proposition}

\begin{proof}
The bilinear form is non-degenerate and obeys%
\begin{eqnarray*}
\langle s_{\lambda }|s_{\mu }s_{\nu }\rangle &=&\sum_{\rho \in (n,k)}C_{\mu
\nu }^{\rho ,d}(T)\sum_{\alpha \in (n,k)}\frac{s_{\lambda }(y_{\alpha
}|t)s_{\rho }(y_{\alpha }|t)}{\mathfrak{e}(y_{\alpha })} \\
&=&\sum_{\alpha ,\beta \in (n,k)}\frac{s_{\lambda }(y_{\alpha }|t)s_{\mu
}(y_{\beta }|t)s_{\nu }(y_{\beta }|t)}{\mathfrak{e}(y_{\beta })\mathfrak{e}%
(y_{\alpha })}\sum_{\rho \in (n,k)}s_{\rho ^{\vee }}(y_{\beta }|T)s_{\rho
}(y_{\alpha }|t) \\
&=&\sum_{\alpha \in (n,k)}\frac{s_{\lambda }(y_{\alpha }|t)s_{\mu
}(y_{\alpha }|t)s_{\nu }(y_{\alpha }|t)}{\mathfrak{e}(y_{\alpha })}
\end{eqnarray*}%
The last result is clearly invariant under permutations of $(\lambda ,\mu
,\nu )$, so we can conclude that $\langle s_{\lambda }|s_{\mu }s_{\nu
}\rangle =\langle s_{\lambda }s_{\mu }|s_{\nu }\rangle $ as required.
\end{proof}

\begin{lemma}
The image of a Schubert class $s_{\lambda }$ under the Frobenius coproduct $%
\Delta _{n,k}:\mathfrak{J}\rightarrow \mathfrak{J}\otimes \mathfrak{J}$ is
given by%
\begin{equation}
\Delta _{n,k}s_{\lambda }=\sum_{\substack{ \mu \in (n,k)  \\ d\geq 0}}%
q^{d}s_{\lambda /d/\mu }\otimes s_{\mu },  \label{Frob_cop}
\end{equation}%
where $s_{\lambda /d/\mu }$ is a generalised factorial skew Schur function,%
\begin{equation}
s_{\lambda /d/\mu }(x|t)=\sum_{\nu \in (n,k)}C_{\mu ^{\vee }\nu ^{\vee
}}^{\lambda ^{\vee },d}(t)s_{\nu }(x|t)\;  \label{equitoricschur}
\end{equation}%
with the coefficients $C_{\mu ^{\vee }\nu ^{\vee }}^{\lambda ^{\vee },d}(t)$
given via (\ref{residue_formula}).
\end{lemma}

\begin{proof}
Let $\Phi :\mathfrak{J}\rightarrow \mathfrak{J}^{\ast }$ denote the
Frobenius isomorphism given by $s_{\lambda }\mapsto \langle s_{\lambda
}|~\cdot ~\rangle $ and $m:\mathfrak{J}\times \mathfrak{J}\rightarrow 
\mathfrak{J}$ the multiplication map. Then%
\begin{equation*}
m^{\ast }\circ \Phi (s_{\lambda })(s_{\mu }\otimes s_{\nu })=\langle
s_{\lambda }|s_{\mu }s_{\nu }\rangle =\sum_{\alpha \in (n,k)}\frac{%
s_{\lambda }(y_{\alpha }|t)s_{\mu }(y_{\alpha }|t)s_{\nu }(y_{\alpha }|t)}{%
\mathfrak{e}(y_{\alpha })}
\end{equation*}%
as we saw earlier. Since this result must match%
\begin{equation*}
\lbrack (\Phi \otimes \Phi )\Delta s_{\lambda }](s_{\mu }\otimes s_{\nu
})=\sum_{\substack{ \rho \in (n,k)  \\ d\geq 0}}q^{d}\langle s_{\lambda
/d/\rho }|s_{\mu }\rangle \langle s_{\rho }|s_{\nu }\rangle
\end{equation*}%
one arrives at the stated definition of $s_{\lambda /d/\rho }$ and the
claimed identity for the coproduct.
\end{proof}

\begin{remark}
For $T_{i}=0$ the function (\ref{equitoricschur}) specialises to Postnikov's
toric Schur function. For $T_{i}\neq 0$ it is a nontrivial equivariant
generalisation.
\end{remark}

\begin{proposition}
The partition functions of the osculating and vicious walker models are
related to the coproduct of $QH_{T}^{\ast }(\limfunc{Gr}_{n,N})$ seen as
Frobenius algebra,%
\begin{eqnarray}
\tilde{Z}_{\lambda ,\mu }(x|t) &=&\sum_{d\geq 0}q^{d}s_{\lambda ^{\vee
}/d/\mu ^{\vee }}(x|T),  \label{Z2toric} \\
\tilde{Z}_{\lambda ,\mu }^{\prime }(x|t) &=&\sum_{d\geq 0}q^{d}s_{\lambda
^{\ast }/d/\mu ^{\ast }}(x|-t),  \label{Z'2toric}
\end{eqnarray}%
where the right hand side of the above identities are the generalised
factorial skew Schur functions defined in (\ref{equitoricschur}).
\end{proposition}

\begin{proof}
Note that proving the two assertions amounts to proving the expansions%
\begin{equation}
\tilde{Z}_{\lambda ,\mu }(x|t)=\sum_{\substack{ \nu \in (n,k)  \\ d\geq 0}}%
q^{d}C_{\mu \nu }^{\lambda ,d}(T)s_{\nu ^{\vee }}(x|T)  \label{Z2facs}
\end{equation}%
and%
\begin{equation}
\tilde{Z}_{\lambda ,\mu }^{\prime }(x|t)=\sum_{\substack{ \nu \in (n,k)  \\ %
d\geq 0}}q^{d}C_{\mu \nu }^{\lambda ,d}(T)s_{\nu ^{\ast }}(x|-t)\;.
\label{Z'2facs}
\end{equation}%
Let us prove the second identity. First we recall from (\ref%
{partition_function}), (\ref{combZ}) that $\tilde{Z}_{\lambda ,\mu }^{\prime
}(x|t)=(x_{1}\cdots x_{k})^{n}\langle \lambda |E(x_{1}^{-1})\cdots
E(x_{k}^{-1})|\mu \rangle $. Employing the result (\ref{specE}) from the
Bethe ansatz and the Cauchy identity (\ref{Cauchy}) for factorial Schur
functions we find,%
\begin{multline*}
\tilde{Z}_{\lambda ,\mu }^{\prime }(x|t)=\sum_{\alpha \in (n,k)}(x_{1}\cdots
x_{k})^{n}\langle \lambda |E(x_{1}^{-1})\cdots E(x_{k}^{-1})|y_{\alpha
}\rangle \langle y_{\alpha }|\mu \rangle \\
=\sum_{\alpha \in (n,k)}\frac{s_{\mu }(y_{\alpha }|t)s_{\lambda ^{\vee
}}(y_{\alpha }|T)}{\mathfrak{e}(y_{\alpha })}\prod_{i=1}^{k}%
\prod_{j=1}^{n}(x_{i}+y_{j}(\alpha )) \\
=\sum_{\nu \in (n,k)}\sum_{\alpha \in (n,k)}\frac{s_{\mu }(y_{\alpha
}|t)s_{\lambda ^{\vee }}(y_{\alpha }|T)}{\mathfrak{e}(y_{\alpha })}s_{\nu
}(y_{\alpha }|t)s_{\nu ^{\ast }}(x|-t) \\
=\sum_{\nu \in (n,k)}C_{\mu \nu }^{\lambda ,d}(T)s_{\nu ^{\ast }}(x|-t)
\end{multline*}%
The identity for the vicious walker model now follows from level-rank
duality (\ref{levelrankmom}) and (\ref{GW_level-rank}),%
\begin{eqnarray*}
\tilde{Z}_{\lambda ,\mu }^{\prime }(x|t) &=&(x_{1}\cdots x_{k})^{n}\langle
\lambda |E(x_{1}^{-1}|t)\cdots E(x_{k}^{-1}|t)|\mu \rangle \\
&=&(x_{1}\cdots x_{k})^{n}\langle \lambda ^{\prime }|H(x_{1}^{-1}|-T)\cdots
H(x_{k}^{-1}|-T)|\mu ^{\prime }\rangle =\tilde{Z}_{\lambda ^{\prime },\mu
^{\prime }}^{\prime }(x|-T)\;.
\end{eqnarray*}
\end{proof}

We can restate the last result in operator form; compare with (\ref{Cauchy}).

\begin{corollary}[noncommutative Cauchy identities]
We have the identities%
\begin{equation}
\tilde{Z}_{n}=\sum_{\alpha }(x|T)^{\alpha }\otimes \tilde{H}_{\alpha ^{\vee
}}=\sum_{\lambda \in \Pi _{n,k}}s_{\lambda ^{\vee }}(x|T)\otimes \tilde{S}%
_{\lambda }\;.  \label{ncCauchy}
\end{equation}%
Here the first sum runs over all compositions $\alpha =(\alpha _{1},\ldots
,\alpha _{n})$ with $\alpha _{i}\leq k$ and $\alpha ^{\vee }=(k-\alpha
_{1},\ldots ,k-\alpha _{n})$. The analogous identities are true for $\tilde{Z%
}_{k}^{\prime }$.
\end{corollary}

\begin{proof}
The first identity is a direct consequence of the definition of the
partition function and (\ref{facH}), the second follows from the last
proposition - see (\ref{Z2facs}) - and (\ref{residue_formula}).
\end{proof}

So far we have concentrated on the expansions of the partition functions
into factorial Schur functions. The first expansion in (\ref{ncCauchy}) is
also related to products in the quantum cohomology ring.

\begin{corollary}[equivariant quantum Kostka numbers]
Let $\tilde{h}_{r}$ and $\tilde{e}_{r}$ be the generators in Mihalcea's
presentation (\ref{M1}) and (\ref{M2}). The coefficients in the product
expansions%
\begin{eqnarray}
\tilde{h}_{\alpha }\ast s_{\mu } &=&\sum_{\lambda \in \Pi
_{n,k}}q^{d}K_{\lambda /d/\mu ,\alpha }(T)~s_{\lambda } \\
\tilde{e}_{\alpha }\ast s_{\mu } &=&\sum_{\lambda \in \Pi
_{n,k}}q^{d}K_{\lambda ^{\prime }/d/\mu ^{\prime },\alpha }(T)~s_{\lambda
^{\prime }}
\end{eqnarray}%
are given by the coefficients of the following polynomials in $q$%
\begin{equation}
\langle \lambda |\tilde{H}_{\alpha }|\mu \rangle =\langle \lambda ^{\prime }|%
\tilde{E}_{\alpha }|\mu ^{\prime }\rangle =\sum_{d\geq 0}q^{d}K_{\lambda
/d/\mu ,\alpha }(T)\;.
\end{equation}
\end{corollary}

Using the determinant formulae (\ref{det_formulae}) the last result provides
a method to compute Gromov-Witten invariants. However, we believe the
following algorithm to be simpler.

\subsection{A determinant formula for Gromov-Witten invariants}

We can use the expansions (\ref{Z2facs}), (\ref{Z'2facs}) to state a
determinant formula for equivariant Gromov-Witten invariants in terms of the
partition function of vicious and osculating walkers. We recall the
following theorem due to Molev and Sagan \cite[Thm 2.1]{MolevSagan}.

\begin{theorem}[Vanishing Theorem]
Let $\lambda ,\mu $ be partitions with $\ell (\lambda ),\ell (\mu )\leq n$
and set $a_{\mu }=(a_{\mu _{1}+n},\ldots ,a_{\mu _{n}+1})$. Then%
\begin{equation}
s_{\lambda }(a_{\mu }|a)=\left\{ 
\begin{array}{cc}
\prod_{(i,j)\in \lambda }(a_{\lambda _{i}+n+1-i}-a_{n-\lambda _{j}^{\prime
}+j}), & \lambda =\mu \\ 
0, & \lambda \nsubseteq \mu%
\end{array}%
\right. \;,  \label{vanish}
\end{equation}%
where $\lambda ^{\prime }$ is the conjugate partition of $\lambda $.
\end{theorem}

\begin{corollary}
Given $\lambda ,\mu ,\nu \in \Pi _{n,k}$ with $|\lambda |+|\mu |-|\nu |=dN$
let $\gamma (\nu )=(s_{\beta }(T_{\alpha }|T))_{\emptyset \leq \alpha ,\beta
\leq \nu }$ with respect to the dominance order of partitions. Denote by $%
\Gamma (\nu )$ the matrix which is obtained by replacing the first column
vector in $\gamma (\nu )$ with $(\tilde{Z}_{\lambda ,\mu }(T_{\nu
}|t),\ldots ,\tilde{Z}_{\lambda ,\mu }(T_{\emptyset }|t))^{t}$. Then we have
the identity%
\begin{equation}
q^{d}C_{\mu \nu ^{\vee }}^{\lambda ,d}(T)=\frac{\det \Gamma (\nu )}{%
\prod_{\rho \subseteq \nu }s_{\rho }(T_{\rho }|T)}\;.  \label{detGW}
\end{equation}%
In particular, setting $\nu =\emptyset $ this simplifies to%
\begin{equation}
q^{d}C_{\mu ~k^{n}}^{\lambda ,d}(T)=\tilde{Z}_{\lambda ,\mu }(t_{k^{n}}|t)\;.
\label{Z=GW}
\end{equation}
\end{corollary}

\begin{proof}
It follows from the expansion (\ref{Z2toric}) that%
\begin{equation}
\tilde{Z}_{\lambda ,\mu }(T_{\nu }|t)=\sum_{\rho \subseteq \nu }q^{d}C_{\mu
\rho ^{\vee }}^{\lambda ,d}(T)~s_{\rho }(T_{\nu }|T)\;.  \label{Zspecial}
\end{equation}%
This defines a linear system of inhomogeneous equations for the
Gromov-Witten invariants $C_{\mu \rho ^{\vee }}^{\lambda ,d}(T)$ where $%
\emptyset \leq \rho \leq \nu $ in the dominance order. Formula (\ref{detGW})
is then simply Cramer's rule, $q^{d}C_{\mu \rho ^{\vee }}^{\lambda
,d}(T)=\det \Gamma (\rho )/\det \gamma (\rho ),$ upon noting that due to (%
\ref{vanish}) the determinant in denominator simplifies as $\gamma (\rho )$
is triangular,%
\begin{equation*}
\det \gamma (\rho )=\det (s_{\beta }(T_{\alpha }|T))_{\emptyset \leq \alpha
,\beta \leq \rho }=\prod_{\alpha \subseteq \rho }s_{\alpha }(T_{\alpha
}|T)\;.
\end{equation*}
\end{proof}

\begin{figure}[tbp]
\begin{equation*}
\includegraphics[scale=0.7]{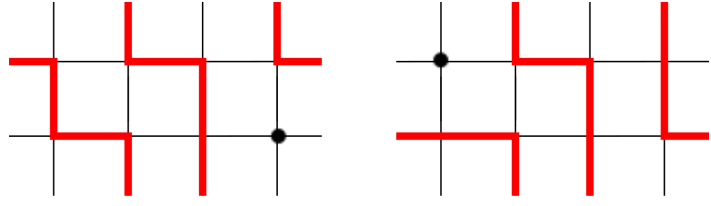}
\end{equation*}%
\caption{Shown are the two possible vicious paths which connect the 01-words
for $\protect\mu =(2,1)$ and $\protect\lambda =(1,1)$.}
\label{fig:quantumex}
\end{figure}

\begin{example}
We use once more $QH_{T}^{\ast }(\limfunc{Gr}_{2,4})$ as a simple example to
demonstrate how to compute Gromov-Witten invariants using (\ref{Zspecial}).
Recall from Figure \ref{fig:gr24ex} that the partition function of the
vicious walker model for $\lambda =(2,2)$ and $\mu =(2,1)$ is%
\begin{equation*}
\tilde{Z}_{\lambda ,\mu
}(x_{1},x_{2}|t)=(x_{1}-T_{4})(x_{2}-T_{3})(x_{2}-T_{4})+(x_{1}-T_{2})(x_{1}-T_{4})(x_{2}-T_{4})\;.
\end{equation*}%
Thus, starting from $\nu =\emptyset $ in (\ref{Zspecial}) we set $%
x=T_{\emptyset }$ to obtain%
\begin{equation*}
C_{\mu ~(2,2)}^{\lambda ,0}(T)=T_{13}T_{14}T_{24}\;.
\end{equation*}%
For the next step we choose $\nu =(1,0)$ and obtain the equation%
\begin{equation*}
\frac{\left\vert 
\begin{array}{cc}
\tilde{Z}_{\lambda ,\mu }(T_{1,0}|t) & 1 \\ 
\tilde{Z}_{\lambda ,\mu }(T_{0,0}|t) & 1%
\end{array}%
\right\vert }{s_{1,0}(T_{1,0}|T)}=-\frac{%
T_{13}T_{14}T_{34}-T_{23}T_{14}T_{34}-C_{\mu ~(2,2)}^{\lambda ,0}(T)}{T_{23}}%
=T_{14}^{2}
\end{equation*}%
Hence, we have recovered our previous result from Figure \ref{fig:KTpuzzles}%
. Continuing in the same manner with $\nu =(1,1),(2,0),(2,1),(2,2)$ one
successively finds the remaining invariants as%
\begin{equation*}
C_{\mu ~(1,1)}^{\lambda ,0}(T)=T_{14},\quad C_{\mu ~(2,0)}^{\lambda
,0}(T)=T_{14},\quad C_{\mu ~(1,0)}^{\lambda ,0}(T)=1,\quad C_{\mu
~(0,0)}^{\lambda ,0}(T)=0\;.
\end{equation*}%
We have deliberately chosen a simple example with $d=0$ so that all results
can be easily checked by the reader using Knutson-Tao puzzles. However, our
combinatorial algorithm works also for $q\neq 0$; see Figure \ref%
{fig:quantumex} which shows the path configurations for $\tilde{Z}%
_{(1,1),(2,1)}(x_{1},x_{2}|t)$ $=q(x_{1}-T_{4})+q(x_{2}-T_{1})$. Thus, we
find $qC_{\mu ~(2,2)}^{(1,1),d=1}(T)=\tilde{Z}%
_{(1,1),(2,1)}(T_{2},T_{1}|t)=q(T_{2}-T_{4})$.
\end{example}

\begin{remark}
During the writing up of this manuscript two works appeared on different
combinatorial approaches to compute Gromov-Witten invariants. The work \cite%
{BKPT} proves a conjecture of Knutson which states that puzzles for two-step
flag varieties describe the product of non-equivariant quantum cohomology
for Grassmannians, while the work \cite{BBT} describes a generalised
rim-hook formula to compute equivariant Gromov-Witten invariants for
Grassmannians. We hope to address how these latter combinatorial approaches
are related to formula (\ref{detGW}) in future work.
\end{remark}

\subsection{Relation with Peterson's basis}

Based on Kostant and Kumar's earlier work \cite{KostantKumar} on the
nil-Hecke ring, Dale Peterson constructed a special commutative subalgebra
to describe the equivariant Schubert calculus of the \emph{homology} of the
affine Grassmannian of an algebraic group $G$ and the \emph{quantum
cohomology} of the partial flag variety $G/P$ where $P\subseteq B$ is a
parabolic subgroup for a given Borel subgroup $B\subset G$. For completeness
we briefly recall its construction; see \cite{Peterson} and \cite{LamShim}
for details.

Let $W$ be the finite Weyl group associated with $G$ and $\hat{W}\cong
W\rtimes \mathcal{Q}^{\vee }$ the affine Weyl group where $\mathcal{Q}^{\vee
}$ is the finite coroot lattice. Let $\mathcal{P}$ be the corresponding
finite weight lattice and set $\boldsymbol{S}=\limfunc{Sym}(\mathcal{P})$ to
be the symmetric algebra and $\limfunc{Frac}(\boldsymbol{S})$ the fraction
field. $\hat{W}$ acts on $\mathcal{P}$ via the level-0 action, i.e. the
affine Weyl reflection acts by reflecting on the hyperplane defined by the
negative highest root. Given $w\in \hat{W}$ and a reduced decomposition $%
w=s_{j_{1}}\cdots s_{j_{r}}$ into simple Weyl reflections define $%
A_{w}=A_{j_{1}}\ldots A_{j_{r}}$ where $A_{j}=\alpha _{j}^{-1}(1-s_{j})$
with $\alpha _{j}$ the $j$th simple root. The $A_{j}$ obey the nil-Coxeter
relations. The level zero graded affine nil-Hecke ring $\mathbb{H}$ is then
given by $\mathbb{H=}\tbigoplus_{w\in \hat{W}}\boldsymbol{S}A_{w}$ with
commutation relations%
\begin{equation}
A_{j}\lambda =(s_{j}\lambda )A_{j}+(\lambda ,\alpha _{j}^{\vee })\;.
\label{Pet_Leibniz}
\end{equation}%
Denote by $\mathbb{P=}\mathcal{Z}_{\mathbb{H}}(\boldsymbol{S})$ the
centralizer of $\boldsymbol{S}$ in $\mathbb{H}$. Following the literature on
this subject we shall refer to this subalgebra as \textquotedblleft Peterson
algebra\textquotedblright . We will be using the following results which are
originally due to Peterson and have been proved in \cite[Lem 3.3 and Thm 4.4]%
{Lam08} \cite[Thm 6.2 and Thm 10.16]{LamShim}

The Peterson algebra $\mathbb{P}$ has a special basis $\{\jmath _{x}:x\in 
\hat{W}/W\}$ which are the images of equivariant Schubert classes under the
isomorphism $H_{T}(\limfunc{Gr}_{G})\cong \mathbb{P}$. Each coset can be
labelled in terms of a unique minimal length representative $\hat{w}_{x}$
and the latter form the set of \emph{Grassmannian affine permutations }$\hat{%
W}^{-}$.

\begin{theorem}[Peterson's basis]
There is an $\boldsymbol{S}$-algebra isomorphism $\jmath :H_{T}(\limfunc{Gr}%
_{G})\rightarrow \mathcal{Z}_{\mathbb{H}}(\boldsymbol{S})$ such that%
\begin{equation}
\jmath (\xi _{x})=A_{x}\func{mod}J\qquad \text{and\qquad }\jmath (\xi )\cdot
\xi ^{\prime }=\xi \xi ^{\prime }  \label{Pet_basis}
\end{equation}%
where $J\subset \mathbb{H}$ is the left ideal $J=\sum_{w\in W\backslash
\{id\}}\mathbb{H}A_{w}$ and $\{\xi _{x}|x\in \hat{W}/W\}$ are the $T$%
-equivariant Schubert classes. For each $x$ the basis element $\jmath
_{x}=\jmath (\xi _{x})$ is uniquely fixed by the two properties in (\ref%
{Pet_basis}).
\end{theorem}

Let $\limfunc{Fun}(\hat{W},\limfunc{Frac}(\boldsymbol{S}))$ be the $\limfunc{%
Frac}(\boldsymbol{S})$-algebra of functions $\hat{W}\rightarrow \limfunc{Frac%
}(\boldsymbol{S})$ with pointwise product. Recall that the torus $T$ acts on 
$\limfunc{Gr}_{G}$ by pointwise conjugation. Restricting a class to the $T$%
-fixed points gives an injective $\boldsymbol{S}$-algebra homomorphism $\phi
:H_{T}(\limfunc{Gr}_{G})\rightarrow \limfunc{Fun}(\hat{W},\boldsymbol{S})$
where the image under $\phi $ are the functions which satisfy the GKM
conditions \cite{GKM} \cite{KostantKumar}, $\xi (w)-\xi (r_{\alpha }w)\in
\alpha \boldsymbol{S}$ for affine real roots $\alpha $. Henceforth, we will
be identifying Schubert classes with their images under $\phi $.

We have the following projection from the homology of the affine
Grassmannian to the cohomology of the partial flag variety \cite[Thm 10.16]%
{LamShim}; for further details see \emph{loc. cit}. Let $W_{P}$ be the Weyl
group of the parabolic subgroup and $\hat{W}_{P}=W_{P}\rtimes \mathcal{Q}%
_{P}^{\vee }$ its affinisation with $\mathcal{Q}_{P}^{\vee }$ being the
parabolic coroot lattice. Each affine Weyl group element factorises as $\hat{%
w}=\hat{w}_{P}\hat{w}^{P}$, where $\hat{w}_{P}\in \hat{W}_{P}$ and $\hat{w}%
^{P}$ is a coset representative in $\hat{W}^{P}=\hat{W}_{P}/W_{P}$ such that 
$\hat{w}^{P}=w^{P}w_{P}\tau _{\lambda }$ with $w_{P}\in W_{P}$, $w^{P}$ a
minimal length coset representative in $W/W_{P}$ and $\tau _{\lambda }$ a
translation; compare with \cite[Lem 10.1-5]{LamShim}. Let $\pi _{P}:\hat{W}%
\rightarrow \hat{W}^{P}$ be the quotient map with $\pi _{P}(\hat{w})=\hat{w}%
^{P}$. The quotient map has the properties that $\pi _{P}(W)=W^{P}$, $\pi
_{P}(w\tau _{\lambda })=\pi _{P}(w)\pi _{P}(\tau _{\lambda })$ and $\pi
_{P}(\tau _{\lambda +\alpha })=\pi _{P}(\tau _{\lambda })$ for $\alpha \in 
\mathcal{Q}_{P}^{\vee }$; see \cite[Prop 10.8]{LamShim}. The following
theorem then states that the quantum cohomology of the partial flag variety $%
G/P$ can be described as a quotient of the homology of the affine
Grassmannian.

\begin{theorem}[{\protect\cite[Thm 10.16]{Peterson,LamShim}}]
Let $J_{P}$ be the ideal $J_{P}=\sum_{x\in \hat{W}^{-}\backslash \hat{W}%
^{P}}S\xi _{x}$. Then the map $\psi _{P}:H_{T}(\limfunc{Gr}%
_{G})/J_{P}\rightarrow QH_{T}^{\ast }(G/P)$ defined by $\xi _{w\pi _{P}(\tau
_{\lambda })}\mapsto q_{\eta _{P}(\lambda )}\sigma ^{w}$ becomes an $%
\boldsymbol{S}$-algebra isomorphism after the appropriate localisation,
where $\eta _{P}$ is the natural projection $\mathcal{Q}^{\vee }\rightarrow 
\mathcal{Q}^{\vee }/\mathcal{Q}_{P}^{\vee }$ and $w\in W^{P}$ is a minimal
length representative of coset in $W/W_{P}$.
\end{theorem}

We now specialise to $G=SL_{N}$, $\hat{W}=\mathbb{\hat{S}}_{N}$, $W=\mathbb{S%
}_{N}$ and let $P\subset SL_{N}$ be the subgroup which maps the subspace $%
\mathbb{C}^{n}\subset \mathbb{C}^{N}$ spanned by the first $n$-basis vectors
into itself. Then Peterson's representation becomes the representation in
terms of divided difference operators (\ref{DBGG}), $A_{j}=\partial _{j}$,
and $\boldsymbol{S}$ can be identified with the polynomial ring $\Lambda $
of the equivariant parameters.

There is a bijection between partitions $\lambda \in \Pi _{n,k}$ and minimal
length representatives of the cosets in $\mathbb{S}_{N}/\mathbb{S}_{n}\times 
\mathbb{S}_{k}$: let $w^{\lambda }\in \mathbb{S}_{N}$ be the permutation
defined by $w^{\lambda }(i)=\lambda _{n+1-i}+i$ for $i=1,\ldots ,n$ and $%
w^{\lambda }(i+n)=\lambda _{k+1-i}^{\prime }+i$ for $i=1,\ldots ,k$. After
the projection onto $QH_{T}^{\ast }(G/P)$ the Peterson basis elements can be
labelled by minimal length representatives and powers in the quantum
parameter $q$. Given $w^{\lambda }$ let $\hat{w}^{\lambda }$ be the affine
Grassmannian permutation with $\pi _{P}(\hat{w}^{\lambda })=w^{\lambda }$.%
\textbf{\ }

\begin{proposition}
Let $Y_{\alpha }=\mathfrak{e}(y_{\alpha })^{-1}|y_{\alpha }\rangle $ be the
renormalised Bethe vectors. Given any $\lambda ,\mu \in \Pi _{n,k}$ we have
the following relation between Peterson's basis and the non-commutative
Schur polynomials (\ref{facS}), 
\begin{equation}
|\lambda \rangle \circledast |\mu \rangle =\sum_{\alpha \in \Pi
_{n,k}}s_{\mu }(y_{\alpha }|t)~(\tilde{S}_{\lambda }Y_{\alpha
})=\sum_{\alpha \in \Pi _{n,k}}(j_{\hat{w}^{\lambda }}s_{\mu }(y_{\alpha
}|t))~Y_{\alpha }\;.  \label{facS2petbasis}
\end{equation}

\begin{proof}
From (\ref{GKM}) it follows that we can identify the coefficients $\xi _{\mu
}=\{s_{\mu }(y_{\alpha }|t)\}_{\alpha \in \Pi _{n,k}}$ of the Bethe vectors
with a localised Schubert class. As we saw earlier, Lemma \ref{lem:facS}, $%
\tilde{S}_{\lambda }$ acts on a Bethe vector by multiplying it with $%
s_{\lambda }(y_{\alpha }|t)$. From the above theorems we know that the
Peterson basis element $j_{\hat{w}^{\lambda }}$ also acts by multiplication
with the GKM class $\xi _{\lambda }(y_{\alpha })=s_{\lambda }(y_{\alpha }|t)$%
. So, starting from (\ref{basischange}) we obtain%
\begin{equation*}
\tilde{S}_{\lambda }|\mu \rangle =\sum_{\alpha }s_{\mu }(y_{\alpha }|t)%
\tilde{S}_{\lambda }Y_{\alpha }=\sum_{\alpha }s_{\lambda }(y_{\alpha
}|t)s_{\mu }(y_{\alpha }|t)Y_{\alpha }=\sum_{\alpha }(j_{\hat{w}^{\lambda
}}s_{\mu }(y_{\alpha }|t))Y_{\alpha }\;.
\end{equation*}
\end{proof}
\end{proposition}

\begin{remark}
Our construction has the following natural generalisation. Let $\mathcal{V}%
_{P}=\limfunc{Fun}(W^{P},\limfunc{Frac}(\boldsymbol{S}))\otimes \mathbb{Z}%
W^{P}$. For any partial flag variety $G/P$ we define a set of commutative
elements $\{S_{w}|w\in W^{P}\}\subset \limfunc{End}\mathcal{V}_{P}$ by
demanding that for any $u,v\in W^{P}$ we have%
\begin{equation*}
\sum_{u\in W^{P}}\xi _{v}(u)\otimes S_{w}Y_{u}=\sum_{u\in W^{P}}(\jmath
_{w}\xi _{v}(u))\otimes Y_{u}
\end{equation*}%
where $\{Y_{u}|u\in W^{P}\}$ are the idempotents of a suitable extension of $%
QH_{T}^{\ast }(G/P)$ over an algebraically closed field, $\xi _{v}$ the GKM
class corresponding to $v$ and $\jmath _{w}$ the projected Peterson basis
element. Then the map%
\begin{equation*}
|v\rangle :=\sum_{u\in W^{P}}\xi _{v}(u)\otimes Y_{u}\mapsto \sigma _{v}
\end{equation*}%
provides an algebra isomorphism $\mathcal{V}_{P}\rightarrow QH_{T}^{\ast
}(G/P)$ where the product in $\mathcal{V}_{P}$ is given by $|w\rangle
\circledast |v\rangle =S_{w}|v\rangle $. The open problem is to find an
explicit description of the operators $\{S_{w}|w\in W^{P}\}$ for partial
flag varieties other than type $A.$
\end{remark}


\begin{thebibliography}{99}
\bibitem{BMO} Braverman, A. , D. Maulik, and A. Okounkov. \textquotedblleft
Quantum cohomology of the Springer resolution", Preprint (2010), 1--35.
arXiv preprint arXiv:1001.0056.

\bibitem{BBT} Beazley, E., A. Bertiger and K. Taipale, \textquotedblleft An
equivariant rim hook rule for quantum cohomology of Grassmannians" extended
abstract, submitted to DMTCS.

\bibitem{Bertram} Bertram, A. \textquotedblleft Quantum Schubert calculus",
Adv. Math. 128 (2) (1997) 289--305.

\bibitem{Buch} Buch, A. S. \textquotedblleft Quantum cohomology of
Grassmannians", Compos. Math. 137 (2) (2003) 227--235.

\bibitem{BKPT} Buch, A. S., A. Kresch, K. Purbhoo, H. Tamvakis,
\textquotedblleft The puzzle conjecture for the cohomology of two-step flag
manifolds", arXiv preprint arXiv:1401.1725 (2014).

\bibitem{BKT} Buch, A. S., A. Kresch, H. Tamvakis, \textquotedblleft
Gromov--Witten invariants on Grassmannians", J. Amer. Math. Soc. 16 (4)
(2003) 901--915 (electronic).

\bibitem{BM} Buch, A. S. and L. C. Mihalcea, \textquotedblleft Quantum
K-theory of Grassmannians", Duke Math. J. Volume 156, Number 3 (2011),
501-538

\bibitem{BuchKtheory} Buch, A. S. \textquotedblleft A Littlewood-Richardson
rule for the K-theory of Grassmannians", Acta Mathematica 2002, Volume 189,
Issue 1, pp 37-78

\bibitem{Bumpetal} Brubaker, B., D. Bump, S. Friedberg, \textquotedblleft
Schur Polynomials and The Yang-Baxter Equation", Commun. Math. Phys. 308,
281--301 (2011); Daniel Bump, Peter J. McNamara, Maki Nakasuji,
\textquotedblleft Factorial Schur functions and the Yang-Baxter equation";
arXiv:1108.3087

\bibitem{Brak} Brak, R. \textquotedblleft Osculating lattice paths and
alternating sign matrices." In Proceeding of Formal Power Series and
Algebraic Combinatorics, vol. 9, p. 120. 1997.

\bibitem{Fisher} Fisher, M. E. \textquotedblleft Walks, walls, wetting, and
melting." Journal of Statistical Physics 34, no. 5-6 (1984): 667-729.

\bibitem{FS} Fomin, S. and R.P. Stanley, Schubert polynomials and the
nilCoxeter algebra, Adv. in Math. 103 (1994) 196-207.

\bibitem{FR} Frenkel, I. B.; Reshetikhin, N. Yu. (1992), \textquotedblleft
Quantum affine algebras and holonomic difference equations", Comm. Math.
Phys. 146 (1): 1--60.

\bibitem{Gepner} Gepner, D. \textquotedblleft Fusion rings and geometry."
Communications in Mathematical Physics 141, no. 2 (1991): 381-411.

\bibitem{GesselKrattenthaler} Gessel, Ira, and Christian Krattenthaler.
\textquotedblleft Cylindric partitions." Transactions of the American
Mathematical Society 349, no. 2 (1997): 429-479.

\bibitem{GP} Gorbounov, V., and V. Petrov. \textquotedblleft Schubert
calculus and singularity theory." Journal of Geometry and Physics 62, no. 2
(2012): 352-360.

\bibitem{Gorbetal} Gorbounov, V., R. Rimanyi, V. Tarasov, and A. Varchenko.
\textquotedblleft Cohomology of the cotangent bundle of a flag variety as a
Yangian Bethe algebra." arXiv preprint arXiv:1204.5138 (2012).

\bibitem{GK} Givental, A. and B. Kim. \textquotedblleft Quantum cohomology
of flag manifolds and Toda lattices". Comm. Math. Phys., 168:609--641, 1995

\bibitem{GV} Ginzburg, V. and \'{E}. Vasserot. \textquotedblleft Langlands
reciprocity for affine quantum groups of type $A_{n}$." International
Mathematics Research Notices 1993, no. 3 (1993): 67-85.

\bibitem{Ginzburg} Ginzburg, V. \textquotedblleft Geometric methods in the
representation theory of Hecke algebras and quantum groups." In
Representation Theories and Algebraic Geometry, pp. 127-183. Springer
Netherlands, 1998.

\bibitem{GKM} Goresky, M., R. Kottwitz, and R. MacPherson. \textquotedblleft
Equivariant cohomology, Koszul duality, and the localization theorem."
Inventiones mathematicae 131, no. 1 (1997): 25-83.

\bibitem{Intriligator} Intriligator, K. \textquotedblleft Fusion residues",
Modern Phys. Lett. A 6 (38) (1991) 3543--3556.

\bibitem{Kim1} Kim, B. \textquotedblleft Quantum cohomology of partial flag
manifolds and a residue formula for their intersection pairings". IMRN,
(1):1--15, 1995.

\bibitem{Kim2} Kim, B. \textquotedblleft On equivariant quantum cohomology".
IMRN, 17:841--851, 1996.

\bibitem{FominStanley} Fomin, S. and R. P. Stanley. "Schubert polynomials
and the nilCoxeter algebra." Advances in Mathematics 103, no. 2 (1994):
196-207.

\bibitem{FominKirillov} Fomin, S. and A. N. Kirillov. \textquotedblleft The
Yang-Baxter equation, symmetric functions, and Schubert polynomials."
Discrete Mathematics 153, no. 1 (1996): 123-143.

\bibitem{KnutsonTao} Knutson, A. and Terence Tao. \textquotedblleft Puzzles
and (equivariant) cohomology of Grassmannians." Duke Mathematical Journal
119, no. 2 (2003): 221-260.

\bibitem{KostantKumar} Kostant, B. and S. Kumar. \textquotedblleft The nil
Hecke ring and cohomology of G/P for a Kac-Moody group G." Proceedings of
the National Academy of Sciences 83, no. 6 (1986): 1543-1545.

\bibitem{VicOsc} Korff, C. \textquotedblleft Quantum cohomology via vicious
and osculating walkers." Accepted for publication in Lett Math Phys; arXiv
preprint arXiv:1204.4109 (2012).

\bibitem{KorffStroppel} Korff, C. and C. Stroppel. \textquotedblleft
The-WZNW fusion ring: A combinatorial construction and a realisation as
quotient of quantum cohomology." Advances in Mathematics 225, no. 1 (2010):
200-268.

\bibitem{Lakshmibaietal} Lakshmibai, V. and K. N. Raghavan, P. Sankaran,
\textquotedblleft Equivariant Giambelli and determinantal restriction
formulas for the Grassmannian", Pure and Applied Mathematics Quarterly, 2
(2006) 699-717

\bibitem{Laksov} Laksov, D. \textquotedblleft Schubert calculus and
equivariant cohomology of Grassmannians", Adv. Math. 217 (2008) 1869-1888

\bibitem{Lam} Lam, T. \textquotedblleft Affine Stanley symmetric functions",
American Journal of Mathematics 128 (2006) 1553--1586

\bibitem{Lam08} Lam, T. \textquotedblleft Schubert polynomials for the
affine Grassmannian." Journal of the American Mathematical Society 21, no. 1
(2008): 259-281.

\bibitem{LamShim} Lam, T. and M. Shimozono. \textquotedblleft Quantum
cohomology of G/P and homology of affine Grassmannian." Acta mathematica
204, no. 1 (2010): 49-90.

\bibitem{LSS} Lam, T., A. Schilling, M. Shimozono, \textquotedblleft
K-theory Schubert calculus of the affine Grassmannian", Compositio Math. 146
(2010), 811-852

\bibitem{Lascoux} Lascoux, Alain. \textquotedblleft Anneau de Grothendieck
de la vari\'{e}t\'{e} de drapeaux." In The Grothendieck Festschrift Volume
III, pp. 1-34. Birkh\"{a}user Boston, 2007.

\bibitem{Macdonald} Macdonald, I. G. \textquotedblleft Schur functions:
theme and variations", S\'{e}minaire Lotharingien de Combinatoire
(Saint-Nabor, 1992) 498 (1992): 5-39

\bibitem{MacdonaldBook} Macdonald, I. G. Symmetric functions and Hall
polynomials. Oxford University Press on Demand, 1998.

\bibitem{MaulikOkounkov} Maulik, D., and A. Okounkov. \textquotedblleft
Quantum Groups and Quantum Cohomology", Preprint (2012), 1--276. arXiv
preprint arXiv:1211.1287.

\bibitem{Mihalcea} Mihalcea, L. C. \textquotedblleft Giambelli formulae for
the equivariant quantum cohomology of the Grassmannian." Transactions of AMS
360 (2008), no. 5, 2285--2301; arXiv:math/0506335

\bibitem{Mihalcea06} Mihalcea, L. C. \textquotedblleft Positivity in
Equivariant Schubert Calculus", American J Math 128 (2006), no. 3, 787-803

\bibitem{MihalceaAIM} Mihalcea, L. C. \textquotedblleft Equivariant quantum
Schubert calculus", Advances in Mathematics 203 (2006), no. 1, 1 - 33

\bibitem{MolevSagan} Molev, A. and B. Sagan. \textquotedblleft A
Littlewood-Richardson rule for factorial Schur functions." Transactions of
the American Mathematical Society 351, no. 11 (1999): 4429-4443.

\bibitem{NekrasovShatashvili} Nekrasov, Nikita A., and Samson L.
Shatashvili. \textquotedblleft Supersymmetric Vacua and Bethe Ansatz."
Nuclear Physics B Proceedings Supplements 192 (2009): 91-112;
\textquotedblleft Quantum Integrability and Supersymmetric Vacua." Progress
of Theoretical Physics Supplement 177 (2009): 105-119.

\bibitem{Peterson} Peterson, D. \textquotedblleft Quantum cohomology of
G/P", Lecture notes, M.I.T., Spring 1997. (unpublished)

\bibitem{Postnikov} Postnikov, A. \textquotedblleft Affine approach to
quantum Schubert calculus", Duke Math. J. 128 (3) (2005) 473--509.

\bibitem{Rietsch} Rietsch, K. \textquotedblleft A mirror symmetric
construction of $qH_{T}^{\ast }(G/P)_{(q)}$." Advances in Mathematics 217,
no. 6 (2008): 2401-2442

\bibitem{Rimanyietal} Rimanyi, R., V. Schechtman, V. Tarasov, and A.
Varchenko. \textquotedblleft Cohomology of a flag variety as a Bethe
algebra." Functional Analysis and Its Applications 45, no. 4 (2011): 252-264.

\bibitem{Smirnov} Smirnov, F. A. \textquotedblleft Form factors in
completely integrable models of quantum field theory." Vol. 14. World
Scientific, 1992.

\bibitem{Tymoczko} Tymoczko, J. S. \textquotedblleft Divided Difference
Operators for Partial Flag Varieties"; arxiv:0912.2545

\bibitem{Vafa} Vafa, C. \textquotedblleft Topological mirrors and quantum
rings", in: Essays on Mirror Manifolds, Int. Press, Hong Kong, 1992, pp.
96--119.

\bibitem{LVW} Lerche, W., C. Vafa, and N. P. Warner. \textquotedblleft
Chiral rings in N = 2 superconformal theories." Nuclear Physics B 324, no. 2
(1989): 427-474.

\bibitem{Vasserot} Vasserot, E. \textquotedblleft Affine quantum groups and
equivariant K-theory." Transformation groups 3, no. 3 (1998): 269-299.

\bibitem{Witten} Witten, E. \textquotedblleft The Verlinde algebra and the
cohomology of the Grassmannian", in: Geometry, Topology, \& Physics, in:
Conf. Proc. Lecture Notes Geom. Topology, vol. IV, Int. Press, Cambridge,
MA, 1995, pp. 357--422.
\end{thebibliography}
\end{document}